\documentclass[a4paper,twoside,english]{article} %con fleqn nelle quadre serve per non centrare le equazioni in equation
\usepackage[dvips]{graphicx}
\usepackage{makeidx}
\usepackage{color}
\usepackage{amsfonts}
\usepackage{listings}          %listati
\usepackage{fancyhdr}
\usepackage{indentfirst}
\usepackage{array} 
\usepackage{newlfont}
\usepackage{microtype} % migliora il riempimento delle righe
\usepackage[english]{babel}
\usepackage[applemac]{inputenc}    %ATTENZIONE A QUESTO COMANDO MAC (di solito compare latin1
\usepackage[T1]{fontenc}
\usepackage{amscd}

\usepackage[fleqn,tbtags]{amsmath}
\usepackage{amssymb,mathrsfs}       %per le formule matematiche fuori dal testo 
\usepackage{dsfont}                            %\mathds{1}
\usepackage{mathtools}
\mathtoolsset{showonlyrefs}			%numera solo le equazioni richiamate con eqref 
\numberwithin{equation}{section}             %numera le equazioni con le sezioni

\usepackage{latexsym}
\usepackage{amsthm}
\usepackage{fancyhdr}

\theoremstyle{definition}                          %stile corsivo
\newtheorem{thm}{Theorem}[section]     %definizione ambiente teorema
\newtheorem{prop}[thm]{Proposition}      %definizione ambiente proposizione
\newtheorem{cor}[thm]{Corollary}            %definizione ambiente corollario
\newtheorem{lem}[thm]{Lemma}             %definizione ambiente lemma
\theoremstyle{definition}               %stile roman
\newtheorem{dfn}[thm]{Definition}%  definizione ambiente definizione
\newtheorem{ese}[thm]{Example}        %definizione ambiente esempio
\newtheorem{nota}[thm]{Notation}

\newtheorem{rem}[thm]{Remark}          

\newtheoremstyle{prova}% hnamei
{3pt}% Space above
{3pt}% Space below
{}% Body font
{}% Indent amounti
{\textbf}% Theorem head font
{.}% hPunctuation after theorem headi
{\newline}% hSpace after theorem headi2
{}% Theorem head spec (can be left empty, meaning `normal')
\theoremstyle{prova}

\def\R{\mathbb R}
\def\N{\mathbb N}
\def\Z{\mathbb Z}

\def\D{\mathbb D}

\def\P{\mathbb P}
\def\W{\mathbb W}
\def\X{\mathbb X}
\def\Y{\mathbb Y}
\def\Z{\mathbb Z}

\def\shd{{\cal D}}
\def\shf{{\cal F}}

\def\shm{{\cal M}}

\def\shs{{\cal S}}
\def\shu{{\cal U}}

\def\1{\mathds{1}}

\def\be{\begin{equation}}
\def\ee{\end{equation}}
\def\bs{\begin{split}}
\def\es{\end{split}}

\begin{document}

\title {Generalized covariation for Banach space valued processes, It\^o formula and applications}
%Part II: Stability results for infinite dimensional Dirichlet processes.}
%\author{Russo Francesco \\
%  Universit\'e Paris 13 \\
%  \and
%  Di Girolami Cristina \\
%  Universit\`a LUISS\\
%  Roma\\}
%
%\author{Cristina Di Girolami \and Francesco Russo }

\author{{\sc Cristina DI GIROLAMI}
   \thanks{
   Laboratoire Manceau de Math\'ematiques, Facult\'e des Sciences et Techniques, Universit\'e du Maine, D\'epartement de Math\'ematiques, Avenue Olivier Messiaen, 72085 Le Mans CEDEX 9 (France).
E-mail: {\tt Cristina.Di\_Girolami@univ-lemans.fr }}
  {\sc}\  
\ {\sc and}\ {\sc Francesco RUSSO} \thanks{ENSTA ParisTech,
Unit\'e de Math\'ematiques appliqu\'ees,
32, Boulevard Victor,
F-75739 Paris Cedex 15 (France)
E-mail: {\tt  francesco.russo@ensta-paristech.fr}.
}
% E-mail:{\tt  russo@math.univ-paris13.fr}}
}

%\date{}
\date{Actual version: January 14th 2013 (First one: December 23th 2010)  }
%\\
% First version: December 23th 2010}
%\date{}   senza niente mette la data, per ometterla \date{}, per scriverla noi \date{18 luglio}
\maketitle

\begin{abstract}
This paper discusses a new notion of quadratic variation and covariation for
 Banach space valued processes (not necessarily semimartingales) and 
related It\^o formula. 
If $\X$ and $\Y$ take respectively values in Banach spaces $B_{1}$ and $B_{2}$
 and $\chi$ is a 
suitable subspace of the dual of the projective tensor product of
 $B_{1}$ and $B_{2}$ (denoted by $(B_{1}\hat{\otimes}_{\pi}B_{2})^{\ast}$), 
we define the so-called $\chi$-covariation of $\X$ and $\Y$. If $\X=\Y$,
 the $\chi$-covariation is called $\chi$-quadratic variation. 
The notion of $\chi$-quadratic variation is a natural generalization of 
the one introduced by M\'etivier-Pellaumail and Dinculeanu which 
is too restrictive for many applications. In particular, if $\chi$ is the whole space $(B_{1}\hat{\otimes}_{\pi}B_{1})^{\ast}$ then the $\chi$-quadratic variation coincides with 
the quadratic variation of a $B_{1}$-valued semimartingale. 
We evaluate the $\chi$-covariation of various processes for several examples of $\chi$ with a particular attention to the case $B_{1}=B_{2}=C([-\tau,0])$ for some $\tau>0$ and 
$\X$ and $\Y$ being \textit{window processes}.
If $X$ is a real valued process, we call window process associated with $X$ 
the $C([-\tau,0])$-valued process $\X:=X(\cdot)$ defined by
$X_t(y) = X_{t+y}$, where $y \in [-\tau,0]$. The It\^o formula
 introduced here is an important instrument to establish 
a representation result of Clark-Ocone type for a class of path
 dependent random variables of type $h=H(X_{T}(\cdot))$, 
$H:C([-T,0])\longrightarrow\R$ 
for not-necessarily semimartingales $X$ with finite quadratic variation. 
This representation will be linked to a  function $u:[0,T]\times C([-T,0])\longrightarrow \mathbb{R}$ solving 
an infinite dimensional partial differential equation.

\end{abstract}

[\textbf{2010 Math Subject Classification}:  ] \   {60G05, 60G07, 60G22, 60H05, 60H99. }
\medskip

%[\textbf{JEL Classification Codes}:\ ] \  {G10, G11, G12, G13} 

\bigskip


%\newpage
\section{Introduction}
The present paper settles the basis for the calculus via regularization
 for processes with values in an infinite dimensional separable Banach 
 space $B$.
We introduce a new approach to face stochastic integration
for infinite dimensional processes, based on an original 
generalization of the  notion of quadratic covariation.
This allows to discuss stochastic calculus in a more general
framework than in the present literature.

The extension of It\^o stochastic integration theory for Hilbert
 valued processes 
dates only from the eighties, the results of which can be found in the monographs \cite{Me84G,MeSN,dpz}  and \cite{walsh} with different techniques.
However the discussion of this last approach is not the aim of this
paper. Extension to nuclear valued spaces is simpler and was done in
 \cite{KaMiW, Ust82}. 
One of the most natural but difficult situations arises when the processes are Banach space valued. \\
%At our knowledge the contribution is new even when $B$ is Hilbert.\\
As for the real case, a possible tool of infinite dimensional stochastic calculus is the concept of quadratic variation, 
or more generally of covariation.
The notion of covariation is historically defined for two real valued
 $(\mathcal{F}_{t})$-semimartingales $X$ and $Y$.
% and it is denoted by $[X,Y]$.
This notion was extended to the case of general processes by means of discretization techniques, for instance by \cite{fo}, or via regularization, 
 in \cite{rv2,Rus05}. 
In this paper we will follow the language of regularization; for simplicity we suppose that either $X$ or $Y$ is continuous. In the whole paper $T$ will be a fixed positive number. 
Every process will be indexed by $[0,T]$, but, if it is continuous, it
 can be extended to the real line for convenience by setting $X_t=X_0$ if $t<0$ and $X_t=X_T$ for $t\geq T$.
%We propose here a slight different approach that \cite{rv2}.
\begin{dfn}			\label{def cov}
Let $X$ and $Y$ be two real processes such that $X$ is continuous and $Y$ has almost surely locally integrable paths. For $\epsilon >0$, we denote 
\begin{eqnarray}
[X,Y]^{\epsilon}_{t}&=&\int_{0}^{t}\frac{ \left(X_{s+\epsilon}-X_{s} \right)\left(Y_{s+\epsilon}-Y_{s} \right)}{\epsilon}ds  \; , \quad t\in [0,T] \; ,\\
I^{-}(\epsilon,Y,dX)_{t}&=&\int_{0}^{t} Y_{s} \frac{X_{s+\epsilon}-X_{s}}{\epsilon}ds  \; , \quad t\in [0,T] \; .
\end{eqnarray}
\begin{enumerate}
\item
We say that $X$ and $Y$ admit a \textbf{covariation} if $\lim_{\epsilon\rightarrow 0} [X,Y]^{\epsilon}_{t}$ exists in probability for every $t \in [0,T]$ and 
 the limiting process admits a continuous version that will be denoted by $[X,Y]$.
If $[X,X]$ exists, we say that $X$ has a \textbf{quadratic variation} and it will 
 also be denoted by $[X]$. If $[X]=0$ we say that $X$ is a \textbf{zero quadratic variation process}.
\item 
The \textbf{forward integral} 
$\int_{0}^{t}Y_{s}d^{-}X_{s}$ is a continuous process $Z$, such that whenever it exists,\\ 
$\lim _{\epsilon\rightarrow 0} I^{-}(\epsilon,Y,dX)_{t}=Z_{t}$ in probability for every 
$t\in [0,T]$.
\item If $\int_{0}^{t}Y_{s}d^{-}X_{s}$ exists for any $0\leq t< T$; $\int_{0}^{T}Y_{s}d^{-}X_{s}$ will symbolize the 
\textbf{improper forward integral} defined by $\lim_{t\rightarrow T}\int_{0}^{t}Y_{s}d^{-}X_{s} $, whenever it exists in probability.
\end{enumerate}
\end{dfn}
\begin{rem} \label{R120}
\begin{enumerate}
\item
 Lemma 3.1 in \cite{rv4} allows to show that, whenever
 $[X,X]$ exists, then $[X,X]^\varepsilon$ also converges 
in the \textbf{uniform convergence in probability 
  (ucp)} sense, see \cite{rv2,Rus05}.
The basic results established there are still valid here, see the
 following items.
% \item
% If $X$ and $Y$ are $(\F_{t})$-semimartingales, then $[X,Y]$ 
% coincides with the classical covariation of their local
% martingale parts, see Proposition 9 and Corollaries 2, 3 in
%  \cite{Rus05}.
 \item 
If $X$ (resp. $A$) is a finite (resp. zero) quadratic variation
 process, then $[A,X]=0$, see Proposition 1 5) of \cite{Rus05}.
\item If $Y$ is a bounded variation (c\`adl\`ag) process, then
$\int_0^t Y d^-X, t \in [0,T],$ exists and equals $Y_t X_t - Y_0 X_0 
- \int_{]0,t]} X dY, t \in [0,T],$ where the latter is a
 pathwise Lebesgue-Stieltjes
integral.
 This is a consequence of items 4) and 7) of Proposition 1
in \cite{Rus05}. 
\end{enumerate}
\end{rem}
Let $(\Omega,\shf,\P)$ be a fixed probability space, equipped with
a given filtration $\mathbb{F}=(\mathcal{F}_{t})_{t\in [0,T]}$ fulfilling the usual conditions. 
\begin{rem}\label{R120Sem}
 If $X$ is an $(\mathcal{F}_{t})$-continuous semimartingale and $Y$ is
 $(\mathcal{F}_{t})$-progressively measurable and  c\`adl\`ag 
(resp. an $(\mathcal{F}_{t})$-semimartingale) 
$\int_{0}^{\cdot}Y_{s}d^{-}X_{s}$ (resp. $[X,Y]$) coincides with the classical 
It\^o integral $\int_{]0,\cdot]} YdX$, also denoted by
 $\int_{0}^{\cdot} YdX$, (resp. the classical covariation of
their local martingale parts).
\end{rem}
The class of real finite quadratic variation processes is much richer than the one of semimartingales. 
%$W$ will always denote a real Brownian motion.
Typical examples of such processes
are $(\mathcal{F}_t)$-Dirichlet processes. $D$ is called 
$(\mathcal{F}_{t})$-\textbf{Dirichlet process} if it admits a decomposition $D=M+A$ where $M$ is an
$(\mathcal{F}_{t})$-local martingale and $A$ is an
$(\mathcal{F}_{t})$-adapted
 zero quadratic
variation process.
% In that case it holds $[D]=[M]$. This class of processes generalizes the semimartingales since a locally bounded variation process has zero quadratic variation.
A  slight generalization of that notion is
the one of weak Dirichlet process, which was
introduced in \cite{er2}. 
%$X$ is called $(\mathcal{F}_{t})$-\textbf{weak Dirichlet} if it
%admits a decomposition $X=M+A$ where $M$ is an $(\mathcal{F}_{t})$
%local martingale and $A$ is a process such that $[A,N]=0$ for any
%continuous $(\mathcal{F}_{t})$-local martingale $N$. A process $A$ with that property is called a {\bf $(\mathcal{F}_t)$-martingale orthogonal process}.
%An $(\mathcal{F}_{t})$-weak Dirichlet process is not necessarily a finite quadratic variation process.
%%, but there are $(\mathcal{F}_{t})$-weak Dirichlet processes with finite quadratic variation that are not Dirichlet processes: 
%On the other hand if $A$ has finite quadratic variation then it holds $[X]=[M]+[A]$. 
Another interesting example is the 
bifractional Brownian motion 
$B^{H,K}$ with parameters $H\in ]0,1[$ and $K\in ]0,1]$ which has finite quadratic variation if and only if $HK\geq 1/2$, see \cite{rtudor}. 
Notice that if $K=1$, then $B^{H,1}$ is a fractional Brownian motion with Hurst parameter $H\in ]0,1[$. 
If $HK=1/2$ it holds $[B^{H,K}]_t=2^{1-K}t$; if $K\neq 1$ this process is not even Dirichlet with respect to its own filtration. %  and if $HK>1/2$ it holds $[B^{H,K}]=0$. 
%Other significant examples are the so-called weak $k$-order Brownian motions, for fixed $k\geq 1$, constructed by \cite{follWuYor}, which are in general not Gaussian.
%$X$ is a weak $k$-order Brownian motion if for every $0\leq t_{1}\leq \cdots\leq t_{k}< +\infty$, 
%$(X_{t_{1}},\cdots,X_{t_{k}})$ is distributed as $(W_{t_{1}},\cdots, W_{t_{k}})$. 
%If $k\geq 4$ then $[X]_{t}=t$. 
%Among Gaussian processes, the processes admitting a \textit{covariance measure structure} have also finite (deterministic) quadratic variation, 
%see \cite{rkt}. \\
One object of this paper consists in investigating a possible useful generalization of the notions of covariation and quadratic variation for 
Banach space valued processes.
Particular emphasis  will be devoted to \textit{window processes} 
with values in the non-reflexive Banach space of real continuous functions
defined on $[-\tau, 0]$, $0< \tau\leq T$. To a real continuous process $X=(X_{t})_{t\in [0,T]}$, one can link  a natural infinite dimensional valued process defined as follows.
\begin{dfn}		\label{dfn WPRO}
Let $0<\tau\leq T$. We call \textbf{window process} associated with $X$, denoted by $X(\cdot)$, 
the $C([-\tau,0])$-valued process 
\begin{equation*}
X(\cdot)=\big(X_{t}(\cdot)\big)_{t\in [0,T]}=\{X_{t}(u):=X_{t+u}; u\in [-\tau,0], t\in [0,T]\} \; .
\end{equation*}
In the present paper, $W$ will always denote a real standard Brownian motion.
%If $X$ is a classical Brownian motion 
The window process $W(\cdot)$ associated with $W$ will be called
 \textbf{window Brownian motion}. 
\end{dfn}
Window  processes, taking values in the non-reflexive space 
$B = C([-\tau,0])$, are, in our opinion, an interesting
object which deserves more attention by stochastic analysis
experts. We enumerate some reasons. 
\begin{enumerate}
\item They naturally appear in functional dependent stochastic
 differential equations as delay equations.
\item  Let $W$ be a classical Wiener process.
% real finite quadratic process such that $[X,X]_t \equiv t$.
Consider $h= \phi(W_T)$ for some Borel non-negative 
$\phi:\R \rightarrow \R$ and let $\shu:[0,T] \times \R \rightarrow \R$
be a solution of $\partial_t \shu_t + \frac{1}{2} \partial^2_{xx} \shu  = 0$ 
with final condition $\shu(T,x) = \phi(x)$.
  By  It\^o formula one can show that
that 
\begin{equation} \label{EEE10}
 h =  h_0 + \int_0^t \xi_s dW_s,
\end{equation}
 where $\xi_s \equiv \partial_x \shu(s,W_s)$ and $h_0= \shu(0,X_0).$ 
A path dependent  random variable     $h$  can be represented as a
 functional of the corresponding window process, i.e. $h = f(\W)$
where $\W = W(\cdot), f: C([-T,0]) \rightarrow \R$. 
If $u$ is a smooth solution of a suitable  partial differential equation,
with space variable in $C([-T,0])$
%h =  h_0 + \int_0^t \xi_s dW_s
using an   $C([-T,0])$-valued It\^o formula,
we expect to be able to express   $h$ as \eqref{EEE10}
 where $h_0$ and $\xi$ depend on $u$.
Those considerations will extend to the case of a finite quadratic variation
 (even non-semimartingale) $X$. 
\item Even if the underlying process $X$ is a semimartingale,
its associated window $\X = X(\cdot)$ is not, in any reasonable sense.
Indeed if $\mu$ is a signed Borel measure on $[-\tau,0]$, i.e.
an element of $B^\ast$, the real valued process $X^\mu$ defined by
$ X^\mu_t =  \langle \mu, \X \rangle_t =
 \int_{[-\tau,0]} \mu(dx) X_{t+x}$ 
is in general not a real semimartingale, as Proposition \ref{prop WBMNS} illustrates.
In fact even if $X$ is a standard Wiener process, $ X^\mu$  is not a semimartingale. For instance if $\mu$ is the sum of Dirac measures $\mu = \delta_0 + \delta_{-\tau}$.
On the other hand if $X$ is a continuous semimartingale vanishing
at zero and $\mu (dx)  = \delta_0(dx) + g(x) dx $ where $g$ is a bounded Borel function then $ X^\mu$ is a semimartingale,
see Remark \ref{RWBMNS}, item 2.
\end{enumerate}

%We emphasize that $C([-\tau,0])$ is typical a non-reflexive Banach space.
We will introduce a notion of covariation for processes with values in general Banach spaces but which will be performing also for window processes.
% As motivating application we refer to a generalized Clark-Ocone formula for processes which have finite quadratic variation processes developed in \cite{DGRnote}. 
This paper settles the theoretical basis for the stochastic calculus part related to the first part of \cite{DGRnote} and which partially appears in \cite{DGR}. 
%Here we will expand a new application to the case when the process have zero quadratic variation which was not included in \cite{DGRnote}.\\
Let $B_{1}$, $B_{2}$ be two general Banach spaces. 
In this paper $\mathbb{X}$ (resp. $\mathbb{Y}$) will be a $B_{1}$ (resp. $B_{2}$) valued stochastic process. 
It is not obvious to define an exploitable 
notion of covariation (resp. quadratic variation) of $\mathbb{X}$ and
 $\mathbb{Y}$ (resp. of $\X$). 
When $\mathbb{X}$ is an $H$-valued martingale and $B_{1}=B_{2}=H$ is a
 separable Hilbert space,
\cite{dpz}, Chapter 3 introduces an operational notion
of quadratic variation.
% for martingales with values in $H$. %This is implemented in the theory of stochastic partial differential equations. 
%\cite{mp} considers the so-called $\pi$-processes with values in a general Banach space. 
%The family of $\pi$-processes are not so far from semimartingales, since their notion is constantly related
%to It\^o type stochastic integrals.  
\cite{dincuvisi} introduces in Definitions A.1 in Chapter 2.15 and B.9 in Chapter 6.23 the notions of  
\textit{semilocally summable} and \textit{locally summable processes with respect to a given bilinear mapping on $B\times B$}; 
see also Definition C.8 in Chapter 2.9 for the definition of \textit{summable} process. Similar notions appears in \cite{mp}. 
Those processes are very close to Banach space valued semimartingales. If $B$ is a Hilbert space, a semimartingale is semilocally summable when the bilinear form is the inner product.
For previous processes, \cite{dincuvisi} defines two natural notions of quadratic variation: 
the real quadratic variation and the tensor quadratic variation. 
For avoiding confusion with the quadratic variation of real processes,
 we will use  the terminology \emph{scalar} instead of \emph{real}.
Even though \cite{mp,dincuvisi} make use of discretizations, we define here, for commodity, two very similar objects but in our regularization language, see Definition \ref{dfn REALTENSOQV}. 
Moreover, the notion below extends to the covariation of two processes $\X$ and $\Y$ for which we remove the assumption of 
\textit{semilocally summable} or \textit{locally summable}. Before that, we remind some properties related to tensor products of two Banach spaces $E$ and $F$, see \cite{rr} for details.
%, the material when $E$ and $F$ are Hilbert spaces being particularly exhaustive in \cite{nevpg}. 
If $E$ and $F$ are Banach spaces, $E\hat{\otimes }_{\pi}F$ (resp. $E\hat{\otimes }_{h}F$) is a Banach space which denotes 
 the \textbf{projective} (resp. \textbf{Hilbert}) \textbf{tensor product of $E$ and $F$}. 
 We recall that $E\hat{\otimes }_{\pi}F$ (resp. $E\hat{\otimes }_{h}F$) is obtained by a completion of the algebraic tensor product $E\otimes F$ equipped with the projective 
norm $\pi$ (resp. Hilbert norm $h$). For a general element $u=\sum_{i=1}^{n}e_{i}\otimes f_{i} $ in $E\otimes F$, $e_{i}\in E$ and $f_{i}\in F$, it holds
$\pi(u)=\inf \left\{  \sum_{i=1}^{n}\|e_{i}\|_{E} \, \| f_{i} \|_{F}:\, u=\sum_{i=1}^{n}e_{i}\otimes f_{i}, e_{i}\in E\, , f_{i}\in F \right\}$. 
For the definition of the Hilbert tensor norm $h$ the reader may refer \cite{rr}, Chapter 7.4.
%(resp. $h^{2}(u)=\sum_{i=1}^{n}\| e_{i}\|_{E}^{2}\|f_{i}\|_{F}^{2}$).
We remind that if $E$ and $F$ are Hilbert spaces the Hilbert tensor product $E\hat{\otimes }_{h}F$ is also Hilbert and its inner product 
between $e_{1}\otimes f_{1}$ and $e_{2}\otimes f_{2}$ equals $\langle e_{1},e_{2}\rangle_{E}\cdot \langle f_{1},f_{2}\rangle_{F}$. 
Let $e\in E$ and $f\in F$, 
the symbol $e\otimes f$ (resp. $e\otimes ^{2}$) will denote an elementary element of the algebraic tensor product $E\otimes F$ (resp. $E\otimes E$).
The Banach space $(E\hat{\otimes}_{\pi}F)^{\ast}$ denotes the topological dual of the projective tensor product 
equipped with the operator norm. As announced we give now the two definitions of scalar and tensor covariation and quadratic variation.
\begin{dfn} 		\label{dfn REALTENSOQV}
Let $\X$ (resp. $\Y$) be a $B_{1}$ (resp. $B_{2}$) valued stochastic process.
\begin{enumerate}
\item $(\X,\Y)$ is said to admit a \textbf{scalar covariation}  if the 
limit for $\epsilon\downarrow 0$ of the sequence 
\begin{equation}		\label{eq DFNREAL}
[\X,\Y]_{\cdot}^{\mathbb{R},\epsilon}=\int_{0}^{\cdot}\frac{\| \X_{s+\epsilon}-\X_{s}\|_{B_{1}} \| \Y_{s+\epsilon}-\Y_{s}\|_{B_{2}}}{\epsilon} ds   \; 
\end{equation} 
exists ucp. That limit will be indeed called 
 scalar covariation of $\X$ and $\Y$ and it will be simply denoted by $[\X,\Y]^{\R}$. 
The scalar covariation $[\X,\X]^{\R}$ will be called
 \textbf{scalar quadratic variation} of $\X$ and simply denoted  by $[\X]^{\R}$.
\item $(\X,\Y)$ admits a \textbf{tensor covariation} if there exists a $(B_{1}\hat{\otimes}_{\pi}B_{2})$-valued process denoted by $[\X,\Y]^{\otimes}$ such that the sequence of Bochner 
$(B_{1}\hat{\otimes}_{\pi}B_{2})$-valued integrals
\begin{equation}		\label{eq DFNTENS}
[\X,\Y]_{\cdot}^{\otimes,\epsilon}=\int_{0}^{\cdot}\frac{(\X_{s+\epsilon}-\X_{s})\otimes (\Y_{s+\epsilon}-\Y_{s})}{\epsilon} ds
\end{equation} 
converges  ucp for $\epsilon\downarrow 0$ (according to the strong topology) to a $(B_{1}\hat{\otimes}_{\pi}B_{2})$-valued process $[\X,\Y]^{\otimes}$.\\
$[\X,\Y]^{\otimes}$ will  indeed be called tensor covariation of $(\X,\Y)$. The tensor covariation $[\X,\X]^{\otimes}$ 
will be called \textbf{tensor quadratic variation} and simply denoted by $[\X]^{\otimes}$.
\end{enumerate}
\end{dfn}
\begin{rem} 
\begin{enumerate}				\label{rem RE1}
\item By use of Lemma 3.1 in \cite{rv4}, if
 $[\X,\Y]_{\cdot}^{\mathbb{R},\epsilon}$ converges, for any $t \in [0,T],$ 
to $Z_{t}$, where $Z$ is a continuous 
process, then the scalar covariation of $(\X,\Y)$ exists and $[\X,\Y]^{\R}=Z$.
\item If $(\X,\Y)$ admits both a scalar 
and  tensor covariation, then 
the tensor covariation process has bounded variation and its
 total variation is bounded by the 
scalar covariation which is clearly an increasing process.
\item
If $(\X,\Y)$ admits a tensor covariation, then we have in particular 
\begin{equation}		\label{eqLLL}
 \frac{1}{\epsilon} \int_{0}^{\cdot} \langle \phi, (\X_{s+\epsilon}-\X_{s})\otimes (\Y_{s+\epsilon}-\Y_{s}) \rangle \, ds \xrightarrow[\epsilon\longrightarrow 0] {ucp}\langle \phi,  [\X,\Y]^{\otimes}\rangle,
\end{equation}
for every $\phi\in (B_{1}\hat{\otimes}_{\pi}B_{2})^{\ast}$, $\langle \cdot,\cdot \rangle $ denoting the duality between $B_{1}\hat{\otimes}_{\pi}B_{2}$ and its dual. 
\item If 
%$\X$ and $\Y$ are such that 
$[\X,\Y]^{\R}=0$, then $(\X,\Y)$
 admits a tensor covariation which also vanishes.
\end{enumerate}
\end{rem}
\begin{prop}		\label{prop P1}
Let $\X$ be an $(\mathcal{F}_{t})$-adapted semilocally summable process with respect to the bilinear maps
(tensor product)  $B\times B \longrightarrow B\hat{\otimes}_{\pi} B $, given by  
$(a,b)\mapsto a\otimes b$ and $(a,b)\mapsto b\otimes a$. Then $\X$ admits a tensor quadratic variation.
\end{prop}
\begin{prop}		\label{prop P2}
Let $\X$ be a Hilbert space valued continuous $(\mathcal{F}_{t})$-semimartingale in the sense of \cite{mp}, section 10.8. 
Then $\X$ admits a scalar quadratic variation.
\end{prop}
A sketch of the proof of the two propositions above are given in the Appendix.
A consequence of Proposition \ref{prop P1} and item 2 of Remark \ref{rem RE1}
is the following.
\begin{cor} \label{C18} 
Let $\X$ be a Banach space valued process which is semilocally summable
 with respect to the tenso product. 
If $\X$ has a scalar quadratic variation, it admits a tensor
 quadratic variation process which has bounded variation. 
\end{cor}
 
\begin{rem} \label{RCovDPZ}
 The tensor quadratic variation can be linked to
 the one of \cite{dpz}; see Chapter 6 in \cite{DGR} for details.
Let $H$ be a separable Hilbert space. If $\mathbb{V}$ is an $H$-valued 
$Q$-Brownian motion with $Tr(Q)<+\infty$ (see \cite{dpz} section 4),
 then $\mathbb{V}$ admits a scalar quadratic variation
 $[\mathbb{V}]_{t}^{\R}=t\,Tr(Q)$ and a tensor 
quadratic variation $[\mathbb{V}]^{\otimes}_{t}=tq$ where $q$ is
 the tensor associated to the nuclear operator $tQ$.
\end{rem}
%
%On the other hand in Proposition    \ref{prop WBMNS},
 We have already observed 
that $W(\cdot)$ is not a $C([-\tau,0])$-valued semimartingale.
Unfortunately,  the window process $W(\cdot)$ associated with 
a real Brownian motion $W$, 
does not even admit  a scalar quadratic variation. 
In fact the limit of 
\begin{equation}		\label{eqE45E}
\int_{0}^{t}\frac{\|W_{s+\epsilon}(\cdot)-W_{s}
(\cdot)\|^{2}_{C([-\tau,0])}}{\epsilon}\,ds	\; ,		\quad t\in [0,T] \; ,
\end{equation}
for $\epsilon$ going to zero does not converge, as we will see in Proposition \ref{pr WBMNOQV}. 
This suggests that when $\X$ is a window process, the tensor 
quadratic variation is not the suitable object in order to perform 
stochastic calculus.
Let $\X$ (resp. $\Y$) be a $B_{1}$ (resp. $B_{2}$)-valued process. In
 Definition \ref{def CHICOV} we will introduce a notion of covariation of
 $(\X,\Y)$ (resp. quadratic variation of $\X$ when $\X=\Y$) 
which generalizes the tensor covariation (resp. tensor quadratic variation). This will be called 
$\chi$-covariation (resp. $\chi$-quadratic variation) in reference to a topological subspace $\chi$ of  the dual 
of $B_{1} \hat{\otimes}_{\pi} B_{2}$ (resp. $B_{1} \hat{\otimes}_{\pi} B_{2}$ with $B_{1}=B_{2}$). 
%According to our strategy, 
We will suppose in particular that
\begin{equation}		\label{eq DFNTENS1}
 \frac{1}{\epsilon} \int_{0}^{t} \langle \phi, (\X_{s+\epsilon}-\X_{s})\otimes (\Y_{s+\epsilon}-\Y_{s}) \rangle\, ds 
\end{equation}
converges for every $\phi\in \chi$ for every $t \in [0,T]$. 
If $\Omega$ were a singleton (the processes being deterministic) and $\chi$
 would coincides with the whole space 
 $(B_{1}\hat{\otimes}_{\pi}B_{2})^{\ast}$ then previous convergence is the one related to the weak star topology in $(B_{1}\hat{\otimes}_{\pi}B_{2})^{\ast\ast}$.\\ 
 %When $\chi$ is separable and it coincides with the whole space  $(B_{1}\hat{\otimes}_{\pi}B_{2})^{\ast}$, this convergence is strongly related to the weak star topology in $(B_{1}\hat{\otimes}_{\pi}B_{2})^{\ast\ast}$.
Our $\chi$-covariation generalizes the concept of tensor
covariation at two levels. 

$\bullet$ First we replace the (strong) convergence of 
\eqref{eq DFNTENS} with a weak star type topology convergence of \eqref{eq DFNTENS1}.

$\bullet$ Secondly the choice of a suitable subspace $\chi$ of 
 $(B_{1}\hat{\otimes}_{\pi}B_{2})^{\ast}$ gives a degree of freedom.
 For instance, 
 compatibly with \eqref{eqE45E}, a window Brownian motion $\X = W(\cdot)$ admits a $\chi$-
quadratic variation only for strict
subspaces $\chi$. 

When $\chi$ equals the whole space $(B_{1} \hat{\otimes}_{\pi} B_{2})^{\ast}$ (resp. $(B_1\hat{\otimes}_{\pi}B_1)^{\ast}$) this will be called 
\textbf{global covariation} (resp. \textbf{global quadratic variation}). This situation corresponds for us to the \textit{elementary situation}.\\
Let $B_{1}=B_{2}$ be the finite dimensional space $\R^{n}$ and $\X=(X^1,\ldots,X^n)$ and $\Y=(Y^1,\ldots, Y^n)$ with values in $\R^n$, 
Corollary \ref{cor finD} says that
$(\X, \Y)$ admits all its mutual brackets (i.e. $[X^i,Y^j]$ exists for all $1\leq i,j\leq n$) if and only if $\X$ and $\Y$ have a global covariation. 
It is well-known that, 
in that case, $(B_{1} \hat{\otimes}_{\pi} B_{2})^{\ast}$ can be identified with the space of matrix $\mathbb{M}_{n\times n}(\R)$. 
If $\chi$ is finite dimensional, then Proposition \ref{pr FinDim} gives a simple characterization for $\X$ to have a $\chi$-quadratic variation.\\
Propositions \ref{prop P1}, \ref{prop P2}, \ref{pr TGH78}
and Remark \ref{RCovDPZ}
will imply that
whenever $\X$ admits one of the classical quadratic variations
(in the sense of \cite{dpz, mp, dincuvisi}), 
it admits a global quadratic variation and they are essentially equal.
% see  Chapter 6 of \cite{DGR} for more details.
In 
 this paper we calculate the $\chi$-covariation of Banach space valued processes in various situations with a particular attention for window processes associated 
to real finite quadratic variation processes, for instance semimartingales, Dirichlet processes, bifractional Brownian motion.
%In the companion paper \cite{DGR2} we investigate properties of stability of various types of processes admitting a $\chi$-covariation through $C^{1}$ transformation. 
%Similarly to what mentioned earlier, the $\chi$-covariation naturally intervenes in It\^o type formulae, see for instance \cite{DGRnote}.\\

The notion of covariation intervenes in Banach space valued stochastic calculus
 for semimartingales, 
especially via It\^o type formula, see for \cite{dincuvisi} and \cite{mp}.
An important result of this paper is an It\^o formula for Banach space valued processes admitting a $\chi$-quadratic variation, see Theorem \ref{thm ITONOM}. 
This generalizes the following formula, valid for real valued processes which is stated below, 
see \cite{rv2}. 
Let $X$ be a real finite quadratic variation process and $f\in C^{1,2}([0,T]\times \R)$. Then the forward integral
 $\int_{0}^{\cdot}\partial_{x}f(s,X_{s})d^{-}X_{s}$ exists and 
\begin{equation} \label{FITO}
f(t,X_{t})=f(0,X_{0})+\int_{0}^{t}\partial_{s}f(s,X_{s})ds
+\int_{0}^{t} \partial_{x}f(s,X_{s})d^{-}X_{s}+\frac{1}{2}\int_{0}^{t}\partial_{xx}^{2}f(s,X_{x})d[X]_{s}  \quad t\in [0,T]. 
\end{equation}
\cite{fo} gives a similar formula in the discretization approach instead regularization.\\
% to pathwise stochastic integration.\\
For that purpose, let $\Y$ (resp. $\X$) be a $B^{\ast}$-valued strongly measurable with a.s. bounded paths 
(resp. $B$-valued continuous) process, $B$ denoting a separable Banach space; 
we define a real valued forward-type integral $\int_{0}^{t} \prescript{}{B^{\ast} }{\langle}Y, d^{-}X\rangle_{B}$, 
see Definition \ref{def integ fwd}.
We emphasize that Theorem \ref{thm ITONOM} constitutes a generalization
 of the It\^o formula in \cite{mp}, section 3.7, 
(see also \cite{dincuvisi}) for two reasons.
First, taking  $\chi=(B\hat{\otimes}_{\pi}B)^{\ast}$, i.e. the
full space,  the integrator 
processes $\X$ that we consider are more general than those in
the class considered in \cite{mp} or \cite{dincuvisi}.
The second, more important reason, is the use of a  space $\chi$ which gives a supplementary degree of freedom.\\
In the final Section \ref{Sec6}, we introduce two applications of our infinite dimensional stochastic calculus. 
 That section concentrates on window processes, which first motivated our
 general construction. 
In Section \ref{6.1} we discuss an application of the It\^o formula to anticipating calculus in a framework for which Malliavin calculus cannot be used
necessarily. 
In Section \ref{SIDPDE}, we discuss the application to
a representation result of \emph{Clark-Ocone type} for not necessarily semimartingales with finite quadratic variation, including zero quadratic variation.
%In Section \ref{sec: evaluation}, we will give examples of evaluation of $\chi$-covariation variation for different classes of processes. In Section \ref{sec: evaluation}, we will give examples of evaluation of $\chi$-covariation variation for different classes of processes. 
Let $X$ be a continuous stochastic process with  quadratic variation $[X]_t=\sigma^2 t, \sigma \ge 0$. 
Our It\^o formula is one basic ingredient to prove a Clark-Ocone type result for path dependent real random variables
 of the type $h:=H(X_{T}(\cdot))$ with $H:C([-T,0])\longrightarrow \R$. 
We are interested in natural sufficient conditions to decompose $h$ into the sum of a real number $H_{0}$ and a forward integral $\int_0^T \xi_t d^-X_t$. % where $H_{0}$ and $\xi$ are  explicitly given.  
Suppose that $u\in C^{1,2}\left([0,T[\times C([-T,0]) \right)$ is a solution of an infinite dimensional partial differential equation (PDE) of the type
\begin{equation}		\label{eq SYST} 
\begin{dcases}
\partial_{t}u(t,\eta)+\prescript{ `` }{ }{\int_{]-t,0]}} D^{\perp}u\,(t,\eta)\; d\eta^{''} +\frac{\sigma^2 }{2} \langle D^{2}u\,(t,\eta)\; ,\; \1_{D_t}\rangle=0  & \\
u(T,\eta)  =H(\eta),   &
\end{dcases}
\end{equation}
where 
$
\1_{D_t} (x,y):=\left\{ 
\begin{array}{ll}
1  & \textrm{if } x=y, \; x,y\in [-t,0]\\
0 & \textrm{otherwise}
\end{array}
\right.
$ and $D^{\perp}u (t,\eta):=Du\, (t,\eta)- Du (t,\eta)(\{0 \})\delta_0$; in fact
$Du(t,\eta)$ (resp. $D^{2}u(t,\eta)$) denotes the first (resp. second) order Fr\'echet derivatives of $u$ with
respect to $\eta$. 
A proper notion of solution for \eqref{eq SYST} will be given in
 Definition \ref{DefSol}. 
Of course, the integral ``$\int_{]-t,0]}D^{\perp}u\,(t,\eta)\; d\eta$'' has to be suitably defined.
At this stage we only say that supposing, for each $(t,\eta)$,  $D^{\perp} u\, (t,\eta)$ 
 absolutely continuous with respect to Lebesgue measure and 
that its Radon–Nikodym derivative has bounded variation,
 then $\int_{]-t,0]}D^{\perp}u\,(t,\eta)\; d\eta$ is well-defined by an integration by parts, see Notation \ref{not IBPCONTFUNC}. 
The term $\langle D^{2}u\,(t,\eta)\; ,\; \1_{D_t}\rangle$ indicates the evaluation of the second order derivative on the increasing 
diagonal of the square $[-t,0]^{2}$, provided that $D^{2}u(t,\eta)$ is
 a Borel signed  measure on $[-T,0]^2$. % (which is the case in all the examples developed, see \cite{DGR3}). 
Our It\^o formula, i.e. Theorem \ref{thm ITONOM}, allows in fact to get the mentioned representation above %, which is of Clark-Ocone type 
with $H_{0}=u(0,X_{0}(\cdot))$, $\xi_{t}=D^{\delta_{0}}u(t, X_{t}(\cdot)):=Du(t, X_{t}(\cdot))(\{0\})$. 
In Chapter 9 of \cite{DGR} we construct explicitly solutions of the infinite dimensional PDE 
\eqref{eq SYST} when $H$ has some smooth regularity in $L^2([-\tau,0])$ or when it depends (even non smoothly) on a finite number of Wiener integrals. \\
A third application of Theorem  \ref{thm ITONOM} appears in \cite{RusFab}. In particular, those two authors calculate and 
use the $\chi$-quadratic variation of a mild solution of a 
stochastic PDE which generally is not a finite quadratic variation process in the sense of \cite{dpz}.

The paper is organized as follows.  
Section \ref{sec: pre} contains general notations and some preliminary results. 
Section \ref{sec: chiqv} will be devoted to the definition of $\chi$-covariation and $\chi$-quadratic variation and some related propositions.
Section \ref{sec: evaluation} provides some explicit calculations   related  to window processes. 
%to the notions introduced in previous section,
%when the process is a window process. 
Section \ref{sec:ito-int} is devoted to the definition of a forward integral for Banach space valued processes and related It\^o formula.
The final section \ref{Sec6} is devoted to applications
of our  It\^o formula to the case of window processes.
%%%%%%%%%%%%%%%%%%%%%%%%%%%%%%%%%%%%%%%%%%%%%%%%%%%%%%%%%%%%%%%%%%%%%
%

\section{Preliminaries}		\label{sec: pre}
Throughout this paper we will denote by 
$(\Omega,\shf,\P)$ a fixed probability space, equipped with
a given filtration $\mathbb{F}=(\mathcal{F}_{t})_{t\geq 0}$ 
%$(\shf_{t}; 0\leq t\leq T)$ 
fulfilling the usual conditions. 
Let $K$ be a compact space; $C(K)$
denotes the linear space of real continuous functions defined on $K$,
equipped with the uniform norm denoted by $\|\cdot\|_{\infty}$. 
$\mathcal{M}(K)$ will denote the dual 
space $C(K)^{\ast}$, i.e. the set of finite signed Borel measures on $K$. 
  In particular, if   $a < b$ are two real numbers, $C([a,b])$ will denote the
 Banach linear  space of real continuous functions. 
%equipped with the uniform norm denoted by $\|\cdot\|_{\infty}$. 
 If $E$ is a topological space, $\mathcal{B}or(E)$ will denote its Borel $\sigma$-algebra.
 %The space of bounded linear mappings from $B$ to $E$ will be denoted by $L(B;E)$ and we will write 
%$L(B)$ instead of $L(B;B)$. 
The topological dual (resp. bidual) space of $B$ %, i.e. when $L(B; \R)$, 
will be denoted by $B^{\ast}$ (resp. $B^{**}$). If $\phi$ is a linear continuous functional on $B$, 
we shall denote the value of $\phi$ of an element $b\in B$ either by $\phi(b)$ or $\langle \phi, b \rangle$ or even 
$\prescript{}{B^{\ast}}{\langle} \phi, b\rangle_{B}$.
 Throughout the paper the symbols $\langle\cdot,\cdot\rangle$ will always
denote some type of duality that will change depending on the context. 
 %Let $K$ be a compact space, $\mathcal{M}(K)$ will denote the dual space $C(K)^{\ast}$, i.e. the set of finite signed Borel measures on $K$. 
%Let $B$ be a normed space;
% a sequence of $B^{\ast}$-valued elements $(\phi^{n})_{n\in \mathbb{N}}$ 
%converges \textbf{weak star} to $\phi\in B^{\ast}$, i.e. in symbols
% $\phi^{n}\xrightarrow [n\longrightarrow +\infty]{w^{\ast}} \phi$, if $\prescript{}{B^{\ast}}{\langle} \phi^{n}, b\rangle_{B} \xrightarrow [n\longrightarrow +\infty]{} 
%\prescript{}{B^{\ast}}{\langle} \phi, b\rangle_{B}$ 
%for every $b\in B$. 
%%By definition, the weak star topology is weaker than the weak topology on $B^{\ast}$. 
%Given a Banach space $B$ and its topological bidual space $B^{\ast\ast}$ the application 
%$J: B\rightarrow B^{\ast\ast}$ will denote the natural injection between a Banach space and its bidual. 
%$J$ is an injective linear mapping, though it is not surjective unless $B$ is reflexive. 
%$J$ is an isometry with respect to the topology defined by the norm in $B$, the so-called \emph{strong topology}, and $J(B)$ which 
%is weak star dense in $B^{\ast\ast}$. 
%The weak star topology on $B^{\ast}$ is the weak topology induced by the image 
% $J(B)\subset B^{\ast\ast}$ of $J$.
%%If a normed space $B$ is separable, then the weak star topology is metrizable on (norm-)bounded subsets of $B^{\ast}$.
%By definition, the weak star topology is weaker than the weak topology on $B^{\ast}$. 
%$T  $ will always be a positive fixed real number. 
Let $E,F,G$ be Banach spaces. $L(E;F)$ stands for the Banach space of linear bounded maps from $E$ to $F$. We shall denote the space of $\R$-valued bounded bilinear forms on the product $E\times F$ by $\mathcal{B}(E,F)$
 with the norm given by $\|\phi\|_{\mathcal{B}}=\sup\{ | \phi(e,f)  |\,:\, \|e\|_{E}\leq 1; \|f\|_{F} \leq 1 \}$. Our principal references about functional analysis and about Banach spaces topologies are \cite{ds, bre}.\\
$T  $ will always be a positive fixed real number.
The capital letters $\X,\Y,\Z$ (resp. $X,Y,Z$) will generally denote
 Banach space   (resp. real) valued processes indexed by the time variable
 $t\in [0,T]$.
 %with $T>0$  (or $t\in \mathbb{R_+}$). 
A stochastic process $\X$ will  also be denoted by $(\X_{t})_{t\in[0,T]}$.
% or $(\X_t)_{t\ge 0}$.
A $B$-valued (resp. $\R$-valued) stochastic process
% $\X$ (resp. $X$)
 $\X:\Omega\times [0,T]\rightarrow B$ (resp. $\X:\Omega\times [0,T]\rightarrow \R$) 
is said to be measurable if $\X:\Omega\times [0,T]\longrightarrow B$ (resp. $\X:\Omega\times [0,T]\rightarrow \R$) is measurable with respect to the $\sigma$-algebras $\mathcal{F}\otimes \mathcal{B}or([0,T])$ and $\mathcal{B}or(B)$ (resp. $\mathcal{B}or(\R)$). 
We recall that 
$\mathbb{X}: \Omega\times [0,T]\longrightarrow B$ (resp. $\R$) is said to be \textbf{strongly measurable} (or \textbf{measurable in the Bochner sense}) if it is the limit of measurable countable valued functions. If $\X$ is measurable,
  c\`adl\`ag and $B$ is separable then $\X$ is strongly measurable. 
%because $B$ is separable. 
If $B$ is finite dimensional then a measurable process $\X$ is also strongly measurable. 
%$\X$ will always be supposed strongly measurable in the Bochner sense (resp. measurable w.r.t. the product sigma-algebra).
All the processes indexed by  $[0,T]$ 
%(respectively $\mathbb{R}^{+}$)
will be naturally prolonged by continuity setting $\X_{t}=\X_{0}$ for
$t\leq 0$ and $\X_{t}=\X_{T}$ for $t\geq T$.
% (respectively $\X_{t}=\X_{0}$ for $t\leq 0$).
 A sequence $(\X^{n})_{n \in \N}$  of continuous $B$-valued processes
indexed by $[0,T]$, will be said to converge \textbf{ucp}
(\textbf{uniformly convergence in probability}) to a process $\X$ if
$\sup_{0\leq t\leq T}\|\X^{n}-\X\|_{B}$ converges to zero in probability
when $n \rightarrow \infty$. The space $\mathscr{C}([0,T])$ 
will denote the linear space of continuous real processes; it is a Fr\'echet space 
(or $F$-space shortly) if 
equipped with the metric 
$d(X,Y)=\mathbb{E}\left[ \sup_{t\in [0,T]}|X_{t}-Y_{t}|\wedge 1\right]$ which governs the ucp topology, 
see Definition II.1.10 in \cite{ds}. For more details about $F$-spaces and their properties see section II.1 in \cite{ds}.\\
%We recall Lemma 3.1 from \cite{rv4} which constitutes a stochastic version of Dini's lemma. 
%%The mentioned lemma states that a sequence of continuous increasing
%%processes converging at each time in probability to a continuous process, converges ucp. 
%%
%\begin{lem}    \label{lem CPUCP}		%lemma Russo Vallois
%Let $(Z^{\epsilon})_{\epsilon>0}$ be a family of continuous real processes. We suppose the following.
%
%1) $\forall \epsilon> 0$, $t\longrightarrow Z^{\epsilon}_{t}$ is increasing.
%
%2) There is a continuous process $(Z_{t})_{t\in [0,T]}$ such that $Z_{t}^{\epsilon}\rightarrow Z_{t}$ in probability for any $t\in [0,T]$ when $\epsilon$ goes to zero.
%
%Then $Z^{\varepsilon}$ converges to $Z$ ucp. 
%\end{lem}
%
%
%
A fundamental property of the tensor product of Banach spaces which will be used in the whole 
paper is the following.
If $\tilde{T}:E\times F \rightarrow \R$ is a continuous
 bilinear form, there exists a unique bounded linear operator 
$T:E\hat{\otimes}F \rightarrow \R$ satisfying 
$\prescript{}{(E\hat{\otimes}_{\pi}F)^{\ast}}{\langle} T,e\otimes f
 \rangle_{E\hat{\otimes}_{\pi}F} =T(e\otimes f)=\tilde{T}(e,f)$ for every $e\in E$, $f\in F$. We observe moreover that there exists a canonical
 identification between $\mathcal{B}(E,F)$ and $L(E;F^\ast)$ which identifies $\tilde{T}$ with $\bar{T}:E\rightarrow F^{\ast}$ by $\tilde{T}(e,f)=\bar{T}(e)(f)$. 
Summarizing, there is an isometric isomorphism between the dual space of the projective tensor product and the space of bounded bilinear forms equipped with the usual norm, i.e. 
\begin{equation}				\label{eq IDBILPIDUAL}
(E\hat{\otimes}_{\pi} F)^{\ast}\cong \mathcal{B}(E,F)\cong L(E;F^{\ast}) \; .
\end{equation}
With this identification, the action of a bounded bilinear form $T$ as a bounded linear functional on $E\hat{\otimes}_{\pi}F$ is given by
\begin{equation}
\prescript{}{(E\hat{\otimes}_{\pi}F)^{\ast}}{\langle} T, \sum_{i=1}^{n}x_{i}\otimes y_{i}\rangle_{E\hat{\otimes}_{\pi}F} = 
T\left(\sum_{i=1}^{n}x_{i}\otimes y_{i} \right)=\sum_{i=1}^{n}\tilde{T}(x_{i},y_{i})=\sum _{i=1}^{n}\bar{T}(x_{i})(y_{i}).
\end{equation}
In the sequel that identification will  often be used without explicit mention.\\
The importance of tensor product spaces and their duals 
is justified  first of all by identification 
\eqref{eq IDBILPIDUAL}: indeed the second order Fr\'echet derivative of a real function defined on a Banach space $E$ belongs to $\mathcal{B}(E,E)$. 
We state a useful result involving Hilbert tensor products and 
Hilbert direct sums. 
\begin{prop}		\label{pr DIRSUMHS}
Let $E$ and $F_{1}$, $F_{2}$ be Hilbert spaces.% such that $Y_{1}\cap Y_{2}=\{0\}$. 
We consider $F=F_{1}\oplus F_{2}$ equipped with the Hilbert direct norm.
Then $E\hat{\otimes}_{h}F=(E\hat{\otimes}_{h}
F_{1})\oplus (E\hat{\otimes}_{h} F_{2})$.
\end{prop}
\begin{proof} \
Since $E\otimes F_{i}\subset E\otimes F$, $i=1,2$ we can write 
$E\otimes_{h} F_{i}\subset X\otimes_{h} Y$ and so 
\begin{equation}		\label{eq RTG4}
(E\hat{\otimes}_{h}
F_{1})\oplus (E\hat{\otimes}_{h} F_{2})
\subset E\hat{\otimes}_{h}F
\end{equation}
Since we handle with Hilbert norms, it is easy to show that the norm topology of 
$E\hat{\otimes}_{h} F_{1}$ and $E\hat{\otimes}_{h} F_{2}$ is the same as
 the one induced by $E\hat{\otimes}_{h}F$.\\
It remains to show the converse inclusion of \eqref{eq RTG4}. This follows because 
$E \otimes F\subset E\hat{\otimes}_{h} F_{1}\oplus E\hat{\otimes}_{h} F_{2}$.
\end{proof}
We recall another important property.
\begin{equation}		\label{eq 1.15bis}
\mathcal{M}([-\tau,0]^{2})= 
\left(C([-\tau,0]^{2})\right)^{\ast}\subset
\left(C([-\tau,0])\hat{\otimes}_{\pi}C([-\tau,0])\right)^{\ast}\cong \mathcal{B}(C([-\tau,0]),C([-\tau,0])) \, .
\end{equation}
With every $\mu\in \mathcal{M}([-\tau,0]^{2}) $ we can associate a unique operator $T^{\mu}\in \mathcal{B}(C([-\tau,0]), C([-\tau,0])) $ defined by 
$T^{\mu}(f,g)=\int_{[-\tau,0]^{2}}f(x)g(y)\mu(dx,dy)$.\\
Let $\eta_{1}$, $\eta_{2}$ be two elements in $C([-\tau,0])$. The element $\eta_{1}\otimes \eta_{2}$ in the 
algebraic tensor product $C([-\tau,0])\otimes^{2}$ will be identified with the element $\eta$ in $C([-\tau,0]^{2})$ 
defined by $\eta(x,y)=\eta_{1}(x)\eta_{2}(y)$ for all $x$, $y$ in $[-\tau,0]$. So if $\mu$ is a measure on
 $\mathcal{M}([-\tau,0]^{2})$, the pairing duality $\prescript{}{\mathcal{M}([-\tau,0]^{2})}{\langle} \mu, \eta_{1}\otimes \eta_{2}\rangle_{C([-\tau,0]^{2})} $ has to be understood as the following pairing duality: 
\begin{equation}	\label{eq PDUALCM}
%\prescript{}{\mathcal{M}([-\tau,0]^{2})}{\langle} \mu, \eta_{1}\otimes \eta_{2}\rangle_{C([-\tau,0]^{2})} =
\prescript{}{\mathcal{M}([-\tau,0]^{2})}{\langle} \mu, \eta\rangle_{C([-\tau,0]^{2})} =\int_{[-\tau,0]^{2}}\eta(x,y)\mu(dx,dy)=\int_{[-\tau,0]^{2}}\eta_{1}(x)\eta_{2}(y)\mu(dx,dy)  \; .
\end{equation}
In the It\^o formula for $B$ valued processes at Section \ref{sec:ito-int}, naturally appear the first and second order Fr\'echet derivatives of some functionals defined on a general Banach space $B$. 
When $B=C([-\tau,0])$, the first derivative belongs to  $\mathcal{M}([-\tau,0])$ and second derivative mostly belongs to $\mathcal{M}([-\tau,0]^{2})$.
In particular in Sections \ref{sec: evaluation} and \ref{Sec6} those spaces and their subsets appear in relation with window processes. 
%In fact those spaces appear naturally when we deal with first and second order Fr\'echet derivatives of functionals on $C([-\tau,0])$. 
%
We introduce a notation which has been already used in the Introduction.
%\begin{rem}
%Let $a\in \R$, we recall that $\delta_{a}$ denotes the Dirac measure concentrated at $a$. In particular if $a=0$, $\delta_0$ is 
%the Dirac measure at $0$.
%\end{rem}
\begin{nota}    	\label{nota2.3}
\begin{enumerate}
\item If $a \in \R$, we remind that $\delta_a$ will denote the  Dirac measure
 concentrated at $a$, so $\delta_0$ stands for the  Dirac measure at zero.
\item Let $\mu$ be a measure on $\mathcal{M}([-\tau,0])$, $\tau>0$. 
$\mu^{\delta_0}$ will denote the scalar defined by $\mu(\{0\})$ and 
$\mu^\perp$ will denote the measure defined by $\mu-\mu^{\delta_0}\delta_{0}$. 
If $\mu^\perp$ is absolutely continuous with respect to Lebesgue measure, 
its density will be denoted with the same letter $\mu^\perp$.
\end{enumerate}
\end{nota}
%We recall the definition of Fr\'echet differentiability. 
Let $B$ be a Banach space and $I$ be a real interval, 
typically $I = [0,T]$ or $I = [0,T[$.  
 A function $F:I \times B \longrightarrow 
\mathbb{R} $, is said to be Fr\'echet of class $C^{1,2}(I \times B)$, 
%or $C^{1,2}$  for short, 
if the following properties are fulfilled.
\begin{itemize}
\item $F$ is once Fr\'echet continuously differentiable; the partial derivative
 with  respect to $t$ will be denoted by $\partial_{t} F : I\times B 
\longrightarrow \mathbb{R}$;
%with respect to $I$ 
\item for any $t \in I$,  $\eta \mapsto DF(t,\eta)$ is of class $C^1$
where $DF:I \times B \longrightarrow B^*$ denotes the Fr\'echet 
derivative with respect to the second argument; 
\item the second order Fr\'echet derivative with respect to the 
second argument $D^2F: I \times B \rightarrow (B\hat{\otimes}_\pi B)^\ast $
is  continuous.
\end{itemize}
Similar notations are self-explained as for instance 
 or $C^{1,1}(I \times B)$.
%%%%%%%%%%%%%%%%%%%%%%%%%%%%%%%%%%%%%%%%%%%%%%%%%%%%%%%%%%
 
\section{Chi-covariation and Chi-quadratic variation}	
	\label{sec: chiqv}

\subsection{Notion and  examples of Chi-subspaces}
%Let $B_{1}$, $B_{2}$, $B$ be three Banach spaces. 
\begin{dfn}		\label{DefChiCOV}
Let $E$ be a Banach space. A Banach space $\chi$  included in $E$ will be said a
{\bf continuously embedded Banach subspace} of $E$
if the inclusion of $\chi$ into $E$ is continuous. \\
If $E = (B_{1}\hat{\otimes}_{\pi}B_{2})^{\ast} $ then
$\chi$ will be said {\bf Chi-subspace} (of $E$). 
%Let $\chi$ be a continuously injected Banach  $\left( \chi, \|\cdot\|_{\chi} \right)$
%of $(B_{1}\hat{\otimes}_{\pi}B_{2})^{\ast}$
% with Banach structure (or simply {\it Banach subspace} of
%$(B_{1}\hat{\otimes}_{\pi}B_{2})^{\ast}$). If the injection is continuous
%into $(B_{1}\hat{\otimes}_{\pi}B_{2})^{\ast}$ 
%$\chi$ will be called a {\bf Chi-subspace (of $(B_{1}\hat{\otimes}_{\pi}B_{2})^{\ast}$)}. We will often omit in the sequel the mention
%`` \textbf{of $(B_{1}\hat{\otimes}_{\pi}B_{2})^{\ast}$} ''.
\end{dfn}
%The results below follow immediately from the definition.

\begin{rem} 	\label{pro 3.2}
\begin{enumerate}
\item Let $\chi$ be a linear subspace 
 of $(B_{1}\hat{\otimes}_{\pi}B_{2})^{\ast}$ with Banach structure.
%, i.e. a linear normed subspace.
$\chi$ is a Chi-subspace %of $(B_{1}\hat{\otimes}_{\pi}B_{2})^{\ast}$ 
if and only if 
$
\|\cdot\|_{ (B_{1}\hat{\otimes}_{\pi}B_{2})^{\ast}}\leq \|\cdot\|_{\chi}$,
where $\|\cdot\|_{\chi}$ is a norm related to the topology of $\chi$.
\item Any continuously embedded Banach subspace of a Chi-subspace
  is a Chi-subspace.
\item Let $\chi_{1}, \cdots, \chi_{n}$ be Chi-subspaces %of $(B_{1}\hat{\otimes}_{\pi}B_{2})^{\ast}$ 
such that, for any $1 \le i \neq j \le n$, $\chi_i \bigcap \chi_j =
\{0\}$  where $0$ is the zero of $(B_{1}\hat{\otimes}_{\pi}B_{2})^{\ast}$.
Then the normed space $\chi=\chi_{1}\oplus \cdots \oplus \chi_{n}$ 
is a Chi-subspace.% of $(B_{1}\hat{\otimes}_{\pi}B_{2})^{\ast}$.
\end{enumerate}
\end{rem}
The last item allows to express a Chi-subspace of $(B_{1}\hat{\otimes}_{\pi}B_{2})^{\ast}$ as direct sum of Chi-subspaces (of $(B_{1}\hat{\otimes}_{\pi}B_{2})^{\ast}$). 
This, together with Proposition \ref{prop somma diretta di qv}, helps to evaluate the
$\chi$-covariations and the $\chi$-quadratic variations of different processes.

Before providing the definition of the so-called $\chi$-covariation
% (and of the $\chi$-quadratic variation)
 of a couple of a $B_{1}$-valued and a $B_{2}$-valued stochastic processes,
 we will give 
some examples of  Chi-subspaces that  we will use in the
paper. 
\begin{ese} {Let $B_{1}$, $B_{2}$ be two Banach spaces.}		\label{ex 4.3}
\begin{itemize}
\item 
$\chi= (B_{1}\hat{\otimes}_{\pi}B_{2})^{\ast}$.
This appears in our elementary situation anticipated in the Introduction, see also Proposition \ref{pr TGH78}.
\end{itemize}
\end{ese}
\begin{ese}{Let $B_{1}=B_{2}=C([-\tau,0])$.}\\	\label{ex 4.5}
This is the natural value space for all the windows of continuous processes.
We list some examples of Chi-subspaces $\chi$ for which some window processes have a $\chi$-covariation or a $\chi$-quadratic variation. Moreover those $\chi$-covariation and $\chi$-quadratic variation 
will intervene in some applications stated at Section \ref{Sec6}.
 Our basic reference Chi-subspace of 
 $(C([-\tau,0]\hat{\otimes}_{\pi}C([-\tau,0]))^{\ast}$
will be the Banach space $\shm([-\tau,0]^{2})$ equipped with the usual total variation
norm,  denoted by $\Vert \cdot \Vert_{Var}$. 
%This is in fact a proper subspace of $(B\hat{\otimes}_{\pi}B)^{\ast}$ according to \eqref{eq 1.15bis}. 
The inequality in item 1 of Remark \ref{pro 3.2} is verified since 
$\|T^{\mu}\|_{(B\hat{\otimes}_{\pi}B)^{\ast}} = \sup_{\|f\|\leq1, \|g\|\leq 1} \left| T^{\mu} (f, g) \right|\leq \|\mu\|_{Var} $
for every $\mu\in \shm([-\tau,0]^{2})$.
All the other spaces considered in the sequel 
of the present example will be shown to be
continuously embedded Banach subspaces of $\shm([-\tau,0]^{2})$; by item 2
 of Remark \ref{pro 3.2} they are Chi-subspaces. % of $(B\hat{\otimes}_{\pi}B)^{\ast}$.\\
Here is a list. Let $a, b$ two fixed given points in $[-\tau,0]$. 
\begin{itemize}
\item $L^{2}([-\tau,0]^{2})\cong L^{2}([-\tau,0])\hat{\otimes}_{h}^{2}$ is a Hilbert subspace of 
$\shm([-\tau,0]^{2})$, equipped with the norm derived from the usual scalar
 product. The Hilbert tensor product $L^{2}([-\tau,0])\hat{\otimes}_{h}^{2}$ will be 
always identified with $L^{2}([-\tau,0]^{2})$, conformally to a quite canonical procedure, see \cite{nevpg}, chapter 6.
\item   			
$\shd_{a,b}([-\tau,0]^{2})$ (shortly $\shd_{a,b}$) which denotes the one dimensional Hilbert space 
of the multiples of the Dirac measure concentrated at 
$(a,b)\in [-\tau,0]^{2}$, i.e.
\begin{equation}
\mathcal{D}_{a,b}([-\tau,0]^{2}):= \{ \mu \in \mathcal{M}([-\tau,0]^{2});\; s.t. \mu(dx,dy)=\lambda \,\delta_{a}(dx)\delta_{b}(dy) \textrm{ with } \lambda \in \mathbb{R} \}
\cong \shd_{a}\hat{\otimes}_{h} \shd_{b}\; . 	\label{eq-def Dij}
\end{equation}
%The space $\shd_{0,0}$ (resp. $\shd_{-\tau,-\tau}$) will be the space of Dirac's measures concentrated at $(0,0)$ (resp. $(-\tau,-\tau)$).
If $\mu=\lambda \,\delta_{a}(dx)\delta_{b}(dy)$ then $\|\mu\|_{Var}=|\lambda|=\|\mu\|_{\shd_{a,b}}$.
\item 			
$\shd_{a}([-\tau,0])\hat{\otimes}_{h} L^{2}([-\tau,0])$ and $L^{2}([-\tau,0]) \hat{\otimes}_{h}   \shd_{a}([-\tau,0])$ where $\shd_{a}([-\tau,0])$ (shortly $\shd_{a}$) denotes the 			
one-dimensional space of multiples of the
 Dirac measure concentrated at $a\in [-\tau,0]$ , i.e. 
\begin{equation}			\label{eq-def Di}
\mathcal{D}_{a} ([-\tau,0]):= 
\{\mu \in \mathcal{M}([-\tau,0]);\,s.t. \mu(dx)=\lambda \,\delta_{a}(dx) \textrm{ with } \lambda\in \mathbb{R} \} \; .
\end{equation}
$\shd_{a}([-\tau,0])\hat{\otimes}_{h} L^{2}([-\tau,0])$ (resp. $L^{2}([-\tau,0]) \hat{\otimes}_{h}   \shd_{a}([-\tau,0])$) is a Hilbert subspace of $\shm([-\tau,0]^{2})$ and for a general element in this space
  $\mu=\lambda\delta_{a}(dx)\phi(y)dy$ (resp. $\mu=\lambda \phi(x)dx\delta_{a}(dy)$ ),
$\phi \in L^{2}([-\tau,0])$,  we have 
$ \|\mu\|_{Var}\leq \|\mu \|_{\shd_{a}([-\tau,0])\hat{\otimes}_{h} L^{2}([-\tau,0])} ( \mbox{resp. } \|\mu \|_{L^{2}([-\tau,0]) \hat{\otimes}_{h}   \shd_{a}([-\tau,0])} ) =|\lambda |\cdot \|\phi\|_{L^{2}}$.
\item $\chi^{0}([-\tau,0]^{2})$, $\chi^{0}$ shortly, which denotes the subspace of measures defined as 
$			
\chi^{0}([-\tau,0]^{2}):=( \mathcal{D}_{0}([-\tau,0]) \oplus L^{2}([-\tau,0]))\hat{\otimes}^{2}_{h}  
$. 
\begin{rem}
An element $\mu$ in $\chi^{0}([-\tau,0]^{2})$ can be uniquely decomposed as
$\mu= \phi_{1} + \phi_{2}\otimes  \delta_{0} + 
 \delta_{0}\otimes\phi_{3}+\lambda \delta_{0}\otimes\delta_{0},
$ where $\phi_{1}\in  L^{2}([-\tau,0]^{2})$, $\phi_{2}$, $\phi_{3}$ are functions in $L^{2}([-\tau,0])$ and $\lambda$ is a real number. We have $ \mu \left(\{0,0 \}  \right)=\lambda  $. 
\end{rem}
\item  $Diag([-\tau,0]^{2})$ (shortly $Diag$), will denote the subset 	of $\mathcal{M}([-\tau,0]^{2})$ defined as follows:
\begin{equation} 			\label{eq-def diag}
Diag([-\tau,0]^{2}):=\left\{\mu^{g} \in \mathcal{M}([-\tau,0]^{2})\, s.t.\,
\mu^{g}(dx,dy)=g(x)\delta_{y}(dx)dy;\, g \in L^{\infty}([-\tau,0])
\right\}
\; .
\end{equation}
$Diag([-\tau,0]^{2})$, equipped with the norm $\| \mu^{g} \|_{Diag([-\tau,0]^{2})}= \| g\|_{\infty}$, is a Banach space. 
Let $f$ be a function in $C([-\tau,0]^{2})$; the pairing duality \eqref{eq PDUALCM} between $f$ and $\mu(dx,dy)=g(x)\delta_{y}(dx)dy\in Diag$ gives
\begin{equation}			\label{eq-def dualita con misura diag}  
\prescript{}{C([-\tau,0]^{2})}{\langle} f,\mu \rangle_{ Diag([-\tau,0]^{2})}%=\int_{[-\tau,0]^{2}} f(x,y) \mu(dx,dy)
= \int_{[-\tau,0]^{2}}f(x,y) g(x)\delta_{y}(dx)dy=\int_{-\tau}^{0}f(x,x)g(x)dx\; .
\end{equation}
\end{itemize}
\end{ese}
 A closed subspace of  $Diag([-\tau,0]^{2})$ is given below.
\begin{nota} \label{nnn}
We denote by  $Diag_{d}([-\tau,0]^{2})$ the subspace constituted by the 
measures $\mu^{g}\in Diag([-\tau,0]^{2})$ for which 
$g$ belongs to  the space  $D([-\tau,0])$ of the (classes of) bounded functions $g:[-\tau,0]\longrightarrow \R$ admitting a  c\`adl\`ag version.
% CRI MOMENTANEAMENTE SOPPRESSO
%We denote by $Diag_{c}([-\tau,0]^{2})$ (resp. $Diag_{d}([-\tau,0]^{2})$) the subset constituted by measures $\mu^{g}\in Diag([-\tau,0]^{2})$ for which 
%$g$ belongs to $C([-\tau,0])$ (resp. in $D([-\tau,0])$). We remind that $D([-\tau,0])$ is the set of the (classes of) bounded functions $g:[-\tau,0]\longrightarrow \R$ admitting a  c\`adl\`ag version.
\end{nota}
%\begin{equation} 			\label{eq-def diagC}
%Diag_{c}([-\tau,0]^{2}):=\left\{\mu\in \mathcal{M}([-\tau,0]^{2})\, s.t.\,
%\mu(dx,dy)=g(x)\delta_{y}(dx)dy;\, g \in C([-\tau,0])
%\right\}
%\; ,
%\end{equation}
%\begin{equation} 			\label{eq-def diagD}
%Diag_{d}([-\tau,0]^{2}):=\left\{\mu\in \mathcal{M}([-\tau,0]^{2})\, s.t.\,
%\mu(dx,dy)=g(x)\delta_{y}(dx)dy;\, g \textrm{ cadlag on } ([-\tau,0])
%\right\}
%\; .
%\end{equation}

%
\subsection{Definition of $\chi$-covariation and some related results}		\label{notionChi}
Let $B_1$, $B_2$ and $B$ be three Banach spaces.
In this subsection, we introduce  the definition of $\chi$-covariation between a $B_{1}$-valued stochastic process $\X$ and a $B_{2}$-valued stochastic process $\Y$. 
 We remind that $\mathscr{C}([0,T])$ denotes the space of continuous processes equipped with the ucp topology.\\
Let $\X$ (resp. $\Y$) be $B_{1}$ (resp. $B_{2}$) valued stochastic process. 
Let $\chi$ be a Chi-subspace of $(B_{1}\hat{\otimes}_{\pi}B_{2})^{\ast}$ and $\epsilon > 0$. 
We denote by $[\X,\Y]^{\epsilon}$,
the following application
\begin{equation}		\label{eq Xepsilon}
[\X,\Y]^{\epsilon}:\chi\longrightarrow \mathscr{C}([0,T])
\quad \textrm{defined by}\quad
\phi
\mapsto
\left( \int_{0}^{t} \prescript{}{\chi}{\langle} \phi,
\frac{J\left(  \left(\X_{s+\epsilon}-\X_{s}\right)\otimes \left(\Y_{s+\epsilon}-\Y_{s}\right)  \right)}{\epsilon} 
\rangle_{\chi^{\ast}} \,ds
\right)_{t\in [0,T]},
\end{equation}
where $ J: B_{1}\hat{\otimes}_{\pi}B_{2} \longrightarrow (B_{1}\hat{\otimes}_{\pi}B_{2})^{\ast\ast}$ is the canonical injection between a space and its bidual. 
%as introduced in subsection \ref{sec: pre}.\\ 
With application $[\X,\Y]^{\epsilon}$ it is possible to associate
 another one, 
denoted by $\widetilde{[\X,\Y]}^{\epsilon}$, defined by
\[
\widetilde{[\X,\Y]}^{\epsilon}(\omega,\cdot):[0,T]\longrightarrow \chi^{\ast}
\quad \textrm{such that}\quad
t\mapsto \left(\phi\mapsto
\int_{0}^{t}  \prescript{}{\chi}{\langle} \phi,\frac{J\left(  
  \left(\X_{s+\epsilon}(\omega)-\X_{s}(\omega)\right)\otimes 
\left(\Y_{s+\epsilon}(\omega)-\Y_{s}(\omega)\right) \right)}{\epsilon}
 \rangle_{\chi^{\ast}} \,ds\right) .
\]
\begin{rem} 
\begin{enumerate}
\item We recall that $\chi\subset(B_{1}\hat{\otimes}_{\pi}B_{2})^{\ast}$ implies 
$(B_{1}\hat{\otimes}_{\pi}B_{2})^{\ast\ast}\subset \chi^{\ast}$. 
\item As indicated, $\prescript{}{\chi}{\langle}\cdot,\cdot\rangle_{\chi^{\ast}}$ denotes the duality
between the space $\chi$ and its dual $\chi^{\ast}$. In fact by assumption,
 $\phi$ is an element of $\chi$ and element $J\left(
  \left(\X_{s+\epsilon}-\X_{s}\right)\otimes   \left(\Y_{s+\epsilon}-\Y_{s}\right) \right)$ naturally belongs to $(B_{1}\hat{\otimes}_{\pi}B_{2})^{\ast\ast}\subset \chi^{\ast}$. 
\item
With a slight abuse of notation, in the sequel the injection $J$ from $B_{1}\hat{\otimes}_{\pi}B_{2}$ to its bidual will be omitted. The tensor product 
$\left(\X_{s+\epsilon}-\X_{s}\right)\otimes \left(\Y_{s+\epsilon}-\Y_{s}\right) $ has to be considered as the element $J\left(   \left(\X_{s+\epsilon}-\X_{s}\right)
\otimes \left(\Y_{s+\epsilon}-\Y_{s}\right)   \right)$ which belongs to $\chi^{\ast}$.
\item
Suppose $B_{1}=B_{2}=B=C([-\tau,0])$ and let $\chi$ be a Chi-subspace.\\
An element of the type $\eta=\eta_{1}\otimes\eta_{2}$, $\eta_{1}$, $\eta_{2} \in B$, can be either considered 
as an element of the type $B\hat{\otimes}_{\pi}B\subset (B\hat{\otimes}_{\pi}B)^{\ast\ast}\subset \chi^{\ast}$ or as an element 
of $C([-\tau,0]^{2})$ defined by $\eta(x,y)=\eta_{1}(x)\eta_{2}(y)$.
When $\chi$ is indeed a closed subspace of $\mathcal{M}([\tau,0]^{2})$, then 
the pairing between $\chi$ and $\chi^{\ast}$ will be compatible with the pairing 
duality between $\mathcal{M}([\tau,0]^{2})$ and $C([-\tau,0]^{2})$ given by \eqref{eq PDUALCM}.
\end{enumerate}
\end{rem}
\begin{dfn}		\label{def CHICOV} 
Let $B_{1}$, $B_{2}$ be two Banach spaces and $\chi$ be a Chi-subspace of 
$(B_{1}\hat{\otimes}_{\pi} B_{2})^{\ast}$. % such that $\|\cdot\|_{\chi}\geq \|\cdot\|_{(B_{1}\hat{\otimes}_{\pi} B_{2})^{\ast}}$ . 
Let $\X$ (resp. $\Y$) be a $B_{1}$ (resp. $B_{2}$) valued stochastic process. 
We say that {\bf $(\X, \Y)$ admits a $\chi$-covariation} if the following assumptions hold.
\begin{description}
\item[H1] For all sequence $(\epsilon_{n})$ it exists a subsequence $(\epsilon_{n_{k}})$ such that 
\begin{equation} 
\sup_{k}\int_{0}^{T} \sup_{\|\phi\|_{\chi}\leq 1}\left| \langle \phi,\frac{(\X_{s+\epsilon_{n_{k}}}-\X_{s})\otimes(\Y_{s+\epsilon_{n_{k}}}-\Y_{s})}{\epsilon_{n_{k}}}\rangle \right|ds
=\sup_{k}\int_{0}^{T} \frac{\left\| (\X_{s+\epsilon_{n_{k}}}-\X_{s})\otimes(\Y_{s+\epsilon_{n_{k}}}-\Y_{s})\right\|_{\chi^{\ast}} }{\epsilon_{n_{k}}} ds
\;< \infty\; a.s.
\end{equation}
\item[H2]
\begin{description}
\item{(i)} There exists an application $\chi\longrightarrow
  \mathscr{C}([0,T])$, denoted by $[\X,\Y]$, such that
\begin{equation}		\label{H2 cONDucp}
[\X,\Y]^{\epsilon}(\phi)\xrightarrow[\epsilon\longrightarrow 0_{+}]{ucp} [\X,\Y](\phi) 
\end{equation} 
for every $\phi \in \chi\subset
(B_{1}\hat{\otimes}_{\pi}B_{2})^{\ast}$.
\item{(ii)} 
There is a measurable process $\widetilde{[\X,\Y]}:\Omega\times [0,T]\longrightarrow \chi^{\ast}$, 
such that
\begin{itemize}
\item for almost all $\omega \in \Omega$, $\widetilde{[\X,\Y]}(\omega,\cdot)$ is a ( c\`adl\`ag) bounded variation function, 
\item 
$\widetilde{[\X,\Y]}(\cdot,t)(\phi)=[\X,\Y](\phi)(\cdot,t)$ a.s. for all $\phi\in \chi$, $t\in [0,T]$.
\end{itemize}
\end{description}
\end{description}
If $(\X, \Y)$ admits a $\chi$-covariation we will call
 $\chi$-\textbf{covariation} of $\X$ and $\Y$ the $\chi^{\ast}$-valued
process  $(\widetilde{[\X,\Y]})_{0\leq t\leq T}$.
%defined for every $\omega\in \Omega $ and $t\in [0,T]$ by $\phi \mapsto \widetilde{[\X,\Y]}(\omega,t)(\phi)=[\X,\Y](\phi)(\omega,t) $. 
By abuse of notation, $[\X,\Y]$ will also be called $\chi$-covariation
and it will be sometimes confused with $\widetilde{[\X,\Y]}$.
\end{dfn}
\begin{dfn}			\label{DChiQV}   
Let $\X=\Y$ be a $B$-valued stochastic process and $\chi$ be a Chi-subspace of $(B\hat{\otimes}_{\pi}B)^{\ast}$. 
The $\chi$-covariation $[\X,\X]$ (or $\widetilde{[\X,\X]}$) will also be denoted 
by $[\X]$ and $\widetilde{[\X]}$; it will be called \textbf{$\chi$-quadratic variation of $\X$} and we will say that $\X$ has a $\chi$-quadratic variation.
\end{dfn}
\begin{rem}		\label{rem PPO}
\begin{enumerate}
\item For every fixed $\phi\in \chi$, the processes $\widetilde{[\X,\Y]}(\cdot,t)(\phi)$ and $[\X,\Y](\phi)(\cdot,t)$ are indistinguishable. 
In particular the $\chi^{\ast}$-valued process $\widetilde{[\X,\Y]}$ is weakly star continuous, i.e. 
$\widetilde{[\X,\Y]}(\phi)$ is continuous for every fixed $\phi$.
\item 
The existence of $\widetilde{[\X,\Y]}$ guarantees that 
$[\X,\Y]$ admits a   bounded variation version which allows to consider it as pathwise integrator.
%Under Assumption {\bf H2}(i), for fixed $( \omega,t)$ 
%the application  $\phi \mapsto [X,X](\phi)(\omega,t)$ is linear, 
%however for fixed $(\omega,t)$  it could be not continuous.
\item The quadratic variation $\widetilde{[\X]}$ will be the object
 intervening in the second order term of the It\^o formula expanding $F(\X)$ for some
$C^2$-Fr\'echet function $F$, see Theorem \ref{thm ITONOM}.
\item In Corollaries \ref{cor SGS} and \ref{corollary} we will show that, whenever $\chi$ is separable (most of the cases), 
the Condition \textbf{H2} can be relaxed in a significant way.  
In fact the Condition \textbf{H2}(i) reduces to the convergence in probability of \eqref{H2 cONDucp} on a dense subspace  and \textbf{H2}(ii) will be automatically 
satisfied. 
 \end{enumerate}
\end{rem}

\begin{rem} \label{rem CSH1}
\begin{enumerate}
\item  
A practical criterion to verify Condition {\bf H1} is 
\begin{equation} \label{eq CSH1}
 \frac{1}{\epsilon}  \int_{0}^{T}\left\| (\X_{s+\epsilon}-\X_{s})
\otimes (\Y_{s+\epsilon}-\Y_{s}) \right\|_{\chi^{\ast}} ds \leq  B(\epsilon)
\end{equation}
where $B(\epsilon)$ converges in probability when $\epsilon$ goes to zero.
 In fact 
the convergence in probability implies the a.s. convergence  of a subsequence. 
%and the convergence implies the boundness.
\item A consequence of Condition {\bf H1} 
is  that for all $(\epsilon_{n})\downarrow 0$ there exists a subsequence
$(\epsilon_{n_{k}})$ such that
\begin{equation}
\sup_{k}\|\widetilde{[\X,\Y]}^{\epsilon_{n_{k}}}\|_{Var([0,T])}<\infty
\quad a.s. 
\end{equation}
In fact $\|\widetilde{[\X,\Y]}^{\epsilon}\|_{Var([0,T])}\leq\frac{1}{\epsilon}
\int_{0}^{T}\|  (\X_{s+\epsilon}-\X_{s})\otimes   (\Y_{s+\epsilon}-\Y_{s})\|_{\chi^{\ast}}ds  $, which implies that 
$\widetilde{[\X,\Y]}^{\epsilon}$ is a $\chi^{\ast}$-valued process with 
bounded variation on $[0,T]$.
As a consequence, for a 
$\chi$-valued continuous stochastic process $\mathbb{Z}$, $t\in [0,T]$, the integral $\int_0^t
\prescript{}{\chi}{\langle} \mathbb{Z}_{s},d\widetilde{[\X,\Y]}_{s}^{\epsilon_{n_{k}}}\rangle_{\chi^{\ast}}$ is a
well-defined 
Lebesgue-Stieltjes type integral for almost all $\omega\in\Omega$.
\end{enumerate}
\end{rem}
\begin{rem}
\begin{enumerate}
\item To a Borel function $G:\chi\longrightarrow C([0,T])$ we can associate $\tilde{G}:[0,T]\longrightarrow
\chi^{\ast}$ setting $\tilde{G}(t)(\phi)= G(\phi)(t)$. 
By definition $\tilde{G}:[0,T]\longrightarrow \chi^{\ast}$ has bounded variation
if
$\|\tilde{G}\|_{Var([0,T])}:=\sup_{\sigma\in\Sigma_{[0,T]}}
\sum_{i \vert (t_{i})_{i}=\sigma}\left\|\tilde{G}(t_{i+1})-\tilde{G}(t_{i})\right\|_{\chi^{\ast}}=
\sup_{\sigma\in\Sigma_{[0,T]}}
\sum_{i \vert (t_{i})_{i} =\sigma}\sup_{\|\phi\|_{\chi}\leq
1}\left|G(\phi)(t_{i+1})-G(\phi)(t_{i})\right| %<+ \infty\, ,
$ 
is finite, where $\Sigma_{[0,T]}$ is the set of all possible partitions $\sigma=(t_{i})_{i}$ of the 
interval $[0,T]$.
This
quantity is the {\bf total variation} of $\tilde{G}$.
For example if $G(\phi)=\int_{0}^{t}\dot{G}_{s}(\phi)\,ds$ with $\dot{G}:\chi\rightarrow C([0,T])$ Bochner integrable, 
then $\|G\|_{Var[0,T]}\leq \int_{0}^{T}\sup_{\|\phi\|_{\chi} \leq
1}|\dot{G}_{s}(\phi)|\,ds$.
\item If $G(\phi)$, $\phi\in \chi$ is a family of stochastic processes, it is not obvious 
to find a good version $\tilde{G}:[0,T]\longrightarrow \chi^{\ast}$ of $G$. This will be the object of Theorem \ref{thm QV}.
%\item Unfortunately the situation is more complicated for
%stochastic processes. 
%Let $F:\chi\longrightarrow \mathscr{C}([0,T])$. For 
%every $\phi\in \chi$, $F(\phi)\in C([0,T])$ a.s. and we 
%associate
%$\tilde{F}(\omega,t)(\phi)=F(\phi)(\omega,t)$. It may happen that for fixed $\omega\in \Omega$, $t\in [0,T]$ 
%the linear form $\tilde{F}(\omega,t)$ is not continuous. 
%In fact given $\phi_{n}\longrightarrow \phi$ in $\chi$,
%$F(\phi_{n})\longrightarrow F(\phi)$ in $\mathscr{C}([0,T])$ with
%the ucp convergence, then there is only a subsequence such that
%$F(\phi_{n_{k}})\longrightarrow F(\phi)$ a.s. in $C([0,T])$.
\end{enumerate}
\end{rem}
\begin{dfn}		\label{DQVWS}
If the $\chi$-covariation exists with $\chi=(B_{1}\hat{\otimes}_{\pi} B_{2})^{\ast}$, we say that $(\X, \Y)$ admits a \textbf{global covariation}. 
Analogously if $\X$ is $B$-valued and the $\chi$-quadratic variation exists with $\chi=(B\hat{\otimes}_{\pi} B)^{\ast}$, we say that $\X$ admits a \textbf{global quadratic variation}.
\end{dfn}
\begin{rem}   \label{rem456}
\begin{enumerate}
\item $\widetilde{[\X,\Y]}$ takes values ``a priori'' in $(B_{1}\hat{\otimes}_{\pi} B_{2})^{\ast\ast}$. 
\item If $[\X,\Y]^{\R}$ exists then Condition \textbf{H1} follows by Remark \ref{rem CSH1}.1.
\end{enumerate}
\end{rem}

\begin{prop}			\label{pr TGH78}
Let $\X$ (resp. $\Y$) be a $B_{1}$-valued (resp. $B_{2}$-valued) process
 such that 
%$\X$ and $\Y$ admit a scalar quadratic variation 
 $(\X,\Y)$ admits a scalar and tensor covariation. Then $(\X,\Y)$ 
admits a global  covariation.
In particular the global covariation process
 takes values in $B_{1}\hat{\otimes}_{\pi}B_{2}$ and $\widetilde{[\X,\Y]}=[\X,\Y]^{\otimes}$ a.s. 
\end{prop}
\begin{proof}
We set $\chi=(B_{1}\hat{\otimes}_{\pi} B_{2})^{\ast}$. 
Taking into account Remark \ref{rem456}.2, it will be enough to verify Condition \textbf{H2}.
Recalling the definition of $[\X,\Y]^{\epsilon}$ at \eqref{eq Xepsilon} and the definition of injection $J$ we observe that 
\begin{equation}		\label{eq SDD}
\begin{split}
[\X,\Y]^{\epsilon}(\phi)(\cdot, t) &= 
%\int_{0}^{t} \prescript{}{(B_{1}\hat{\otimes}_{\pi}B_{2})^{\ast} }{\langle} \phi, \frac{J\left(  \left(\X_{s+\epsilon}-\X_{s}\right)\otimes  \left(\Y_{s+\epsilon}-\Y_{s}\right)   \right)}{\epsilon} 
%\rangle_{(B_{1}\hat{\otimes}_{\pi}B_{2})^{\ast\ast}} \,ds \\
%&
%= 
\int_{0}^{t} \prescript{}{(B_{1}\hat{\otimes}_{\pi}B_{2})^{\ast} }{\langle} \phi,
\frac{ \left(\X_{s+\epsilon}-\X_{s}\right)\otimes  \left(\Y_{s+\epsilon}-\Y_{s}\right)}{\epsilon} 
\rangle_{B_{1}\hat{\otimes}_{\pi}B_{2}} \,ds  \; .
\end{split}
\end{equation}
Since Bochner integrability implies Pettis integrability, for every $\phi\in (B_{1}\hat{\otimes}_{\pi}B_{2})^{\ast}$, we also have 
\begin{equation}		\label{eq SDD1}
 \prescript{}{(B_{1}\hat{\otimes}_{\pi}B_{2})^{\ast} }{\langle} \phi,  [\X,\Y]^{\otimes, \epsilon}_{t}\rangle_{B_{1}\hat{\otimes}_{\pi}B_{2}} = 
 \int_{0}^{t} \prescript{}{(B_{1}\hat{\otimes}_{\pi}B_{2})^{\ast} }{\langle} \phi,   \frac{ \left(\X_{s+\epsilon}-\X_{s}\right)\otimes  \left(\Y_{s+\epsilon}-\Y_{s}\right)}{\epsilon} 
\rangle_{B_{1}\hat{\otimes}_{\pi}B_{2}} \,ds   \; .
\end{equation}
\eqref{eq SDD} and \eqref{eq SDD1} imply that 
\begin{equation}		\label{eq SDD4}
[\X,\Y]^{\epsilon}(\phi)(\cdot, t)= \prescript{}{(B_{1}\hat{\otimes}_{\pi}B_{2})^{\ast} }{\langle} \phi,  [\X,\Y]^{\otimes, \epsilon}_{t}\rangle_{B_{1}\hat{\otimes}_{\pi}B_{2}} \hspace{1cm} \textrm{a.s.}
\end{equation}
Concerning the validity of Condition  \textbf{H2} we will show that
\begin{equation}		\label{eq SDD2}
\sup_{t\leq T} \left| [\X,\Y]^{\epsilon}(\phi)(\cdot, t) -   \prescript{}{(B_{1}\hat{\otimes}_{\pi}B_{2})^{\ast} }{\langle} \phi,  [\X,\Y]^{\otimes}_{t}\rangle_{B_{1}\hat{\otimes}_{\pi}B_{2}} \right| \xrightarrow[\epsilon\longrightarrow 0]{\mathbb{P}} 0  \; .
\end{equation}
By \eqref{eq SDD4} the left-hand side of \eqref{eq SDD2} 
%and using \eqref{eq SDD4},
 gives
\[
%\begin{split}
%\sup_{t\leq T} \left| [\X,\Y]^{\epsilon}(\phi)(\cdot, t) -   \prescript{}{(B_{1}\hat{\otimes}_{\pi}B_{2})^{\ast} }{\langle} \phi,  [\X,\Y]^{\otimes}_{t}\rangle_{B_{1}\hat{\otimes}_{\pi}B_{2}} \right|& =
\sup_{t\leq T} \left| \prescript{}{(B_{1}\hat{\otimes}_{\pi}B_{2})^{\ast} }{\langle} \phi,  [\X,\Y]^{\otimes, \epsilon}_{t} - [\X,\Y]^{\otimes}_{t} \rangle_{B_{1}\hat{\otimes}_{\pi}B_{2}} \right| 
%\\ &
\leq \left\| \phi\right\|_{(B_{1}\hat{\otimes}_{\pi}B_{2})^{\ast}}  \; \sup_{t\leq T} \left\|[\X,\Y]^{\otimes, \epsilon}_{t} - [\X,\Y]^{\otimes}_{t}  
\right\|_{B_{1}\hat{\otimes}_{\pi}B_{2}},
%\end{split}
\] 
where the last quantity converges to zero in probability by Definition \ref{dfn REALTENSOQV} item 2 of the tensor quadratic variation; this implies \eqref{eq SDD2}.
The tensor quadratic variation has always bounded variation because of item 2 of Remark \ref{rem RE1}. 
In conclusion \textbf{H2}(ii) is also verified.
\end{proof}
\begin{rem}{We observe some interesting features related to the global covariation, i.e. the $\chi$-covariation 
 when $\chi=(B_{1}\hat{\otimes}_{\pi} B_{2})^{\ast}$.}  
\begin{enumerate}
\item When $\chi$ is separable, for  any $t \in [0,T]$, 
there exists a null subset $N$ of $\Omega$ and a sequence $(\epsilon_n)$ 
such that 
$
\widetilde{[\X,\Y]}^{\epsilon_n}(\omega,t)
\xrightarrow [\epsilon\longrightarrow 0]{}\widetilde{[\X,\Y]}(\omega,t)  
$  
weak star for $\omega \notin N$, see Lemma \ref{lem IMPR}.
This confirms the relation between the global covariation and the 
 weak star convergence in the space $(B_{1}\hat{\otimes}_{\pi} B_{2})^{\ast\ast}$ 
as anticipated in the Introduction.
\item We recall that $J(B_{1}\hat{\otimes}_{\pi} B_{2})$ is isometrically embedded (and weak star dense) in 
$(B_{1}\hat{\otimes}_{\pi} B_{2})^{\ast\ast}$. In particular it is the case if $B_{1}$ or $B_{2}$ has infinite dimension. 
If the Banach space $B_{1}\hat{\otimes}_{\pi} B_{2}$ is not reflexive, then 
 $(B_{1}\hat{\otimes}_{\pi} B_{2})^{\ast\ast}$ strictly contains
 $B_{1}\hat{\otimes}_{\pi} B_{2}$.
The weak star convergence is weaker then the strong convergence in $J(B_{1}\hat{\otimes}_{\pi} B_{2}) $, 
%, i.e. the convergence with respect to the topology defined by the norm. 
required in the definition of the tensor quadratic variation, see Definition \ref{dfn REALTENSOQV} item 2. 
The global covariation is therefore truly more general than the tensor covariation.
%In a finite dimensional spaces all topologies are equivalent.
\item In general $B_{1}\hat{\otimes}_{\pi} B_{2}$ is not reflexive even if $B_{1}$ and $B_{2}$ are Hilbert spaces, see for instance \cite{rr} at Section 4.2. 
%see Remark \ref{pr piTPNOTREFL}.3.
\end{enumerate}
\end{rem}
We go on with some related results about the $\chi$-covariation and the $\chi$-quadratic variation.
\begin{prop} 			\label{prop somma diretta di qv}
Let $\X$ (resp. $\Y$) be a $B_{1}$-valued (resp. $B_{2}$-valued) process and
 $\chi_{1}$, $\chi_{2}$ be
two Chi-subspaces of $(B_{1}\hat{\otimes}_{\pi} B_{2})^{\ast}$ with $\chi_{1}\cap\chi_{2}=\{0\}$.
% spaces fulfilling condition \eqref{relazione norma chi cov}. 
Let $\chi=\chi_{1}\oplus \chi_{2}$. If $(\X,\Y)$ admit a $\chi_{i}$-covariation $[\X,\Y]_{i}$ for $i=1,2$ then 
they admit a $\chi$-covariation $[\X,\Y]$ and it holds
 $[\X,\Y](\phi)=[\X,\Y]_{1}(\phi_{1})+[\X,\Y]_{2}(\phi_{2})$ for all $\phi\in \chi$ with 
unique decomposition $\phi=\phi_{1}+\phi_{2}$, $\phi_{1}\in \chi_{1}$ and $\phi_{2}\in \chi_{2}$.
\end{prop}
\begin{proof}  \
$\chi$ is a Chi-subspace because of item 3 of Remark \ref{pro 3.2}.
It will be enough to show the result for a fixed norm in the space $\chi$. We set $\|\phi\|_{\chi}=\|\phi_{1}\|_{\chi_{1}}+\|\phi_{2}\|_{\chi_{2}}$ and 
we remark that $\|\phi\|_{\chi}  \geq \|\phi_{i}\|_{\chi_{i}}$, $i=1,2$. Condition 
\textbf{H1} follows immediately by inequality
\[
\begin{split}
\int_{0}^{T} \sup_{\|\phi\|_{\chi}\leq 1 } \left| \prescript{}{\chi}{\langle} \phi, (\X_{s+\epsilon}-\X_{s})\otimes  (\Y_{s+\epsilon}-\Y_{s}) \rangle_{\chi^{\ast}} \right| ds & \leq
\int_0^{T}\sup_{\|\phi_{1}\| _{\chi_{1}}\leq 1}  \left| \prescript{}{\chi_{1}}{\langle} \phi_{1}, (\X_{s+\epsilon}-\X_{s})\otimes  (\Y_{s+\epsilon}-\Y_{s}) \rangle_{\chi_{1}^{\ast}} \right| ds +\\
&+
\int_0^{T}\sup_{\|\phi_{2}\| _{\chi_{2}}\leq 1}  \left| \prescript{}{\chi_{2}}{\langle }\phi_{2}, (\X_{s+\epsilon}-\X_{s})\otimes  (\Y_{s+\epsilon}-\Y_{s})\rangle_{\chi_{2}^{\ast}}  \right| ds \; .
\end{split}
\]
Condition \textbf{H2}(i) follows by linearity; in fact
\begin{displaymath}
\begin{split}
[\X,\Y]^{\epsilon}(\phi)& =\int_{0}^{t}\prescript{}{\chi}{\langle}  \phi_{1}+\phi_{2}, (\X_{s+\epsilon}-\X_{s})\otimes  (\Y_{s+\epsilon}-\Y_{s})  \rangle_{\chi^{\ast}} ds=\\
&=
\int_{0}^{t}\prescript{}{\chi_{1}}{\langle} \phi_{1}, (\X_{s+\epsilon}-\X_{s})\otimes  (\Y_{s+\epsilon}-\Y_{s})  \rangle_{\chi_{1}^{\ast}} ds
+\int_{0}^{t}\prescript{}{\chi_{2}}{\langle}  \phi_{2}, (\X_{s+\epsilon}-\X_{s})\otimes (\Y_{s+\epsilon}-\Y_{s}) \rangle_{\chi_{2}^{\ast}} ds\\
&
\xrightarrow [\epsilon\rightarrow 0]{ucp}   [\X,\Y]_{1}(\phi_{1})+[\X,\Y]_{2}(\phi_{2}) \; .
\end{split}
\end{displaymath}
Concerning Condition {\bf H2}(ii), for $\omega\in \Omega$, $t\in [0,T]$ we can obviously set 
$\widetilde{[\X,\Y]}(\omega,t)(\phi)=\widetilde{[\X,\Y]_{1}}(\omega,t)(\phi_{1})+\widetilde{[\X,\Y]_{2}}(\omega,t)(\phi_{2})$.
\end{proof}
\begin{prop}		\label{prop XQVSM}
Let $\X$ (resp. $\Y$) be a $B_{1}$-valued (resp. $B_{2}$-valued) stochastic process.
\begin{enumerate}
\item
Let $\chi_{1}$ and $\chi_{2}$ be 
two subspaces $\chi_{1}\subset \chi_{2}\subset (B_{1}\hat{\otimes}_{\pi}B_{2})^{\ast}$, $\chi_{1}$ being 
a Banach subspace continuously embedded into $\chi_{2}$ and $\chi_{2}$ a Chi-subspace.
If $(\X,\Y)$ admit a $\chi_{2}$-covariation $[\X,\Y]_{2}$, 
then they also admit a $\chi_{1}$-covariation $[\X,\Y]_{1}$ and it holds 
$[\X,\Y]_{1}(\phi)=[\X,\Y]_{2}(\phi)$ for all $\phi\in \chi_{1}$.
\item In particular if $(\X,\Y)$ admit a tensor quadratic variation, then $\X$ and $\Y$ admit a $\chi$-quadratic variation for any Chi-subspace $\chi$.
\end{enumerate}
\end{prop}
\begin{proof}
\begin{enumerate}
\item
If Condition \textbf{H1} is valid for $\chi_{2}$ then it is also verified  for $\chi_{1}$. 
In fact we remark that $(\X_{s+\epsilon}-\X_{s})\otimes (\Y_{s+\epsilon}-\Y_{s})$ is an element in $(B_{1}\hat{\otimes}_{\pi} B_{2})\subset 
(B_{1}\hat{\otimes}_{\pi} B_{2})^{\ast\ast}\subset \chi_{2}^{\ast}\subset\chi_{1}^{\ast}$. 
If $A:=\left\{  \phi  \in \chi_{1} \; ; \|\phi\|_{\chi_{1}\leq 1}  \right\}$ and $B:=\left\{  \phi  \in \chi_{2} \; ; \|\phi\|_{\chi_{2} \leq 1}  \right\}$, then $A\subset  B$ and clearly 
$\int_{0}^{t}\sup_{\phi\in A}|\langle\phi, (\X_{s+\epsilon}-\X_{s})\otimes (\Y_{s+\epsilon}-\Y_{s})\rangle |ds 
\leq \int_{0}^{t} \sup_{\phi\in B}|\langle\phi, (\X_{s+\epsilon}-\X_{s})\otimes (\Y_{s+\epsilon}-\Y_{s}) \rangle |ds$. This implies the inequality 
$\left\| (\X_{s+\epsilon}-\X_{s})\otimes (\Y_{s+\epsilon}-\Y_{s}) \right\|_{\chi_{1}^{\ast}}\leq \left\| (\X_{s+\epsilon}-\X_{s})\otimes (\Y_{s+\epsilon}-\Y_{s})\right\|_{\chi_{2}^{\ast}}$ and 
Assumption \textbf{H1} follows immediately. Assumption {\bf H2}(i) is trivially verified because, by restriction, we have 
$[\X,\Y]^{\epsilon}(\phi)\xrightarrow[\epsilon \rightarrow 0] {ucp} [\X,\Y]_{2}(\phi)$ for all $\phi\in \chi_{1}$. 
We define $[\X,\Y]_{1}(\phi)=[\X,\Y]_{2}(\phi)$, $\forall \; \phi\in \chi_{1}$ and $\widetilde{[\X,\Y]_{1}}(\omega, t)(\phi)=\widetilde{[\X,\Y]_{2}}(\omega, t)(\phi)$, for all 
$\omega\in \Omega$, $t\in [0,T]$, $\phi\in \chi_{1}$. Condition \textbf{H2}(ii) follows because given 
$G:[0,T]\longrightarrow \chi_{1}$ we have $\|G(t)-G(s)\|_{\chi_{1}^{\ast}}\leq \|G(t)-G(s)\|_{\chi_{2}^{\ast}}$, $\forall\; 0\leq s \leq t\leq T$.
%Moreover 
%$\widetilde{[X,Y]_{1}}(t)(\phi)=[X,Y]_{1}(\phi)(t)=\widetilde{[X,Y]_{2}}(t)(\phi)$ and $\| \widetilde{[X,Y]_{1}}\|_{Var[0,T]}\leq \| \widetilde{[X]_{2}}\|_{Var[0,T]} $. 
%So that also point $(ii)$ of condition {\bf H2} is established. We conclude that 
%$X$ and $Y$ admit $\chi_{1}$-quadratic variation and it holds $[X,Y]_{1}(\phi)=[X,Y]_{2}(\phi)$ for all $\phi\in \chi_{1}$.
\item It follows from 1. and Proposition \ref{pr TGH78}.
\end{enumerate} 
\end{proof}

%\begin{rem} 
%Let $\X$ be a $B$-valued process and $\chi_{1}, \chi_{2}$ be two 
%subspaces as in Proposition \ref{prop XQVSM}. 
%It may happen that $\X$ does not admit 
%a global quadratic variation or not even a $\chi_{2}$-quadratic variation but it admits a $\chi_{1}$-quadratic variation. 
%For this reason the fact to introduce a Chi-subspace gives much more opportunities of calculus. 
%For example that the window Brownian motion admits a 
%$\chi^{2}$-quadratic variation but it does not have a
%global quadratic variation. See Proposition \ref{pr WBMNOQV}.
% \end{rem}
%
%
We continue with some general properties of the $\chi$-covariation.
\begin{lem}   \label{lem CONV ZERO PROB}
Let $\X$ (resp. $\Y$) be a $B_{1}$-valued (resp. $B_{2}$-valued) 
stochastic process and $\chi$ be a Chi-subspace. 
Suppose that $\frac{1}{\epsilon}\int_{0}^{T}\| (\X_{s+\epsilon}-\X_{s})\otimes (\Y_{s+\epsilon}-\Y_{s}) \|_{\chi^{\ast}}  \,ds$
converges to $0$ in probability when $\epsilon$ goes to zero. 
\begin{enumerate}
\item 
Then $(\X,\Y)$ admits a zero $\chi$-covariation.
\item 
If $\chi=(B_{1}\hat{\otimes}_{\pi}B_{2})^{\ast}$, then $(\X,\Y)$ admits
a zero scalar and tensor covariation.
% and $(\X,\Y)$ admits a tensor covariation.
\end{enumerate}
\end{lem}
\begin{proof} \
Concerning item 1 Condition \textbf{H1} is verified because of Remark 
\ref{rem CSH1} item 1. We verify {\bf H2}(i) directly. For every fixed $\phi\in \chi$ we have 
\[
\begin{split}
\left| [\X,\Y]^{\epsilon}(\phi)(t)\right| &
= \left| \int_{0}^{t} \prescript{}{\chi}{\langle} \phi,\frac{(\X_{s+\epsilon}-\X_{s})\otimes (\Y_{s+\epsilon}-\Y_{s})}{\epsilon} \rangle_{\chi^{\ast}}  \, ds\right|  \leq %\\
%&\leq
% \int_{0}^{t}\left|\prescript{}{\chi}{\langle} \phi,\frac{(\X_{s+\epsilon}-\X_{s})\otimes(\Y_{s+\epsilon}-\Y_{s})}{\epsilon} \rangle_{\chi^{\ast}} \right| \, ds \leq\\
 %&\leq
   \int_{0}^{T}\left| \prescript{}{\chi}{\langle} \phi,\frac{(\X_{s+\epsilon}-\X_{s})\otimes(\Y_{s+\epsilon}-\Y_{s})}{\epsilon} \rangle_{\chi^{\ast}} \right| \, ds.
\end{split}
\]
So we obtain
\[
\begin{split} 
\sup_{t\in [0,T]}\left| [\X,\Y]^{\epsilon}(\phi)(t)\right| 
%& \leq \int_{0}^{T}\left| \prescript{}{\chi}{\langle} \phi,\frac{(\X_{s+\epsilon}-\X_{s})\otimes (\Y_{s+\epsilon}-\Y_{s})}{\epsilon} \rangle_{\chi^{\ast}} \right| \, ds \\
&
\leq \|\phi\|_{\chi} \frac{1}{\epsilon} \int_{0}^{T}\| (\X_{s+\epsilon}-\X_{s})\otimes (\Y_{s+\epsilon}-\Y_{s}) \|_{\chi^{\ast}} \, ds \xrightarrow [\epsilon\longrightarrow 0]{}0
\end{split}
\] 
in probability by the hypothesis. Since condition  {\bf H2}(ii) holds  trivially, 
we can conclude for the first result. 
Concerning item 2. the scalar covariation vanishes by hypothesis, which
 also forces the tensor covariation to be zero, see Remark \ref{rem RE1}, item 4.
\end{proof}
\subsection{Technical issues}			\label{sec teRes}
\subsubsection{Convergence of infinite dimensional Stieltjes integrals}
We state now an important technical result which will be used in the proof of the It\^o formula appearing in Theorem \ref{thm ITONOM}. %and also in \cite{DGR2}.

\begin{prop}		\label{pr IMPR} 
Let $\chi$ be a separable Banach space, a sequence
 $F^{n}:\chi\longrightarrow \mathscr{C}([0,T])$
 of
linear continuous maps and  measurable random fields
$\widetilde{F}^{n}:\Omega\times [0,T]\longrightarrow \chi^{\ast}
$  such that $\widetilde{F}^{n}(\cdot,t)(\phi)=F^{n}(\phi)(\cdot,t)$ a.s. $\forall\; t\in [0,T]$, $\phi\in \chi$. 
%Moreover we assume that $t\mapsto \widetilde{F}^n(\cdot,t)$ is a. s. of bounded variation for every $n$. 
We suppose the following.
\begin{description}
\item [i)] For every $n$, $t\mapsto \widetilde{F}^n(\cdot,t)$ is a. s.
 of bounded variation  and
for all $(n_{k})$ there is a subsequence $(n_{k_{j}})$ such that
 $\sup_{j}\|\widetilde{F}^{n_{k_{j}}}\|_{Var([0,T])}<\infty$ a.s.
\item [ii)] There is a linear continuous map $F:\chi\longrightarrow
\mathscr{C}([0,T])$ such that for all $t\in[0,T]$ and for every
$\phi\in\chi$ $F^{n}(\phi)(\cdot,t)\longrightarrow
F(\phi)(\cdot,t)$ in probability.
\item [iii)] There is measurable random field $\widetilde{F}:\Omega \times [0,T]\longrightarrow
\chi^{\ast}$ of such that for $\omega$ a.s.  
$\widetilde{F}(\omega, \cdot): [0,T]\longrightarrow
\chi^{\ast}$ has
bounded variation and
$\widetilde{F}(\cdot,t)(\phi)=F(\phi)(\cdot,t)$a.s. $\forall
 \, t\in [0,T]$ and $\phi\in \chi$.
\item [iv)] $F^{n}(\phi)(0)=0$ for every $\phi\in \chi$.
\end{description}
Then for every $t\in [0,T]$ and every continuous process $H:\Omega \times [0,T]  \longrightarrow \chi$
\begin{equation}
\int_{0}^{t} \prescript{}{\chi}{\langle} H(\cdot, s),d\widetilde{F}^{n}(\cdot,s) \rangle_{\chi^{\ast}} \longrightarrow
\int_{0}^{t} \prescript{}{\chi}{\langle} H(\cdot, s),d\widetilde{F}(\cdot,s) \rangle_{\chi^{\ast}} \quad \quad \textrm{ in probability.}
\end{equation}
\end{prop}
\begin{proof} See Appendix \ref{app proof}.
\end{proof}
\begin{cor}       \label{pr CONVCCOV}
Let $B_{1}$, $B_{2}$ be two Banach spaces and $\chi$ be a Chi-subspace of 
$(B_{1}\hat{\otimes}_{\pi} B_{2})^{\ast}$. 
Let $\X$ and $\Y$ be two stochastic processes with values 
respectively in $B_{1}$ and $B_{2}$ such that $(\X,\Y)$ admits   
 a $\chi$-covariation 
and $\mathbb{H}$) be a continuous measurable process $\mathbb{H}:\Omega\times [0,T] \longrightarrow \mathcal{V}$ where $\mathcal{V}$ is a 
closed separable subspace of $\chi$. 
Then, for every $t\in [0,T]$,
\begin{equation}   		\label{eq SDFR}  
\int_{0}^{t} \prescript{}{\chi}{\langle} \mathbb{H}(\cdot,s),d\widetilde{[\X,\Y]}^{\epsilon}(\cdot,s)\rangle_{\chi^{\ast}}
\xrightarrow[\epsilon \longrightarrow 0]{} \int_{0}^{t} \prescript{}{\chi}{\langle} \mathbb{H}(\cdot,s),d\widetilde{[\X,\Y]}(\cdot,s)\rangle_{\chi^{\ast}} \quad \mbox{in probability.}
\end{equation}
\end{cor}

\begin{proof} \
By item 2 in Remark \ref{pro 3.2}, $\mathcal{V}$ is a Chi-subspace. 
%of  $(B_{1}\hat{\otimes}_{\pi}B_{2})^{\ast}$.
By Proposition \ref{prop XQVSM}, $(\X,\Y)$ admits a $\mathcal{V}$-covariation $[\X,\Y]_{\mathcal{V}}$ 
and $[\X,\Y]_{\mathcal{V}}(\phi)=[\X,\Y](\phi)$ for all 
$\phi\in \mathcal{V}$; in the sequel of the proof, 
$[\X,\Y]_{\mathcal{V}}$ will be still denoted by 
$[\X,\Y]$. 
%The pairing between $\chi$ and $\chi^{\ast}$ being
% compatible with the pairing 
%duality between $\mathcal{V}$ and $\mathcal{V}^{\ast}$, 
%for the left-hand side of \eqref{eq SDFR} holds the following equality.
%\begin{equation}		\label{eq SDFR1}
%\int_{0}^{t} \prescript{}{\chi}{\langle} H(\cdot,s),d\widetilde{[X,Y]}^{\epsilon}(\cdot,s)\rangle_{\chi^{\ast}}=
% \int_{0}^{t} \prescript{}{\mathcal{V}}{\langle} H(\cdot,s),d\widetilde{[X,Y]}^{\epsilon}(\cdot,s)\rangle_{\mathcal{V}^{\ast}}		\; .
% \end{equation}
Since the ucp convergence implies  the convergence in probability for every $t\in[0,T]$, by Proposition \ref{pr IMPR} and definition of $\mathcal{V}$-covariation,
it follows 
\begin{equation}		\label{eq SDFR2}
\int_{0}^{t} \prescript{}{\mathcal{V}}{\langle} \mathbb{H}(\cdot,s),d\widetilde{[\X,\Y]}^{\epsilon}(\cdot,s)\rangle_{\mathcal{V}^{\ast}}
\xrightarrow[\epsilon\longrightarrow 0]{\mathbb{P}} \int_{0}^{t} \prescript{}{\mathcal{V}}{\langle} \mathbb{H}(\cdot,s),d\widetilde{[\X,\Y]}(\cdot,s)\rangle_{\mathcal{V}^{\ast}}  \; .
\end{equation}
Since the pairing duality between $\chi$ and $\chi^{\ast}$ is
 compatible with the one between $\mathcal{V}$ and $\mathcal{V}^{\ast}$, 
%Again by item 2 in Remark \ref{pro 3.2} and by pairing compatibility, the right-hand side of \eqref{eq SDFR2} equals 
%\[
%\int_{0}^{t} \prescript{}{\chi}{\langle} H(\cdot,s),d\widetilde{[X]}(\cdot,s)\rangle_{\chi^{\ast}}  \; .
%\]
the result \eqref{eq SDFR} is now established.
\end{proof}
\subsubsection{Weaker conditions for the existence of the $\chi$-covariation}

An important and useful theorem which helps to  find sufficient conditions
 for the existence of
 the $\chi$-quadratic variation of a Banach space valued process is given below. 
It will be a consequence of
a Banach-Steinhaus type result for Fr\'echet spaces, see Theorem II.1.18, page. 55 in \cite{ds}.  
We start with a remark.
\begin{rem}		\label{rem TRE}
\item
\begin{enumerate} 
\item Let $(\Y^{n})$ be a sequence of random elements with values in a Banach space $(B,\|\cdot\|_{B})$ 
such that $\sup_{n} \,\|\Y^{n}\|_{B}\leq Z$ a.s. for some real positive random variable $Z$. 
Then $(\Y^{n})$ is {\bf{bounded}}\footnote{This notion plays a role in Banach-Steinhaus theorem in \cite{ds}.
Let $E$ be a Fr\'echet spaces, $F$-space shortly.  
A subset $C$ of $E$ is called {\bf bounded} if 
for all $\epsilon >0$ it exists $\delta_{\epsilon}$ such
that for all $0< \alpha\leq
\delta_{\epsilon}$, $\alpha C$ is included 
in the open ball $\mathcal{B}(0,\epsilon):=\{e\in E; d(0,e)< \epsilon\}$} in the $F$-space of random elements equipped with the convergence in probability 
which is governed by the metric 
$
d(\X,\Y)=\mathbb{E}\left[ \|\X-\Y\|_{B} \wedge 1 \right].
$
In fact by Lebesgue dominated convergence theorem it follows 
$\lim_{\gamma\rightarrow 0}\mathbb{E}[\gamma Z \wedge 1]=0$.
\item In particular taking $B=C([0,T])$ a sequence of continuous processes $(\Y^{n})$ such that 
$\sup_{n}\left\| \Y^{n}\right\|_{\infty}\leq Z$ a.s. is bounded for the usual metric 
in $\mathscr{C}([0,T])$ equipped with the topology related to the ucp convergence.
\end{enumerate}
\end{rem}

\begin{thm}   \label{thm QV}
Let $F^{n}:\chi\longrightarrow \mathscr{C}([0,T])$ be a sequence of
linear continuous maps such that
 $F^{n}(\phi)(0)=0$ a.s. and 
there is $\tilde{F}^{n}:\Omega\times [0,T]\longrightarrow \chi^{\ast}$ 
having a.s. bounded variation.
We formulate the following assumptions.
\begin{description}
\item [i)] $F^{n}(\phi)(\cdot,t)=\tilde{F}^{n}(\cdot,t)(\phi)$  a.s. $\forall\; t\in [0,T]$, $\phi\in\chi$.
\item [ii)] $\forall\;\phi\in \chi $, $t\mapsto \tilde{F}^{n}(\cdot, t)(\phi)$ is  c\`adl\`ag.
\item [iii)] $\sup_{n}\|\tilde{F}^{n}\|_{Var([0,T])}< \infty \quad$ a.s.
\item [iv)] There is a subset $\shs \subset \chi$ such that
$\overline{Span(\shs)}=\chi$ and 
%a family of continuous processes $F(\phi)$, for every $\phi \in S$, (i.e it exists 
a linear application $F: \shs \longrightarrow \mathscr{C}([0,T])$ such that 
$F^{n}(\phi)\longrightarrow F(\phi)$ ucp for every
$\phi\in \shs$.
%Condition \textbf{iv)} can be replaced with the one below \textbf{iv')}.
%\item [iv')] There is a subset $\shs \subset \chi$ such that
%$\overline{Span(\shs)}=\chi$ and 
%a linear application $F: \shs \longrightarrow \mathscr{C}([0,T])$ such that 
%for every
%$\phi\in \shs$. 
%\begin{itemize}
%\item $ F^{n}(\phi)(t)\longrightarrow F(\phi)(t) $ for every $t\in
%[0,T]$ in probability.
%\item $F^{n}(\phi)$ is an increasing process.
%\end{itemize}
\end{description}
\begin{description}
\item[1)] Suppose that $\chi$ is separable.\\
Then there is a linear and continuous extension $F:\chi\longrightarrow \mathscr{C}([0,T])$ and 
there is a measurable random field $\tilde{F}:\Omega \times[0,T]\longrightarrow \chi^{\ast}$ such that $\tilde{F}(\cdot,t)(\phi)=F(\phi)(\cdot,t)$ a.s. 
for every $t\in [0,T]$. 
Moreover the following properties hold.
\begin{description}
\item [a)] For every $\phi\in \chi$, $F^{n}(\phi)\xrightarrow{ucp} F(\phi)$.\\ 
In particular for every $t\in[0,T]$, $\phi\in\chi$,
$F^{n}(\phi)(\cdot, t)\xrightarrow {\mathbb{P}}F(\phi)(\omega, t)$.
\item [b)] $\tilde{F}$ has bounded variation   %, $\|\tilde{F}\|_{Var}<\infty$ a.s.
and $t\mapsto \tilde{F}(\cdot,t)$ is 
%$\omega$- 
weakly star continuous a.s.
\end{description}
\item [2)]Suppose the existence of a measurable $\tilde{F}:\Omega\times [0,T]\longrightarrow \chi^{\ast}$ such that a.s. 
$t\mapsto \tilde{F}(\cdot,t)$ has bounded variation and is 
weakly star  c\`adl\`ag such that 
\[
\tilde{F}(\cdot,t)(\phi)=F(\phi)(\cdot,t)  \hspace{0.5cm} \textrm{ a.s. }  \hspace{0.5cm}  \forall\; t\in [0,T],\;  \forall\; \phi\in \mathcal{S}\; .
\]
Then point \textbf{a)} still follows.
\end{description}
\end{thm}
\begin{rem}
In point \textbf{2)} we do not necessarily suppose $\chi$ to be separable.
\end{rem}
\begin{proof}See Appendix \ref{app proof}.
\end{proof}
Important implications of Theorem \ref{thm QV} are Corollaries \ref{cor SGS} and \ref{corollary}, which give us easier conditions 
for the existence of the 
$\chi$-covariation as anticipated in Remark \ref{rem PPO} item 4.
\begin{cor}	\label{cor SGS}
Let $B_{1}$ and $B_{2}$ be Banach spaces, $\X$ (resp. $\Y$) be a $B_{1}$-valued (resp. $B_{2}$-valued) stochastic process 
and $\chi$ be a separable Chi-subspace of $(B_{1}\hat{\otimes}_{\pi}B_{2})^{\ast}$. 
We suppose the following.
\begin{description}
\item [H0'] There is $\mathcal{S}\subset \chi$ such that $\overline{Span({\cal S})}=\chi$. 
\item [H1] For every sequence $(\epsilon_{n})\downarrow 0$ there is a
subsequence $(\epsilon_{n_{k}})$ such that 
\[
\sup_{k}\int_{0}^{T}\sup_{\|\phi\|_{\chi}\leq 1}\left|\prescript{}{\chi}{\langle} \phi,\frac{(\X_{s+\epsilon_{n_{k}}}-\X_{s})
\otimes (\Y_{s+\epsilon_{n_{k}}}-\Y_{s})}{\epsilon_{n_{k}}}\rangle_{\chi^{\ast}} \right|ds\quad < +\infty  \; .
\]
\item [H2'] There is $\mathcal{T}:\chi\longrightarrow \mathscr{C}([0,T])$ such that 
$[\X,\Y]^{\epsilon}(\phi)(t)\rightarrow \mathcal{T}(\phi)(t)$ ucp for all $\phi\in \mathcal{S}$.
\end{description}
Then $(\X,\Y)$ admits a $\chi$-covariation and the
 application $[\X,\Y]$ is equal to $\mathcal{T}$.
\end{cor}
\begin{proof}
Condition {\bf H1} is verified by assumption. Conditions {\bf H2}(i) 
and (ii) follow by Theorem \ref{thm QV}
setting $ F^n(\phi)(\cdot,t) = [\X,\Y]^{\epsilon_n}(\phi)(t)$ and
$\tilde {F}^n = \widetilde{[\X,\Y]}^{\epsilon_n}$
for a suitable sequence $(\epsilon_n)$.
\end{proof}
In the case $\X=\Y$ and $B=B_{1}=B_{2}$ we can further relax the hypotheses.
\begin{cor}		 \label{corollary} 
Let $B$ be a Banach space, $\X$ a be $B$-valued stochastic processes and $\chi$ 
be a separable Chi-subspace. % of $(B\hat{\otimes}_{\pi}B)^{\ast}$. 
We suppose the following.
\begin{description}
\item [H0''] There are subsets $\mathcal{S}$, $\mathcal{S}^{p}$ of $\chi$ such that 
$\overline{Span({\cal S})}=\chi$, $Span({\cal S})=Span({\cal S}^{p})$ and 
$\mathcal{S}^{p}$ is constituted by \textbf{positive definite} elements $\phi$ 
in the sense that 
$\langle \phi, b \otimes b\rangle \geq 0$ for all $b \in B$. 
\item [H1] For every sequence $(\epsilon_{n})\downarrow 0$ there is a
subsequence $(\epsilon_{n_{k}})$ such that 
\[
\sup_{k}\int_{0}^{T}\sup_{\|\phi\|_{\chi}\leq 1}\left|\prescript{}{\chi}{\langle} \phi,\frac{(\X_{s+\epsilon_{n_{k}}}-\X_{s})
\otimes^{2}}{\epsilon_{n_{k}}}\rangle_{\chi^{\ast}} \right|ds\quad < +\infty  \; .
\]
\item [H2''] There is $\mathcal{T}:\chi\longrightarrow \mathscr{C}([0,T])$ such that 
$[\X]^{\epsilon}(\phi)(t)\rightarrow \mathcal{T}(\phi)(t)$ in probability for every $\phi\in \mathcal{S}$ and for every $t\in [0,T]$.
\end{description}
Then $\X$ admits a $\chi$-quadratic variation and application $[\X]$ is equal to $\mathcal{T}$.
\end{cor}

\begin{proof} \
We verify the conditions of Corollary \ref{cor SGS}. Conditions {\bf H0'} and {\bf H1} are verified by assumption. 
We observe that, for every $\phi\in \shs^{p}$, $[\X]^{\epsilon}(\phi)$ is an increasing process. 
By linearity, it follows that for any $\phi\in \shs^{p}$, $[\X]^{\epsilon}(\phi)(t)$ converges in probability to $\mathcal{T}(\phi)(t)$ for any
 $t \in [0,T]$.
Lemma 3.1 in  \cite{rv4} implies that $[\X]^{\epsilon}(\phi)$ converges ucp for every $\phi\in \shs^{p}$ and therefore in $\shs$. 
Conditions {\bf H2'} of Corollary \ref{cor SGS} is now verified.
%
%We go on with the second condition $[H2]$.
%We apply Theorem \ref{thm QV} for the applications $[X,X]^{\epsilon}$ then
%we have the existence of an application
%$[X]:\chi\longrightarrow\mathscr{C}([0,T])$ such that for every $(\epsilon_{n})\downarrow 0$ there is a
%subsequence $(\epsilon_{n_{k}})$ such that
%$[X](\phi)=\lim^{ucp}[X,X]^{\epsilon_{n_{k}}}$. This is
%equivalent to the convergence ucp of $[X]^{\epsilon}$ car
%$\mathscr{C}([0,T])$ with the ucp topology is a metric space.
\end{proof}
When $\chi$ is finite dimensional the notion of $\chi$-quadratic variation becomes very natural.
\begin{prop}	\label{pr FinDim}
Let $\chi=Span\{ \phi_{1},\ldots, \phi_{n}\}$, $\phi_{1},\ldots, \phi_{n}\in (B\hat{\otimes}_{\pi}B)^{\ast}$ of positive type and linearly independent. 
$\X$ has a $\chi$-quadratic variation if and only if  there are continuous processes $Z^{i}$ such that
$
[\X]^{\epsilon}_{t}(\phi_{i})
$
converges in probability to $Z^{i}_{t}$ for $\epsilon$ going to zero for all $t\in [0,T]$ and $i=1,\ldots, n$. 
\end{prop}
\begin{proof} 
We only need to show that the condition is sufficient, the converse implication resulting immediately. 
We verify the hypotheses of Corollary \ref{corollary} taking $\mathcal{S}=\{\phi_{1}, \ldots, \phi_{n} \}$. 
Without restriction to generality we can suppose $\|\phi_{i}\|_{(B\hat{\otimes}_{\pi}B)^{\ast}}=1$, for $1\leq i\leq n$.
Conditions \textbf{H0''} and \textbf{H2''} are straightforward. 
It remains to verify \textbf{H1}.
Since $\chi$ is finite dimensional it can be equipped with the norm 
$\| \phi\|_{\chi} = \sum_{i=1}^{n}|a_{i}|$ if 
$\phi=\sum_{i=1}^{n}a_{i}\, \phi_{i}$ with $a_{i}\in \R$.
For $\phi$ such that $\|\phi\|_{\chi}=\sum_{i=1}^{n}|a_{i}| \leq 1$ we have
\begin{equation}
\begin{split}
\frac{1}{\epsilon}\int_{0}^{T}\left| \langle \phi \, ,\, \X_{s+\epsilon}-\X_{s})\otimes^{2} \rangle \right| ds 
&
\leq \sum_{i=1}^{n} \frac{1}{\epsilon}\int_{0}^{T} \left| \langle a_{i}\, \phi_{i} \, ,\, (\X_{s+\epsilon}-\X_{s})\otimes^{2} \rangle   \right| ds =
\sum_{i=1}^{n} \frac{ |a_{i}|}{\epsilon}\int_{0}^{T}   \langle \phi_{i} \, ,\, (\X_{s+\epsilon}-\X_{s})\otimes^{2} \rangle  ds,
\end{split}
\end{equation}
because $ \phi_{i}$ are of positive type. Previous expression is smaller or equal than 
\[
\sum_{i=1}^{n}\frac{ 1}{\epsilon}\int_{0}^{T} \langle \phi_{i}\, , \, (X_{s+\epsilon}-X_{s})\otimes^{2} \rangle=\sum_{i=1}^{n}\; [\X]^{\epsilon}_{T}(\phi_{i})
\]
because $|a_{i}|\leq 1$ for $1\leq i\leq n$. Taking the supremum over $\|\phi\|_{\chi}\leq 1$ and using the hypothesis of convergence in probability of the quantity 
$[\X]^{\epsilon}_{T}(\phi_{i})$ for $1\leq i\leq n$, the result follows.
\end{proof}
\begin{cor}  \label{cor finD}
Let $B_{1}=B_{2}=\R^{n}$. $\X$ admits all its mutual brackets if and only if $\X$ admits a global quadratic variation.
\end{cor}
%
%
%
%\section{Evaluations of $\chi$-covariations for window processes}
\section{Calculations related to  window processes}		\label{sec: evaluation}
In this section we consider $X$ and $Y$ as real continuous processes as usual
prolonged by continuity and $X(\cdot)$ and $Y(\cdot)$ their associated window processes. 
We set $B=C([-\tau,0])$. We will proceed to the evaluation 
of some $\chi$-covariations 
(resp. $\chi$-quadratic variations) 
for window processes $X(\cdot)$ and $Y(\cdot)$ 
(resp. for process $X(\cdot)$) 
with values in $B=C([-\tau,0])$.
%Spaces $\chi$ will be a Chi-subspaces of $(B\hat{\otimes}_{\pi}B)^{\ast}$, as listed in Example \ref{ex 4.5}.
We start with some examples of $\chi$-covariation calculated directly through the definition. 
\begin{prop}		\label{prop ZQVHC} 
Let $X$ and $Y$ be two real valued processes with H\"older
continuous paths of parameters $\gamma$ and $\delta$ 
such that $\gamma+\delta>1$. 
Then $(X(\cdot),Y(\cdot))$ admits a zero scalar and tensor covariation.
In particular $(X(\cdot), Y(\cdot))$ admit a 
zero global covariation.
\end{prop}

\begin{proof} \
By Remark \ref{rem RE1} item 4 and Proposition \ref{pr TGH78} we only need to show that $(X(\cdot),Y(\cdot))$ admit a zero scalar covariation, i.e. 
the convergence to zero in probability of following quantity.
\small{
\begin{equation}		\label{eq ZGQV2}
\frac{1}{\epsilon}  \int_{0}^{T}\|X_{s+\epsilon}(\cdot)-X_{s}(\cdot)\|_{B}\|Y_{s+\epsilon}(\cdot)-Y_{s}(\cdot)\|_{B}\,ds=
\frac{1}{\epsilon}\int_{0}^{T}\sup_{u\in[-\tau,0]}\left|X_{s+u+\epsilon}-X_{s+u}\right|\sup_{v\in[-\tau,0]}\left|Y_{s+v+\epsilon}-Y_{s+v}\right|ds\; .
%\leq\\
%&\leq 
%\frac{1}{\epsilon}\int_{0}^{T}\epsilon^{2\gamma}Y\,ds=\epsilon^{2\gamma-1}\,Y\,T =:Z(\epsilon)\\
%\longrightarrow 0 \textrm{ for }\gamma>\frac{1}{2}\\
\end{equation}}
Since $X$ (resp. $Y$) is a.s. $\gamma$-H\"older continuous (resp. $\delta$-H\"older continuous), 
there is a non-negative finite random variable $Z$ such that the right-hand side of \eqref{eq ZGQV2} is bounded by a sequence of random variables $Z(\epsilon)$ defined by 
$Z(\epsilon):=\epsilon^{\gamma+\delta-1}\,Z\,T$. This implies that \eqref{eq ZGQV2} 
converges to zero a.s. for $\gamma+\delta>1$. 
\end{proof}
\begin{rem} 
As a consequence of previous proposition every window process $X(\cdot)$ associated with a continuous process with 
H\"older continuous paths of parameter $\gamma>1/2$ admits zero real, tensor and global quadratic variation.
%By Proposition \ref{prop XQVSM} $X(\cdot)$ admits also zero $\chi$-quadratic variation for every Chi-subspace $\chi$.
\end{rem}
\begin{rem}
Let $B^{H}$ (resp. $B^{H,K}$) be a real fractional Brownian motion with parameters $H\in ]0,1[$ (resp. real 
bifractional Brownian motion with parameters $H\in ]0,1[$, $K\in ]0,1]$), 
see \cite{rtudor} and \cite{hv} for elementary facts about the bifractional Brownian motion. 
As immediate consequences
 of Proposition \ref{prop ZQVHC} we obtain the following results. 
1) The fractional window Brownian motion $B^{H}(\cdot)$ with $H>1/2$ admits a zero scalar, tensor and global quadratic variation. 
2) The bifractional window Brownian motion $B^{H,K}(\cdot)$ with $KH>1/2$ admits a zero scalar, tensor and global quadratic variation.
3) We recall that the paths of a Brownian motion $W$ are a priori only a.s. H\"older
continuous of parameter $\gamma < 1/2$ so that we can not use Proposition \ref{prop ZQVHC}.
\end{rem}
Propositions \ref{prop WBMNS} and \ref{pr WBMNOQV} show that the
 stochastic calculus developed by \cite{dpz}, 
\cite{dincuvisi} and \cite{mp} cannot be applied for $\X$ being
 a window Brownian motion $W(\cdot)$.
\begin{dfn}
Let $B$ be a Banach space and $\X$ be a $B$-valued stochastic process. 
We say that $\X$ is a \textbf{Pettis semimartingale} if, 
for every $\phi \in B^{\ast}$, $\langle \phi, \X_{t}\rangle$ is 
a real semimartingale.
%with respect to a filtration $(\mathcal{G}_{t})$. 
\end{dfn}
We remark that if $\X$ is a $B$-valued semimartingale in the sense of Section 1.17, \cite{mp}, then it is also a Pettis semimartingale.
%
%We will show that the window Brownian motion is not even a weak semimartingale, then is not a martingale and we can not define a stochastic integral with the classical method for integration with respect to Banach valued semimartingale.
\begin{prop} \label{prop WBMNS}
The $C([-\tau,0])$-valued window Brownian $W(\cdot)$ motion is not a Pettis semimartingale.
\end{prop}

\begin{proof} \
%Let $(\mathcal{F}_{t})$ be the natural filtration generated by the real Brownian motion $W$. 
It is enough to show that the existence of      an element $\mu$ in $B^{\ast}=\mathcal{M}([-\tau,0])$ such that 
$\langle \mu, W_{t}(\cdot)\rangle=\int_{[-\tau,0]}W_{t}(x)\mu(dx)$ is not a semimartingale with respect to any filtration. We will
 proceed by contradiction: we suppose that $W(\cdot)$ is a 
Pettis semimartingale, so that in particular if we take $\mu=\delta_{0}+
\delta_{-\tau}$, the process $\langle \delta_{0}+\delta_{-\tau}, 
W_{t}(\cdot)\rangle= W_{t}+W_{t-\tau}:=X_{t}$ is a 
semimartingale with respect
 to some filtration $(\mathcal{G}_{t})$. 
Let $(\mathcal{F}_{t})$ be the natural filtration generated by the real Brownian motion $W$.
Now $W_{t}+W_{t-\tau}$ is $(\mathcal{F}_{t})$-adapted, so by Stricker's theorem (see 
Theorem 4, page. 53 in \cite{prot2}), $X$ is a semimartingale with respect to filtration $(\mathcal{F}_{t})$. 
We recall that a  $(\mathcal{F}_{t})$-\textbf{weak Dirichlet} is
% process with respect to some fixed filtration $(\mathcal{G}_{t})$ is 
the sum of a local martingale $M$ and a process $A$ which is adapted and
 $[A,N]=0$ for any continuous $(\mathcal{F}_{t})$-local martingale $N$; 
$A$ is called the $(\shf_{t})$-martingale orthogonal process. 
On the other hand $(W_{t-\tau})_{t\geq \tau}$ is a strongly predictable process with respect to $(\mathcal{F}_{t})$, see Definition 3.5 in \cite{crwdp}. 
By Proposition 4.11 in \cite{crnsm2}, it follows that 
%we have $[W_{\cdot-\tau},N]=0$ for every continuous $\mathcal{F}_{t}$-local martigale $N$, so 
$(W_{t-\tau})_{t\geq \tau}$  is an $(\shf_{t})$-martingale orthogonal process.
Since $W$ is an $(\mathcal{F}_{t})$-martingale, the process $X_{t}=W_{t}+W_{t-\tau}$ is an $(\mathcal{F}_{t})$-weak Dirichlet process. 
By uniqueness of the decomposition for  $(\mathcal{F}_{t})$-weak
 Dirichlet processes, 
% and for an $(\mathcal{F}_{t})$ semimartingale 
$(W_{t-\tau})_{t \ge \tau}$ has to be a bounded variation process. 
This generates a contradiction because $(W_{t-\tau})_{t\geq \tau}$ is not a
 zero quadratic variation process. 
In conclusion $\langle \mu, W_{t}(\cdot)\rangle, t \in [0,T]$ is 
not a semimartingale.
  \end{proof}

\begin{rem} \label{RWBMNS}
\begin{enumerate}
\item Process $X$ defined by $X_{t}=W_{t}+W_{t-\tau}$ is an example of
 $(\mathcal{F}_{t})$-weak Dirichlet process with finite quadratic variation which is not an $(\mathcal{F}_{t})$-Dirichlet process.
\item Let $X$ be a semimartingale  and $\mu$ be a signed Borel measure on $[-T,0]$. We define the real valued process $X^\mu$ by $X^\mu_t:= \int_{[-T,0]} X_{t+x} d\mu(x)$. 
If $ \mu(dx) = \gamma \delta_0(dx) + g(x)dx$, $\gamma \in \R$ and
$g$ being a bounded Borel function on $[-T,0]$, then $X^\mu$
is a semimartingale such that $X^\mu_t=\gamma X_t +\int_0^t 
\tilde g(y-t)dX_y$, $ t \in [0,T],$ and $\tilde g(x) = - \int_x^0 g(y)dy,
x \in [-T,0].$
\end{enumerate}
\end{rem}
\begin{prop}		\label{pr WBMNOQV}
If $W$ is a classical Brownian motion, then $W(\cdot)$ does not admit a scalar quadratic variation. In particular $W(\cdot)$ does not admit a global quadratic variation.
\end{prop}
\begin{proof}
We can prove that 
\begin{equation}			\label{eq QS32}
\int_{0}^{T} \frac{1}{\epsilon} \left\| W_{u+\epsilon}(\cdot)-W_{u}(\cdot) \right\|^{2}_{B}du \geq T \, A^{2}(\tilde{\epsilon}) \ln (1/\tilde{\epsilon}),  
\quad \textrm{where} \quad \tilde{\epsilon}=\frac{2\epsilon}{T}
\end{equation}
and $(A(\epsilon))$ is a family of non negative r.v. such that $\lim _{\epsilon\rightarrow 0}A(\epsilon)=1$ a.s. In fact the left-hand side of 
\eqref{eq QS32} gives
\[
\begin{split}
\int_{0}^{T}  \frac{1}{\epsilon} \sup_{x\in [0,u]}| W_{x+\epsilon}-W_{x} |^{2}du
&
\geq 
\int_{T/2}^{T}  \frac{1}{\epsilon} \sup_{x\in [0,u]}| W_{x+\epsilon}-W_{x} |^{2}du
%\\
%&
\geq 
\int_{T/2}^{T}  \frac{1}{\epsilon} \sup_{x\in [0,T/2-\epsilon]}| W_{x+\epsilon}-W_{x} |^{2}du
\\
&
=\frac{T}{2\epsilon}  \sup_{x\in [0,T/2-\epsilon]}| W_{x+\epsilon}-W_{x} |^{2} \; .
\end{split}
\]
Clearly we have $W_{t}=\sqrt{\frac{T}{2}} B_{\frac{2t}{T}}$ where $B$
 is another standard Brownian motion. Previous expression gives 
\[
\begin{split}
\frac{T^{2}}{4\epsilon} \sup_{x\in [0,T/2-\epsilon]}| B_{(x+\epsilon)\frac{2}{T}}-B_{\frac{2x}{T}}|^{2}
=
\frac{T^{2}}{4\epsilon} \sup_{y \in [0,1- \frac{2\epsilon}{T}]} | B_{y+\frac{2\epsilon}{T}}-B_{y}|^{2}
\end{split}
\]
We choose $\tilde{\epsilon}=\frac{2 \epsilon}{T} $. Previous expression gives 
$
T \ln(1/\tilde{\epsilon}) A^{2}(\tilde{\epsilon})
$
where 
$$
A(\epsilon)=\left( \frac{\sup_{x\in [0,1-\epsilon]}| B_{x+\epsilon}-B_{x} | }{\sqrt{2\, \epsilon \, \ln(1/\epsilon)}} \right) \; .
$$
According to Theorem 1.1 in \cite{ChenCs}, $\lim_{\epsilon\rightarrow 0} A(\epsilon)=1$ a.s. and the result is established.
\end{proof}
%
%
%\begin{rem} 			\label{rem WBMNOQV}
%In principle the window Brownian motion $W(\cdot)$ does not even admit an $\mathcal{M}([-T,0]^{2})$-quadratic variation because the 
%first condition is not verified. However we do not have a quite formal proof of this. 
%Presumably the window Brownian motion $W(\cdot)$ does not admit a global quadratic variation, even though 
% it is possible to show that 
%the expectation 
%\begin{equation}
%\lim_{\epsilon\rightarrow 0} \mathbb{E}\left[ 
%\int_{0}^{T}\frac{1}{\epsilon} \left\| W_{u+\epsilon}(\cdot)-W_{u}(\cdot) \right\|^{2}_{B}du
%\right]=+\infty.
%\end{equation}
%This is a consequence of the following result.
%\end{rem}

%\begin{prop}		\label{pr WBMNOQV}
%  
%Let $W$ be a classical Brownian motion. Let $0<\tau_{1}<\tau_{2}$, then there are positive constants $C_{1},C_{2}$ such that
%\begin{equation*}
%C_{1}\leq
% \mathbb{E}\left[ 
%\sup_{u\in [\tau_{1},\tau_{2}]} \frac{ \left| W_{u+\epsilon}-W_{u} \right|^{2}}{\epsilon \ln(1/\epsilon)}
%\right]
%\leq C_{2}
%\end{equation*}
%\end{prop}
%\begin{proof}
%See \cite{Hu}.
%\end{proof}
%
Below we will see that $W(\cdot)$, even if it does not admit a global quadratic variation, it admits a $\chi$-quadratic variation for several Chi-subspaces $\chi$. 
More generally we can state a significant existence result of  $\chi$-covariation for finite quadratic variation processes with the help 
of Corollaries \ref{cor SGS} and \ref{corollary}. 
We remind that $\shd_{a}([-\tau,0])$ and $\shd_{a,b}([-\tau,0]^{2})$ were
 defined at \eqref{eq-def Di} and \eqref{eq-def Dij}.
\begin{prop}   \label{pr QV123}
Let $X$ and $Y$ be two real continuous processes with finite quadratic variation and $0<\tau\leq T$. Let $a,b$ two given points in $[-\tau,0]$. 
The following properties hold true. 
\begin{enumerate}
\item $(X(\cdot),Y(\cdot))$ admits a zero $\chi$-covariation, where 
$\chi=L^{2}([-\tau,0]^{2})$.
\item $(X(\cdot),Y(\cdot))$ admits a zero $\chi$-covariation where
 $\chi$ equals $L^{2}([-\tau,0])\hat{\otimes}_{h}  \shd_{a}([-\tau,0])$
or  $\shd_{a}([-\tau,0])\hat{\otimes}_{h}   L^{2}([-\tau,0])$.
\end{enumerate}
If moreover the covariation $[X_{\cdot+a},Y_{\cdot+b}]$ 
exists, the following statement is valid.
\begin{enumerate}
\item [3.]
 $(X(\cdot),Y(\cdot))$ admits a $\chi$-covariation, where $\chi=\shd_{a,b}([-\tau,0]^{2})$, and it equals 
\begin{equation}		\label{eq QV Dij}
[X(\cdot), Y(\cdot)] (\mu)=\mu(\{a,b\})[X_{\cdot+a},Y_{\cdot+b}], \quad \forall \mu \in \chi.
\end{equation}
\end{enumerate}
\end{prop}

\begin{proof} \ The proof will be similar in all the three cases. 
%Concerning the second case we will only discuss
%the  set $  L^{2}([-\tau,0]) \hat{\otimes}_{h}   \shd_{a}([-\tau,0])$, 
%the other one being very close. \\
% $\shd_{a}([-\tau,0])\hat{\otimes}_{h}   L^{2}([-\tau,0])$ 
%is just a specular set.\\
As mentioned in Example \ref{ex 4.5}, all  the  involved sets 
$\chi$ are  Chi-subspaces, which moreover are separable.\\
 %of $(B\hat{\otimes}_{\pi}B)^{\ast}$.\\ 
Let $\{e_{j}\}_{j\in \mathbb{N}}$ be a topological
  basis for $L^{2}([-\tau,0])$; $\{\delta_{a}\}$ is clearly a basis for $\mathcal{D}_{a}([-\tau,0])$. 
Then $\{ e_{i}\otimes e_{j}\}_{i,j\in \mathbb{N}}$ is a basis of $L^{2}([-\tau,0]^{2})$, 
$\{e_{j}\otimes\delta_{a} \}_{j\in \mathbb{N}}$ is a basis of 
$L^{2}([-\tau,0])\hat{\otimes}_{h}  \shd_{a}([-\tau,0])$ and $\{ \delta_{a}\otimes \delta_{b}\}$ is a basis of $\shd_{a,b}([-\tau,0]^{2})$. 
The results will follow using Corollary \ref{corollary}. %First, condition {\bf H0'} is verified.
To verify Condition {\bf H1} we consider 
\[
A(\epsilon):=
\frac{1}{\epsilon}\int_{0}^{T}
\sup_{\left\| \phi\right\|_{\chi}\leq 1}\left|  \prescript{}{\chi}{\langle} \phi,\left(X_{s+\epsilon}(\cdot)-X_{s}(\cdot)\right)\otimes \left(Y_{s+\epsilon}(\cdot)-Y_{s}(\cdot)\right) \rangle_{\chi^{\ast}} \right|\,ds
\]
for all the  Chi-subspaces mentioned above. 
In all the three situations we will show the existence of a family of random variables $\{ B(\epsilon)\}$ 
converging in probability to some random variable $B$, such that 
$A(\epsilon)\leq B(\epsilon)$ a.s. By Remark \ref{rem CSH1}.1 this will imply Assumption {\bf H1}.
%, i.e. for every sequence $(\epsilon_{n})\downarrow 0$ there is a
%subsequence $(\epsilon_{n_{k}})$ such that $B(\epsilon_{n_{k}})$ is convergent almost surely, in particular such that the quantity
%$A(\epsilon_{n_{k}})$ is bounded. This means, in particular, that {\bf H1} is proved.
\begin{enumerate}
\item [1.] Suppose $\chi=L^{2}([-\tau,0]^{2})$. 
%term $A(\epsilon)$ equals
% \[
% \frac{1}{\epsilon} \int_{0}^{T}\sup_{\left\| \phi \right\|_{L^{2}([-T,0]^{2})}\leq 1}\left| \langle \phi ,\left(X_{s+\epsilon}(\cdot)-X_{s}(\cdot)\right)\otimes^{2}\rangle \right|\,ds
% \]
By Cauchy-Schwarz inequality we have
\begin{displaymath}
\begin{split}
A(\epsilon) & 
\leq  \frac{1}{\epsilon} \int_{0}^{T}\sup_{\left\| \phi\right\|_{L^{2}([-\tau,0]^{2})}\leq 1}\left\|\phi \right\|^{2}_{L^{2}([-\tau,0]^{2})}\cdot \left\|
X_{s+\epsilon}(\cdot)-X_{s}(\cdot)\right\|_{L^{2}([-\tau,0])} \cdot \left\|
Y_{s+\epsilon}(\cdot)-Y_{s}(\cdot)\right\|_{L^{2}([-\tau,0])} ds  \\
&
\leq\frac{1}{\epsilon} \int_{0}^{T} \sqrt{\int_{0}^{s} \left(X_{u+\epsilon}-X_{u}\right)^{2}\,du\,}  \sqrt{\int_{0}^{s} \left(Y_{v+\epsilon}-Y_{v}\right)^{2}\,dv\,} ds\leq T\, B (\epsilon) \quad \mbox{where}
%&\leq 
%\frac{1}{\epsilon} \int_{0}^{T}\int_{0}^{T} \left(X_{u+\epsilon}-X_{u}\right)^{2}\,du\,ds=B (\epsilon) \\
%&
%=T\, \int_{0}^{T} \frac{\left(X_{u+\epsilon}-X_{u}\right)^{2}}{\epsilon}\,du=B (\epsilon)\rightarrow
%T\,[X]_{T}
\end{split}
\end{displaymath}  
\be	\label{Beps}
B (\epsilon)=\sqrt{\int_{0}^{T} \frac{\left(X_{u+\epsilon}-X_{u}\right)^{2}}{\epsilon}\,du \, \int_{0}^{T} \frac{\left(Y_{v+\epsilon}-Y_{v}\right)^{2}}{\epsilon}\,dv}
\ee
which converges in probability to $\sqrt{[X]_{T}[Y]_{T}}$.
\item [2.]  
We proceed similarly for $\chi=L^{2}([-\tau,0])\hat{\otimes}_{h}  \shd_{a}([-\tau,0])$. 

We consider $\phi$ of the form $\phi=\tilde{\phi}\otimes \delta_{a}$, where $\tilde{\phi}$ 
is an element of $L^{2}([-\tau,0])$. We first observe
$$
\left\| \phi \right\|_{L^{2}([-\tau,0])\hat{\otimes}_{h}\mathcal{D}_{a}}=
\left\|\tilde{ \phi }\right\|_{L^{2}([-\tau,0])}\cdot \left\| \delta_{a} \right\|_{\mathcal{D}_{a}}=
\sqrt{ \int_{[-\tau,0]}\tilde{\phi }(s)^{2} \, ds}  \; .
$$
Then
\begin{displaymath}
\begin{split}
A (\epsilon)& 
=\frac{1}{\epsilon} \int_{0}^{T} \sup_{\left\| \phi \right\|_{L^{2}([-\tau,0])\hat{\otimes}_{h}\mathcal{D}_{a}}\leq 1}
\left|  \left(X_{s+\epsilon}(a)-X_{s}(a)\right)   \int_{[-\tau,0]}\left(Y_{s+\epsilon}(x)-Y_{s}(x)\right)\tilde{\phi}(x) \,dx  \right|    \,ds\leq\\
&
\leq \frac{1}{\epsilon}\int_{0}^{T} \sup_{\left\| \phi \right\|\leq 1}
  \left \{ \left(  \sqrt{ \left(X_{s+\epsilon}(a)-X_{s}(a)\right)^{2}}   \right)\cdot \right. \\
&
\hspace{4cm}
 \left. \cdot \left( \left\|\tilde{ \phi}\right\|_{L^{2}([-\tau,0])} \sqrt{\int_{[-\tau,0]}  \left(Y_{s+\epsilon}(x)-Y_{s}(x)\right)^{2} \,dx}  \right)  \right\} \, ds \leq\\
&
\leq \int_{0}^{T} \sqrt{ \frac{\left(X_{s+\epsilon}(a)-X_{s}(a)\right)^{2} }{\epsilon} }\sqrt{\int_{[-T,0]} \frac{ \left(Y_{s+\epsilon}(x)-Y_{s}(x)\right)^{2} }{\epsilon}\,dx}\;ds\leq \sqrt{T}B(\epsilon)
%\\
%&
%\leq \int_{0}^{T}\frac{(X_{y+\epsilon}-X_{y})^{2}}{\epsilon}\, dy=B(\epsilon) \longrightarrow  [X]_{T}\\
\end{split}
\end{displaymath}
where $B(\epsilon)$ is the same family of r.v. defined in \eqref{Beps}. 
% which converges in probability to $\sqrt{[X]_{T}[Y]_{T}}$. \\
The case $\shd_{a}([-\tau,0])\hat{\otimes}_{h}   L^{2}([-\tau,0])$
can be handled symmetrically.
%\[
%B(\epsilon)=\sqrt{T  \int_{0}^{T}\frac{(X_{x+\epsilon}-X_{x})^{2}}{\epsilon}\,dx \,\int_{0}^{T}\frac{(Y_{y+\epsilon}-Y_{y})^{2}}{\epsilon}\, dy } \; ,
%\]
%converges in probability to $\sqrt{T\, [X]_{T}[Y]_{T}}$ for $\epsilon\rightarrow 0$.
% 
\item [3.] The last case is $\chi=\shd_{a,b}([-\tau,0]^{2})$. 
A general element $\phi$ which belongs to $\chi$ admits a representation $\phi=\lambda \;\delta_{(a,b )}$, with 
norm equals to $\left\| \phi \right\|_{\mathcal{D}_{a,b}}=\vert\lambda\vert$. 
We have
\begin{equation}		\label{eq 5.4bis}
\begin{split}
A (\epsilon)& = \frac{1}{\epsilon} \int_{0}^{T}\sup_{\left\| \phi \right\|_{\mathcal{D}_{a,b}}\leq 1} \left| \lambda \left(X_{s+a+\epsilon}-X_{s+a}\right)
\left(Y_{s+b+\epsilon}-Y_{s+b }\right)\right|\,ds  \\ 
&
\leq \frac{1}{\epsilon}  \int_{0}^{T} \left| \left(X_{s+a+\epsilon}-X_{s+a }\right)
\left(Y_{s+b+\epsilon}-Y_{s+b }\right) \right| \,ds \;  ;
\end{split}
\end{equation}
using again Cauchy-Schwarz inequality, previous quantity is bounded by
\begin{equation}		
 \sqrt{ \int_{0}^{T}  \frac{\left(X_{s+a+\epsilon}-X_{s+a}\right)^{2} }{\epsilon}\,ds}   \sqrt{ \int_{0}^{T}\frac{\left(Y_{v+b+\epsilon}-Y_{v+b}\right)^{2}  }{\epsilon}\,dv} \leq
 B(\epsilon). 	\label{eq 5.4}
\end{equation}
%where $B(\epsilon)$ is the same family of r.v. defined in \eqref{Beps}. 
\end{enumerate}
We verify now the Conditions {\bf H0''} and {\bf H2''}.  
\begin{enumerate}
\item[1.] %We have to verify ucp convergence of $[X(\cdot)]^{\epsilon}(\phi)$ for all $\phi \in \{e_{i}\otimes e_{j}\}_{i,j\in \mathbb{N}}$. 
A general element in $\{e_{i}\otimes e_{j}\}_{i,j\in \mathbb{N}}$ is difference 
of two positive definite elements in the set $\shs^{p}=\{ e_{i}\otimes^{2}, (e_{i}+e_{j})\otimes^{2} \}_{i,j\in \mathbb{N}}$. We also define 
$\shs=\{e_{i}\otimes e_{j}\}_{i,j\in \mathbb{N}}$. The fact that $Span({\cal S})=Span({\cal S}^{p})$ implies {\bf H0''}. To conclude we need to show the validity of Condition {\bf H2''}. 
For this we have to verify 
\begin{equation}		\label{eq 1}
[X(\cdot), Y(\cdot)]^{\epsilon}(e_{i}\otimes e_{j})(t) \xrightarrow[\epsilon\longrightarrow 0]{} 0
\end{equation}
in probability for any $i,j\in \mathbb{N}$. Clearly
we can suppose $\{e_{i}\}_{i\in \mathbb{N}}\in C^{1}([-\tau,0])$. 
We fix $\omega \in \Omega$, outside some null  set, fixed but omitted. We have 
\begin{equation}   
[X(\cdot), Y(\cdot)]^{\epsilon}(e_{i}\otimes e_{j})(t)= \int_{0}^{t}
 \frac{\gamma_{j}(s,\epsilon)  \; \gamma_{i}(s,\epsilon)}{\epsilon} ds  \quad\mbox{where}
%\int_{0}^{t}\gamma(s,\epsilon)\int_{-s}^{0}e_{i}(x)\left( X_{s+x+\epsilon}-X_{s+x}\right)dx\,ds
\label{formula 1}
\end{equation}
\[
\gamma_{j}(s,\epsilon)=
\int_{(-\tau)\vee(-s)}^{0}e_{j}(y)\left( X_{s+y+\epsilon}-X_{s+y}\right)dy \quad \mbox{and}
\quad 
\gamma_{i}(s, \epsilon)=
\int_{(-\tau)\vee(-s)}^{0}e_{i}(x)\left( Y_{s+x+\epsilon}-Y_{s+x}\right)dx .
\]
Without restriction of generality, in the purpose not to overcharge notations,
 we can suppose from now on that $\tau=T$. 
For every $s\in [0,T]$, we have 
\begin{equation}		\label{eq ESTIGAMM}
\begin{split}
\left| \gamma_{j}(s,\epsilon) \right|
&=
\left| \int_{-s}^{0}\left( e_{j}(y-\epsilon)-e_{j}(y)\right)X_{s+y}dy  + \int_{0}^{\epsilon}e_{j}(y-\epsilon)X_{s+y}dy-\int_{-s}^{-s+\epsilon} e_{j}(y-\epsilon)X_{s+y}dy \right|
\\
&
\leq \epsilon \left( \int_{-T}^{0}|\dot{e_{j}}(y)| dy  +2\|e_{j}\|_{\infty} \right) \sup_{s\in[0,T]} \left| X_{s} \right|  \; .
\end{split}		
\end{equation}
For $t \in [0,T]$, this implies that 
\[
\begin{split}
\int_{0}^{t} \left|  \frac{\gamma_{j}(s,\epsilon)  \; \gamma_{i}(s,\epsilon)}{\epsilon} \right|ds
&
\leq \int_{0}^{T}  \left|  \frac{\gamma_{j}(s,\epsilon)  \; \gamma_{i}(s,\epsilon)}{\epsilon} \right|  ds\\
&
\hspace{-2.5cm}\leq T\; \epsilon  \; \left(\int_{-T}^{0}|\dot{e_{j}}(y)| dy  +2\|e_{j}\|_{\infty} \right)\left(\int_{-T}^{0}|\dot{e_{i}}(y)| dy  +2\|e_{i}\|_{\infty} \right) \left(  \sup_{s\in[0,T]} \left| X_{s} \right| \right) \left(  \sup_{u\in[0,T]} \left| Y_{u} \right| \right) 
\end{split}
\]
which trivially converges a.s. to zero when $\epsilon$ goes to zero which yields \eqref{eq 1}.
\item[2.] A generic element in $\{e_{j}\otimes \delta_a\}_{j\in \mathbb{N}}$ is difference of 
two positive definite elements of type $\{ e_{j}\otimes^{2},\delta_a \otimes^{2} ,(e_{j}+\delta_a )\otimes^{2} \}_{j\in \mathbb{N}}$. This shows {\bf H0''}.
%In fact it holds
%\begin{displaymath}
%[X(\cdot)]^{\epsilon}(e_{j}\otimes f_{i}) =\frac{ [X(\cdot)](e_{j}+f_{i})\otimes^{2} -[X](e_{j}\otimes^{2})-[X(\cdot)](f_{i}\otimes^{2}) }{2}
%\end{displaymath}
It remains to show that 
\begin{equation}
[X(\cdot), Y(\cdot)]^{\epsilon}\left(e_{j}\otimes \delta_a \right)(t)\longrightarrow 0
\end{equation}
in probability for every $j\in \mathbb{N}$. In fact the left-hand side equals
\[
\int_{0}^{t}\frac{\gamma_{j}(s,\epsilon)}{\epsilon} \left( X_{s+a+\epsilon}-X_{s+a}\right) \, ds  \; .
\]
Using estimate \eqref{eq ESTIGAMM}, we obtain 
\[
\begin{split}
\int_{0}^{t}\left| \frac{\gamma_{j}(s,\epsilon)}{\epsilon} 
\left( Y_{s+a+\epsilon}-Y_{s+a}\right) \right|ds 
&
\leq  T  \left( \int_{-T}^{0}|\dot{e_{j}}(y)| dy  +2\|e_{j}\|_{\infty} \right)\left( \sup_{s\in[0,T]} \left| X_{s} \right|\right) \varpi_{Y} (\epsilon) \xrightarrow[\epsilon\longrightarrow 0]{a.s.} 0
\end{split}
\]
where $\varpi_{Y} (\epsilon) $ is the usual (random in this case) continuity modulus, so the result follows.
\item[3.] An element $\delta_a \otimes \delta_b $ is difference of 
two positive definite elements $(\delta_a+\delta_b)\otimes^{2}$ and $\delta_a \otimes^{2} +\delta_b \otimes^{2}$. So that Condition {\bf H0''} 
is fulfilled. 
Concerning Condition {\bf H2''} we have
\[
[X(\cdot), Y(\cdot)]^{\epsilon}\left(\delta_a \otimes\delta_b \right)(t)=
\frac{1}{\epsilon}	
\int_{0}^{t}\left(X_{s+a+\epsilon}-X_{s+a} \right) \left(Y_{s+b+\epsilon}-Y_{s+b} \right) ds\; .
\]
This converges to $[X_{\cdot+a},Y_{\cdot+b}]$ which exists by hypothesis. 

\end{enumerate}
This finally concludes the proof of Proposition \ref{pr QV123}.
\end{proof}

\begin{cor}			\label{cor DIAG}
Let $X$ and $Y$ be two real continuous processes such that $[X]$, $[Y]$ and $[X,Y]$ exist and $a$ is a given point in $[-\tau,0]$. 
%\begin{enumerate}
%\item 
% $X(\cdot)$ and $Y(\cdot)$ admit a zero $\chi$-covariation, where $\chi=\shd_{a}([-\tau,0])\hat{\otimes}_{h}   L^{2}([-\tau,0])$.
%\item
Then $(X(\cdot),Y(\cdot))$ admits a
 $\chi^{0}([-\tau,0]^{2})$-covariation which equals 
\be		\label{eq QVCHI0}
[X(\cdot), Y(\cdot)](\mu)=\mu(\{0,0\}) [X,Y], \ \forall \mu \in \chi^0 .
\ee
%\end{enumerate}
\end{cor}
\begin{proof} 
Using Proposition \ref{pr DIRSUMHS}, it follows that $\chi^{0} ([-\tau,0]^{2})$ 
can be decomposed into  the finite direct sum decomposition 
$
L^{2}([-\tau,0]^{2}) \oplus
L^{2}([-\tau,0])\hat{\otimes}_{h}\mathcal{D}_{0}([-\tau,0]) \oplus
\mathcal{D}_{0}([-\tau,0])\hat{\otimes}_{h} L^{2}([-\tau,0])\oplus
\mathcal{D}_{0,0}([-\tau,0]^{2})
$. The results follow immediately applying Propositions \ref{prop somma diretta di qv} and \ref{pr QV123}. 
\end{proof}
When  $ \chi = \mathcal{D}_{0,0}([-\tau,0]^{2})$ the existence of a $\chi$-covariation for $(X,Y)$ holds even under weaker hypotheses. 
\begin{prop}
Let $X$, $Y$ be continuous processes such that $[X,Y]$ exists and for every sequence $(\epsilon_{n})\downarrow 0$, it exists a subsequence $(\epsilon_{n_{k}})$ such that 
\begin{equation} 			\label{eq CSBVCOV}
\sup_{k}  \frac{1}{\epsilon_{n_{k}}}\int_{0}^{T} \left| X_{s+\epsilon_{n_{k}}}-X_{s} \right|  \cdot \left|Y_{s+\epsilon_{n_{k}}}-Y_{s} \right|  ds\quad < + \infty\; .
\end{equation}
Then 1) the real covariation process $[X,Y]$ has bounded variation and 2) $X(\cdot)$ and $Y(\cdot)$ admit a $\mathcal{D}_{0,0}([-\tau,0]^{2})$-covariation and 
$
[X(\cdot),Y(\cdot)]_{t}(\mu)=\mu(\{0,0\})[X,Y]_{t} $.
\end{prop}
\begin{proof}
1) The  processes $X$ and $Y$ take  values in $B = \R$ and  the (separable) space $\chi=(B\hat{\otimes}_{\pi}B)^{\ast}$
 coincides with $\R$. Taking into account Corollary \ref{cor SGS}, 
$(X,Y)$ admits therefore a global covariation 
 which coincides with the classical covariation $[X,Y]$ defined in Definition \ref{def cov} and in particular $[X,Y]$ has bounded variation.
2) The proof is again very similar to the one of Proposition \ref{pr QV123}. The only relevant difference consists in the way of checking 
the validity of condition \textbf{H1}. This will be verified identically until \eqref{eq 5.4bis}; the successive step 
will follow by \eqref{eq CSBVCOV}.
\end{proof}

Before mentioning some examples, we give some information
about the covariation structure of bifractional Brownian motion.
\begin{prop}		\label{pr BFMP}
Let $B^{H,K}$ be a bifractional Brownian motion with $HK=1/2$. Then $[B^{H,K}]_t=2^{1-K}t$ and 
$[B^{H,K}_{\cdot+a},B^{H,K}_{\cdot+b}] =0$ for $a\neq b\in [-\tau,0]$.
\end{prop}
\begin{rem}
\begin{itemize}
\item
If $K=1$, then $H=1/2$ and $B^{H,K}$ is a Brownian motion. 
\item
In the case $K\neq 1$ we recall that the bifractional Brownian motion $B^{H,K}$ is not a semimartingale, see Proposition 6 from \cite{rtudor}.
\end{itemize}
\end{rem}
\begin{proof}[Proof of Proposition \ref{pr BFMP}.] \
%If $K=1$ then $H=1/2$ $B^{H,K}$ is a Brownian motion, case already treated. It is interesting the case $K\neq 1$, in fact we recall by 
%Proposition 6 from \cite{rtudor} that if $HK=1/2$ and $K\neq 1$ 
%then the bifractional Brownian motion $B^{H,K}$ is not a semimartingale. 
Proposition 1 in \cite{rtudor} says that $B^{H,K}$ has finite quadratic
 variation which is equal to 
$[B^{H,K}]_t=2^{1-K}t$. 
%We will show that the mixed covariation 
%$[B^{H,K}_{\cdot+a_{i}},B^{H,K}_{\cdot+a_{j}}]$ exists and is equal to zero for $i\neq j$.
By Proposition 1 and Theorem 2 in \cite{nuaLei} there are two constants 
$\alpha$ and $\beta$ depending on $K$,		%$\alpha= \sqrt{\frac{2^{-K}K}{\Gamma(1-K)} }$, $\beta=2^{\frac{1-K}{2}}$. 
a centered Gaussian process $X^{H,K}$ with 
absolutely continuous trajectories on $[0,+\infty[$ 
and a standard Brownian motion $W$
such that $\alpha X^{H,K}+B^{H,K}=\beta W$. 
Then %following equality between covariation holds
\be  \label{eqTER}
[\alpha X_{\cdot+a}^{H,K}+B_{\cdot+a}^{H,K},\alpha X_{\cdot+b}^{H,K}+B_{\cdot+b}^{H,K} ] =
\beta^{2}[ W_{\cdot+a}, W_{\cdot+b}].
\ee
Using the bilinearity of the covariation, we expand the left-hand side 
in \eqref{eqTER} into the sum of four terms 
\begin{equation}	\label{eq SUM}
\alpha^{2} [ X_{\cdot+a}^{H,K}, X_{\cdot+b}^{H,K}]+
\alpha [B_{\cdot+a}^{H,K},  X_{\cdot+b}^{H,K}]+
\alpha [X_{\cdot+a}^{H,K}, B_{\cdot+b}^{H,K} ] +
[B^{H,K}_{\cdot+a},B^{H,K}_{\cdot+b }]. 
\end{equation}
Since $X^{H,K}$ has bounded variation then the first three terms of
 \eqref{eq SUM} vanish because of point 6) of Proposition 1 in \cite{Rus05}.
%By Theorem 2 in \cite{nuaLei} $X^{H,K}$ has trajectories absolutely continuous on $[0,+\infty[$ then first three terms on \eqref{eq SUM} are zero.
On the other hand the  right-hand side of
 \eqref{eqTER} is equal to zero for $a\neq b$ since $W$ is a semimartingale,
 see Example \ref{ese PROCQV}, item 1.
We conclude that $[B^{H,K}_{\cdot+a},B^{H,K}_{\cdot+b}] =0$ if $a\neq b$.
\end{proof}

\begin{ese}			\label{ese PROCQV}
We list some examples of processes $X$ for which $X(\cdot)$ admits a $\chi$-quadratic variation through Proposition \ref{pr QV123} and Corollary \ref{cor DIAG} and 
it is explicitly given by the quadratic variation structure $[X]$ of the real process $X$.
\begin{enumerate}
\item 
All continuous real semimartingales $S$ (for instance  Brownian motion). In
 fact $S$ is a finite quadratic variation process; moreover $[S_{\cdot+a},S_{\cdot+b}]=0$ for $a\neq b$, as it easily follows by Corollary 3.11 in \cite{crwdp}.
\item 
Let $B^{H,K}$ be a bifractional Brownian motion with parameters $H$ and $K$ and such that $HK=1/2$. As shown in Proposition \ref{pr BFMP}, 
$B^{H,K}$  satisfies the hypotheses of the Corollary \ref{cor DIAG}. %admits a zero $\mathcal{D}_{a,b}$-quadratic variation for $a\neq b$.
\item 
Let $D$ be a real continuous $(\mathcal{F}_{t})$-Dirichlet process with decomposition $D=M+A$, $M$ local martingale and $A$ zero quadratic variation process. 
Then $D$ satisfies the hypotheses of the Corollary \ref{cor DIAG}. 
In fact $[D]=[M]$ and $[D_{\cdot+a},D_{\cdot+b }]=0$ for $a\neq b$. 
\end{enumerate}
\end{ese} 

We go on evaluating other $\chi$-covariations.
\begin{prop}
Let $V$ and $Z$ be two real absolutely continuous %bounded variation 
processes such that 
$V', Z' \in L^{2}([0,T])$ $\omega$-a.s.
Then  $ (V(\cdot), Z(\cdot))$ has  
zero scalar and tensor covariation. In particular $(V(\cdot), 
Z(\cdot))$  admits
 a zero global covariation.
\end{prop}
\begin{proof} 
Similarly to the proof of Proposition \ref{prop ZQVHC}, 
by Remark \ref{rem RE1} item 4. and Proposition \ref{pr TGH78} we only
 need to show that $(V(\cdot),Z(\cdot))$ admits a zero scalar covariation,
 i.e.  the convergence to zero in probability of the quantity
\begin{equation}		\label{eq ZE1}
\int_{0}^{T}
\frac{1}{\epsilon} \left \| 
V_{s+\epsilon}(\cdot)-V_{s}(\cdot)
\right\|_{B}  \left \| 
Z_{s+\epsilon}(\cdot)-Z_{s}(\cdot)
\right\|_{B}  ds  .
\end{equation}
%which equals 
%\[
%\int_{0}^{T}
%\frac{1}{\epsilon} \left \| 
%  \left( V_{s+\epsilon}(\cdot)-V_{s}(\cdot)\right) \otimes \left( Z_{s+\epsilon}(\cdot)-Z_{s}(\cdot) \right)
%\right\|_{(B\hat{\otimes}_{\pi}B)^{\ast\ast}}  ds  \, .
%\]
By Cauchy-Schwarz, \eqref{eq ZE1} is bounded by
\begin{equation}		\label{eq ZE11}
\sqrt{\int_{0}^{T}
\frac{1}{\epsilon} \sup_{x\in [-\tau,0]} 
\left | 
V_{s+\epsilon}(x)-V_{s}(x)
\right|^{2} ds}\cdot
\sqrt{
 \int_{0}^{T}
\frac{1}{\epsilon} \sup_{x\in [-\tau,0]} 
\left | 
Z_{u+\epsilon}(x)-Z_{u}(x)
\right|^{2} du}\; , 
\end{equation}
which will be shown to converge even a.s. to zero. The square of the first square root in \eqref{eq ZE11} equals
\[
\int_{0}^{T}
\frac{1}{\epsilon} \sup_{x\in [-\tau,0]} 
\left | 
\int_{s+x}^{s+x+\epsilon} V'(y)dy
\right|^{2} ds
\leq 
\int_{0}^{T}
\frac{1}{\epsilon} \max_{x\in [-\tau,0]}  
\int_{s+x}^{s+x+\epsilon} V'(y) ^{2}  dy
ds
\leq 
T \,\varpi_{\int_{0}^{\cdot}(V'^{2})(y)dy}(\epsilon) 
\xrightarrow[\epsilon\longrightarrow 0]{a.s.} 0 \, ,
\]
since $\varpi_{\int_{0}^{\cdot}(V'^{2})(y)dy}(\epsilon) $ denotes the modulus of continuity of the a.s. continuous function $t\mapsto \int_{0}^{t}(V'^{2})(y)dy$. 
The square of the second square root in \eqref{eq ZE11} can be treated analogously and the result is finally established.
\end{proof}
If $X$ is a finite quadratic variation  processes then $\X=X(\cdot)$ admits a $Diag([-\tau,0]^{2})$-quadratic variation, where $Diag([-\tau,0]^{2})$ was defined in \eqref{eq-def diag}. 
This is the object of Proposition \ref{pr DIAG tau}.
\begin{prop}  		\label{pr DIAG tau}    
Let $0< \tau \leq T$. 
Let  $X$ and $Y$ be two real continuous processes such that $[X,Y]$ exists and
\eqref{eq CSBVCOV} is verified. 
Then $(X(\cdot),Y(\cdot))$ admits a $Diag([-\tau,0]^{2})$-covariation.
Moreover we have 
%\begin{equation*}			
%[X(\cdot),Y(\cdot)]\,: Diag([-\tau,0]^{2})  \longrightarrow \mathscr{C}([0,T])   \\
%\end{equation*}
%given by 
\begin{equation}			\label{eq QV DIAG tau}
\widetilde{[X(\cdot),Y(\cdot)]}_{t}(\mu)=\int_{0}^{t\wedge \tau } g(-x) [X,Y]_{t-x} dx \; , \hspace{2cm}		t\in[0,T] \; ,
\end{equation}
where $\mu$ is a generic element in $Diag([-\tau,0]^{2})$ of the type $\mu(dx,dy)=g(x)\delta_{y}(dx)dy$, 
with associated $g$ in $L^{\infty}([-\tau,0])$. 
\end{prop}
\begin{rem}\label{R416}
Taking into account the usual convention
 $[X,Y]_{t}=0$ for $t < 0$, the process\\
 $\left( \int_{0}^{t\wedge \tau } g(-x) [X,Y]_{t-x} dx\right)_{0 \le t \le T}$ 
can also be written as 
$\left( \int_{0}^{\tau } g(-x) [X,Y]_{t-x} dx\right)_{0 \leq t \le T}$. 
\end{rem}
\begin{proof} [Proof of Proposition \ref{pr DIAG tau}]
We recall that, for a generic element $\mu$, we have $\|\mu\|_{Diag}=\|g\|_{\infty}$.\\ 
First we verify Condition {\bf H1}. We can write
\begin{displaymath}
\begin{split}
&
\frac{1}{\epsilon}\int_{0}^{T}\sup_{\| \mu \|_{Diag }  \leq 1}  \left|  \langle \mu, \left( X_{s+\epsilon}(\cdot)-X_{s}(\cdot) \right)
\otimes \left( Y_{s+\epsilon}(\cdot)-Y_{s}(\cdot) \right)\rangle   \right| \, ds 
\\
& \leq 
 \frac{1}{\epsilon} \int_{0}^{T} \sup _{\| g \|_{\infty} \leq 1} \left|  \int_{-T}^{0}g(x)  \left( X_{s+\epsilon}(x)-X_{s}(x) \right) \left( Y_{s+\epsilon}(x)-Y_{s}(x) \right)\, dx \right| \,ds
  \\
&=
\int_{0}^{T}   \sup _{\| g \|_{\infty} \leq 1}  \left|  \int_{0}^{s}   \frac{  \left(X_{x+\epsilon}-X_{x}\right) \left(Y_{x+\epsilon}-Y_{x}\right)  }{\epsilon}g(x-s)\,dx\right| ds  \; .
\end{split}
\end{displaymath}
Condition {\bf H1} is verified because of Hypothesis \eqref{eq CSBVCOV}.\\
It remains to prove Condition {\bf H2}. 
Using Fubini's theorem, we write
\begin{equation}		\label{eq DQV}
\begin{split}
 [X(\cdot),Y(\cdot)]^{\epsilon}_{t}(\mu)&=
\frac{1}{\epsilon}\int_{0}^{t}\langle \mu(dx,dy),\left(X_{s+\epsilon}(\cdot)-X_{s}(\cdot)\right)\otimes \left(Y_{s+\epsilon}(\cdot)-Y_{s}(\cdot)\right)\rangle\,ds\\
%&=\frac{1}{\epsilon} \int_{0}^{t}\int_{[-\tau,0]^{2}}\left(X_{s+\epsilon}(x)-X_{s}(x)\right)\left(Y_{s+\epsilon}(y)-Y_{s}(y)\right)g(x)\delta_{x}(dy)dx\;ds\\
&=\frac{1}{\epsilon}\int_{0}^{t}\int_{[-\tau,0]}\left(X_{s+\epsilon}(x)-X_{s}(x)\right) \left(Y_{s+\epsilon}(x)-Y_{s}(x)\right)g(x)dx\;ds\\
&=\int_{(-t)\vee (-\tau)}^{0} g(x)\int_{-x}^{t} \frac{  \left(X_{s+x+\epsilon}-X_{s+x}\right) \left(Y_{s+x+\epsilon}-Y_{s+x}\right) }{\epsilon}\,ds\,dx\\
&		
=
\int_{(-t)\vee (-\tau)}^{0} g(x)\int_{0}^{t+x}\frac{\left( X_{s+\epsilon}-X_{s}\right)\left( Y_{s+\epsilon}-Y_{s}\right) }{\epsilon}\,ds\;dx\\
&
=\int_{0}^{ t\wedge \tau } g(-x)\int_{0}^{t-x}\frac{\left( X_{s+\epsilon}-X_{s}\right)\left( Y_{s+\epsilon}-Y_{s}\right) }{\epsilon}\,ds\;dx  \; .
\end{split}
\end{equation}
To conclude the proof of \textbf{H2}(i) it remains to show that
% the following cup convergence for previous term.
\[
\left( \int_{0}^{t\wedge \tau} g(-x)\int_{0}^{t-x}\frac{\left( X_{s+\epsilon}-X_{s}\right)\left( Y_{s+\epsilon}-Y_{s}\right) }{\epsilon}\,ds\;dx \right)_{t\in [0,T]}
\xrightarrow[\epsilon\longrightarrow 0]{ucp} 
\left(\int_{0}^{t\wedge \tau} g(-x)[X,Y]_{t-x}\,dx \right)_{t\in [0,T]},
\] 
\begin{equation}		\label{eq AMONT}
\textrm{i.e.}
\quad 
\sup_{t\leq T} \left|  \int_{0}^{t\wedge \tau } \left( g(-x)\int_{0}^{t-x}\frac{\left( X_{s+\epsilon}-X_{s}\right) \left( Y_{s+\epsilon}-Y_{s}\right) }{\epsilon}\,ds  -[X,Y]_{t-x} \right) \;dx  
\right| \xrightarrow [\epsilon \longrightarrow 0 ]  {\mathbb{P}}0  \, .
\end{equation} 
The left-hand side of \eqref{eq AMONT} is bounded by 
\[
\begin{split}
&
\int_{0}^{T}|g(-x)|  \sup_{t\in[0,T]}  
\left| \int_{0}^{t-x} \frac{ \left(X_{s+\epsilon}-X_{s}\right)\left(Y_{s+\epsilon}-Y_{s}\right) }{\epsilon} \,ds- [X,Y]_{t-x} \right| dx 
%\\
%&
%\leq \int_{0}^{T} |g(-x)|  \sup_{t\in[0,T]}  \left| \int_{0}^{t} \frac{ \left(X_{s+\epsilon}-X_{s}\right) \left(Y_{s+\epsilon}-Y_{s}\right)}{\epsilon} \,ds- [X,Y]_{t} \right| dx 
\\
&
\leq
T \left\|g\right\|_{\infty} 
\sup_{t\in[0,T]}  
\left| \int_{0}^{t} \frac{ \left(X_{s+\epsilon}-X_{s}\right) \left(Y_{s+\epsilon}-Y_{s}\right) }{\epsilon} \,ds- [X,Y]_{t} \right| \; .
\end{split}
\]
Since $X$ and $Y$ admit a covariation, previous expression converges to zero.
This shows  Condition \textbf{H2}(i).\\
Concerning  Condition \textbf{H2}(ii), we have
\[
[X(\cdot),Y(\cdot)]_{t}(\mu)=\int_{0}^{t\wedge \tau } g(-x)[X,Y]_{t-x}\,dx=
\begin{dcases}
\int_{0}^{t}g(-x)[X,Y]_{t-x} dx &  0\leq t\leq \tau \\
\int_{0}^{\tau}g(-x)[X,Y]_{t-x}dx     &			\tau< t\leq T\; .
\end{dcases} 
\]
%Concerning Condition \textbf{H2}(ii), we observe that
 Previous expression has an obvious modification $\widetilde{[X(\cdot),Y(\cdot)]}$
which has finite variation with values in $\chi^\ast$.
The total variation is  in fact easily dominated by $\int_{0}^{T} | [X,Y]_{x}|dx$. 
% which because of \eqref{eq CSBVCOV}.
%by $\tau [X,Y]_T$.
\end{proof}
A useful proposition related to Proposition \ref{pr DIAG tau} is the following. 
We recall that $D([-\tau,0])$ denotes the space of  c\`adl\`ag functions equipped with
 the uniform norm and  $Diag_{d}([-\tau,0]^{2})$ was introduced in Notation \ref{nnn}. 
\begin{prop}		\label{prop 5GT6}
Let $X$ be a finite quadratic variation process. Let $G:[0,T]\longrightarrow \chi:=Diag_{d}([-\tau,0]^{2})$, c\`adl\`ag. 
We have 
\begin{equation}			\label{eq A1}
\int_{0}^{T} \prescript{}{\chi}{\langle} G(s)\,,  \, d\widetilde{[X(\cdot)]}_{s}\rangle_{\chi^{\ast}}=
\int_{0}^{\tau} \left(  \int_{x}^{T} g(s,-x)[X]_{ds-x}  \right) dx=
\int_{0}^{\tau}\left( \int_{0}^{T-x} g(s+x,-x)d[X]_{s}\right)dx ,
\end{equation}
where $G(s)=g(s,x)\delta_{y}(dx)dy$ for some bounded Borel function $g :[0,T]\times [-\tau,0]\longrightarrow \R$ and $[X]_{ds-x}$ represents the
measure differential associated with
the increasing function $s\mapsto [X]_{s+x}$. 
\end{prop}
\begin{proof} 
We remark that $t\mapsto g(t,\cdot)$ is left continuous from $[0,T]$ to $D( [-\tau,0])$ equipped with the $\|\cdot\|_{\infty}$ norm.
By item 2 in Remark \ref{pro 3.2}, Proposition \ref{prop XQVSM} item 2 and Proposition \ref{pr DIAG tau}, $X(\cdot)$ admits a
 $\chi$-quadratic variation. The proof will be established fixing $\omega\in \Omega$. 
We first suppose that  
\begin{equation}			\label{eq A2}
G(s)=\sum_{i=0}^{N-1} A_{i}\1_{]t_{i},t_{i+1}]}(s)+A_{0}\1_{\{0\} }(s) ,
\end{equation}
where, for some positive integer $N\in \mathbb{N}$, 
$0=t_{0} < \ldots  < t_{N}=T$;
% is an element of subdivisions of $[0,T]$;
$A_{0},\ldots,A_{N}\in \chi$; in particular there are $a_{0},\ldots, a_{N}\in 
D_d([-\tau,0])$ with 
\begin{equation}		\label{eq A3}
A_{i}(dx,dy)=a_{i}(x)\delta_{y}(dx)dy  \quad \textrm{ for all } i\in \{0,\ldots , N\} \; .
\end{equation}
Then \eqref{eq A1} holds by use of Proposition \ref{pr DIAG tau}.\\
To treat the general case we approach a general $G$ by a sequence $\left( G^{n}\right)$ of type \eqref{eq A2}, i.e. 
\begin{equation}		\label{eq A4}
G^{n}(s)=\sum_{i=0}^{N-1} A_{i}^{n}\1_{]t_{i},t_{i+1}]}(s)+A_{0}^{n}\1_{\{0\} }(s)
\end{equation}
where $A_{i}^{n}=G(t_{i})$, $0\leq i\leq (N-1)$, 
$0=t_{0} < \ldots  < t_{N}=T$
 is a an element of subdivisions of $[0,T]$ indexed by $n$ whose mesh goes to zero when $n$ diverges to infinity.
Let $a^{n}_{0},\ldots, a^{n}_{N}\in D([-\tau,0])$ related to 
$A_{0}^{n}, \ldots, A_{N}^{n}$ through relation \eqref{eq A3}. 
Consequently we have 
\begin{equation}		\label{eq A5}
\int_{0}^{T}\prescript{}{\chi}{\langle}G^{n}(s)\, , \, d\widetilde{[X(\cdot)]}_{s} \rangle_{\chi^{\ast}} =
\int_{0}^{\tau} \left(  \int_{x}^{T} g^{n}(s,-x)[X]_{ds-x}  \right) dx
\end{equation}
with $g^{n}(s,x)=\sum_{i=0}^{N-1}a_{i}^{n}(x)\1_{]t_{i},t_{i+1}]}(s)+a_{0}^{n}$.
 In particular $a_{i}^{n}=g(t_{i},\cdot)$.\\
By assumption, for every $s\in [0,T]$ we have 
\[
\lim_{n\rightarrow +\infty} \sup_{x\in [-\tau,0]} \left|  g^{n}(s,x)-g(s,x) \right|=0 \; .
\]
Consequently, for every $x\in [0,\tau]$, by Lebesgue dominated convergence theorem,
\[
\lim_{n\rightarrow +\infty} \int_{x}^{T} \left(g^{n}(s,-x)-g(s,-x) \right)[X]_{ds-x}=0\; .
\]
Moreover 
\[
\left| \int_{x}^{T} \left(g^{n}(s,-x)-g(s,-x) \right)[X]_{ds-x} \right|
\leq
\left( \sup_{n} \|g^{n}\|_{\infty}+\|g\|_{\infty}\right) [X]_{T}  .
 \]
 Again by Lebesgue dominated convergence theorem, the right-hand side of \eqref{eq A5} converges to the right-hand side of \eqref{eq A1} and the result follows.
\end{proof}
\begin{rem}
If $[X]$ is absolutely continuous with respect to Lebesgue, the identities \eqref{eq A1} are still valid with $\chi=Diag([-\tau,0]^{2})$.
\end{rem}
\section{It\^{o} formula}			\label{sec:ito-int}
We need now to formulate the definition of the forward type integral for $B$-valued integrator and $B^{\ast}$-valued integrand, where $B$ is a separable Banach space.
\begin{dfn}		\label{def integ fwd}  		 
Let $(\mathbb{X}_{t})_{t\in[0,T]}$ (respectively $(\mathbb{Y}_{t})_{t\in[0,T]}$) be a
$B$-valued (respectively a $B^{\ast}$-valued) stochastic process. 
We suppose $\mathbb{X}$ to be continuous and $\mathbb{Y}$ to be strongly measurable %(in the Bochner sense) 
such that $\int_{0}^{T}\|\mathbb{Y}_{s}\|_{B^{\ast}} ds <+\infty$ a.s.
For every fixed $t\in [0,T]$ we define the {\bf definite forward integral of $\mathbb{Y}$ with respect
to $\mathbb{X}$} denoted by $\int_{0}^{t}\prescript{}{B^{\ast}}{\langle} \mathbb{Y}_{s}, d^{-}\mathbb{X}_{s}\rangle_{B}$ as the following limit in probability:
\begin{equation}		\label{def INTFWD}
\int_{0}^{t} \prescript{}{B^{\ast}}{\langle} \mathbb{Y}_{s}, d^{-}\mathbb{X}_{s}\rangle_{B}: =\lim_{\epsilon\rightarrow
0}\int_{0}^{t} \prescript{}{B^{\ast}}{\langle} \mathbb{Y}_{s},\frac{\mathbb{X}_{s+\epsilon}-\mathbb{X}_{s}}{\epsilon}\rangle_{B} ds		\; .
\end{equation}
We say that the {\bf forward stochastic integral of $\mathbb{Y}$ with respect
to $\mathbb{X}$} exists if the process
\[
\left(\int_{0}^{t} \prescript{}{B^{\ast}}{\langle} \mathbb{Y}_{s}, d^{-}\mathbb{X}_{s}\rangle_{B}\right)_{t\in[0,T]}
\]
admits a continuous version. In the sequel indices $B$ and $B^{\ast}$ will  often be omitted.
\end{dfn}

We are now able to state an It\^{o} formula for stochastic processes with values in a general separable Banach space.
\begin{thm}  			 \label{thm ITONOM}
Let $\chi$ be a Chi-subspace %of $(B\hat{\otimes}_{\pi}B)^{\ast}$ 
and $\X$ a $B$-valued continuous process
admitting a $\chi$-quadratic variation.
Let $F:[0,T]\times B\longrightarrow \mathbb{R}$  Fr\'echet of class $C^{1,2}$
% be a function once continuously 
%differentiable with respect to the first variable $t$ and 
%of class $C^{2}$ in the
% Fr\'echet sense with respect to the second variable 
such that $D^{2}F(t,\eta)\in \chi$ for all $t\in [0,T]$ and $\eta\in C([-T,0])$ and 
$	
D^{2}F:[0,T]\times B\longrightarrow \chi$ is continuous.\\
Then for every $t\in [0,T]$ the forward integral 
\[
\int_{0}^{t}\prescript{}{B^{\ast}} {\langle} DF(s,\X_{s}),d^{-}\X_{s}\rangle_{B}
\] 
exists and the following formula holds.
\begin{equation}					\label{eq ITONOM}
F(t,\X_{t})=F(0, \X_{0})+\int_{0}^{t}\partial_{t}F(s,\X_{s})ds+\int_{0}^{t}\prescript{}{B^{\ast}}{\langle} DF(s,\X_{s}),d^{-}\X_{s}\rangle_{B} 
+\frac{1}{2}\int_{0}^{t} \prescript{}{\chi}{\langle} D^{2}F(s,\X_{s}),
d\widetilde{[\X]}_{s}\rangle_{\chi^{\ast}}. 
\end{equation}
\end{thm}
\begin{rem} \label{Roperational}
The statement of Theorem \ref{thm ITONOM} induces  some operational comments.
The Chi-subspace $\chi$ of $(B\hat{\otimes}_{\pi}B)^{\ast}$ constitutes a degree of freedom in the statement of 
It\^o formula. In order to find the suitable expansion for $F(t,\X_{t})$ we may proceed as follows.
\begin{itemize}
\item Let $F:[0,T]\times B \longrightarrow \R$ of class $C^{1,1}([0,T]\times B )$ we compute the second order 
derivative $D^{2}F$ if it exists.
\item We look for the existence of a Chi-subspace $\chi$ %of $(B\hat{\otimes}_{\pi}B)^{\ast}$ 
for which the range of 
$D^{2}F:[0,T]\times B\longrightarrow (B\hat{\otimes}_{\pi}B)^{\ast}$ is included in $\chi$ and it is continuous with respect to the 
topology of $\chi$.
\item We verify that $\X$ admits a $\chi$-quadratic variation.
\end{itemize}
We observe that whenever $\X$ admits a global quadratic variation, i.e.
a $\chi$-quadratic variation with   $\chi=(B\hat{\otimes}_{\pi}B)^{\ast}$, 
the condition on $F$ to be checked is that it belongs to 
$ C^{1,2}([0,T]\times B )$. 
When $\X$ is a semimartingale (or more generally a semilocally summable $B$-valued process with respect to the tensor product) then it admits a tensor quadratic variation and in particular previous result 
generalizes the classical It\^o formula in \cite{mp}, Section 3.7.

\end{rem}
\begin{proof} [Proof of Theorem \ref{thm ITONOM}]
%We fix $t\in [0,T]$ and 
We observe that the quantity
\begin{equation}  				\label{eq ABB}
I_{0}(\epsilon,t) =\int_{0}^{t} \frac{F(s+\epsilon, \X_{s+\epsilon})-F(s,\X_{s})}{\epsilon}ds, \ t \in [0,T],
\end{equation}
converges ucp for $\epsilon\rightarrow 0$ to
$F(t, \X_{t})-F(0,\X_{0})$ since $\big( F(s,\X_{s})\big)_{s\geq 0}$ is continuous.
At the same time,   \eqref{eq ABB} 
can be written as the sum of the two terms:
\begin{equation}		\label{eq ABBI1}
I_{1}(\epsilon,t)
=
\int_{0}^{t}\frac{F(s+\epsilon, \X_{s+\epsilon})-F(s,\X_{s+\epsilon})}{\epsilon}ds
\end{equation}
and
\begin{equation}		\label{eq ABBI2}
I_{2}(\epsilon,t)
=
\int_{0}^{t}\frac{F(s, \X_{s+\epsilon})-F(s,\X_{s})}{\epsilon}ds \; , 		\quad \epsilon> 0 , \quad \;  t \; \in [0,T] \; .
\end{equation}
%We fix $t\in [0,T]$ and 
We prove that 
\begin{equation}		\label{eq I1}
I_{1}(\epsilon,\cdot)\longrightarrow  \int_{0}^{\cdot} \partial_{t}F(s,\X_{s})ds
\end{equation}
ucp.
% for every fixed $t\in [0,T]$.
 In fact
\begin{equation}		\label{eq ABC}
I_{1}(\epsilon,t)
= \int_{0}^{t}\partial_{t} F(s,\X_{s+\epsilon})ds+R_{1}(\epsilon,t), t \in [0,T],
\end{equation}
where 
\[
\begin{split}
R_{1}(\epsilon,t)
&
=
\int_{0}^{t}\int_{0}^{1} \left(  \partial_{t}F\big(s+\alpha \epsilon,
\X_{s+\epsilon}\big)-\partial_{t}F(s,\X_{s+\epsilon}) \right) d\alpha ds, \
t \in [0,T]  \; .
\end{split}
\]
For fixed $\omega \in \Omega$ we denote 
by $\mathcal{V}(\omega):=\{ \X_{t}(\omega);\, t\in [0,T]\}$ and 
\begin{equation}		\label{def U}
\mathcal{U}=\mathcal{U}(\omega)=\overline{conv(\mathcal{V}(\omega))},
\end{equation} 
i.e. 
the set $\mathcal{U}$ is the closed convex hull of the compact subset $\mathcal{V}(\omega)$ of $B$. 
For $x \in \Omega$,  we have 
\[ 
\sup_{t\in [0,T]} \vert R_{1}(\epsilon,t)\vert\leq T\,
 \varpi^{[0,T]\times \mathcal{U}}_{\partial_{t}F}(\epsilon),
\] 
where 
$ \varpi^{[0,T]\times \mathcal{U}}_{\partial_{t}F}(\epsilon)$ is the continuity modulus in $\epsilon$ of the application 
$\partial_{t}F:[0,T]\times B \longrightarrow \R$ restricted to $[0,T]\times\mathcal{U}$.
%and $ \mathcal{U} = \mathcal{U}(\omega)$ is the (random) compact set 
% defined in \eqref{def U}. 
%and there we also proved that it is a compact set, as well as $[0,T]\times \mathcal{U}$.   
From the continuity of the $\partial_{t}F$ as function from $[0,T]\times B$ to $\R$, it follows that the restriction on $[0,T]\times \mathcal{U}$ is uniformly continuous and 
$ \varpi^{[0,T]\times \mathcal{U}}_{\partial_{t}F}$ is a positive, increasing function on $\mathbb{R}^{+}$ converging to 
$0$ when the argument converges to zero.
In particular we have  proved  that 
$R_{1}(\epsilon,\cdot)\rightarrow 0$ ucp as $\epsilon\rightarrow 0$.\\
On the other hand the first term in \eqref{eq ABC} can be rewritten as 
\[
\int_{0}^{t}\partial_{t}F(s,\X_{s})ds+R_{2}(\epsilon,t)
\]
where $R_{2}(\epsilon,t)\rightarrow 0$ ucp arguing similarly as for $R_{1}(\epsilon,t)$  and so the convergence \eqref{eq I1} is  established.\\
We fix now $t \in [0,T]$.
The second addend $I_{2}(\epsilon,t)$ in \eqref{eq ABBI2}, can be 
approximated by Taylor's expansion and it can be written as the sum of the following three terms:
\[
\begin{split}
I_{21}(\epsilon,t)
&
=\int_{0}^{t}\prescript{}{B^{\ast}}{\langle} DF(s,\X_{s}),\frac{\X_{s+\epsilon}-\X_{s}}{\epsilon}\rangle_{B} ds  \; ,
\\
I_{22}(\epsilon,t)
&
=\frac{1}{2}\int_{0}^{t} \prescript{}{\chi}{\langle} D^{2}F(s,\X_{s}),\frac{(X_{s+\epsilon}-\X_{s})\otimes^{2}}{\epsilon}\rangle_{\chi^{\ast}} ds  \; ,
\\
I_{23}(\epsilon,t)
&
=\int_{0}^{t}\left[ \int_{0}^{1}\alpha  \prescript{}{\chi}{\langle}  D^{2}F\left(s,(1-\alpha)\X_{s+\epsilon}+\alpha \X_{s}\right)-D^{2}F(s,\X_{s}),\frac{ (\X_{s+\epsilon}-\X_{s})\otimes^{2} }{\epsilon}\rangle_{\chi^{\ast}} \,d\alpha  \right]\,ds  \; .
\end{split}
\]
Since $D^{2}F: [0,T] \times B  \longrightarrow \chi$ is continuous and $B$ separable, we observe that the process $H$ defined by
$H_{s}= D^{2}F(s,X_{s})$ takes values in a separable closed subspace $\mathcal{V}$ of $\chi$. Applying Corollary \ref{pr CONVCCOV},
it yields
\[
I_{22}( \epsilon,t)\xrightarrow[\epsilon\rightarrow 0] {\mathbb{P}} \frac{1}{2}\int_{0}^{t}\prescript{}{\chi}{\langle} D^{2}F(s,\X_{s}),d\widetilde{[\X]}_{s}\rangle_{\chi^{\ast}}
\quad {\mbox{for every $t \in [0,T]$.}}
\]
We analyze now $I_{23}( \epsilon, t)$ and we show that $I_{23}( \epsilon, t)\xrightarrow[\epsilon\longrightarrow 0]{\mathbb{P}} 0$. 
In fact we have
\[
\begin{split}
\left| I_{23}(\epsilon, t)\right| & 
\leq
\frac{1}{\epsilon}\int_{0}^{t}\int_{0}^{1} \alpha \left|\prescript{}{\chi}{\langle}  D^{2}F\left(s,(1-\alpha)\X_{s+\epsilon}+\alpha \X_{s}\right)-D^{2}F(s,\X_{s}),(\X_{s+\epsilon}-\X_{s})\otimes^{2}\rangle_{\chi^{\ast}}
\right|\,d\alpha \,ds
\\
&
\leq
\frac{1}{\epsilon}\int_{0}^{t}\int_{0}^{1}\alpha \left\|D^{2}F\left(s,(1-\alpha)\X_{s+\epsilon}+\alpha \X_{s}\right)-D^{2}F(s,\X_{s})\right\|_{\chi}\,\left\|(\X_{s+\epsilon}-\X_{s})\otimes^{2}\right\|_{\chi^{\ast}}\,d\alpha \,ds
\\
&
\leq 
\varpi^{[0,T]\times \mathcal{U}}_{D^{2}F}(\epsilon)
\int_{0}^{t}  \sup_{\|\phi\|_{\chi}\leq 1}
\left| \langle \phi,\frac{(\X_{s+\epsilon}-\X_{s})\otimes^{2}}{\epsilon}\rangle \right| \,ds    \; ,
\end{split}
\]
where $\varpi^{[0,T]\times \mathcal{U}}_{D^{2}F}(\epsilon)$ is the continuity modulus of the application 
$D^{2}F:[0,T]\times B \longrightarrow \chi$ restricted to  $[0,T]\times \mathcal{U}$
where $\mathcal{U}$ is the same random compact set
introduced  in \eqref{def U}.
%and there we also proved that it is a compact set.   
%The set $\mathcal{U}$ is the closed convex hull of the compact subset $\mathcal{V}$ of $B$ defined, for every fixed $\omega$, by 
%$$\mathcal{V}=\mathcal{V}(\omega):=\{ X_{t}(\omega);\, t\in [0,T]\}.$$ 
%According to Theorem 5.35 from \cite{InfDimAn}, $\mathcal{U}=\overline{conv(\mathcal{V})}$ is compact, 
Again $D^{2}F$ on $[0,T]\times \mathcal{U}$ is uniformly continuous and 
$ \varpi^{[0,T]\times \mathcal{U}}_{D^{2}F}$ is a positive, increasing function on $\mathbb{R}^{+}$ converging to 
$0$ when the argument converges to zero.
%We denote $I_{4}(\epsilon, t)$ for term $\int_{0}^{t}  \sup_{\|\phi\|_{\chi}\leq 1}
%\left| \langle \phi,\frac{(X_{s+\epsilon}-X_{s})\otimes^{2}}{\epsilon}\rangle \right| \,ds$. 
Taking into account condition {\bf H1} in the definition of $\chi$-quadratic variation, 
%given a sequence $(n_{k})$ there is a 
%subsequence $(n_{k_{j}})$ such that 
%%term $I_{4}(\epsilon_{n_{k_{j}}},t)$ is bounded 
%%and consequently $I_{3}(\epsilon_{n_{k_{j}}},t)=\varpi^{\mathcal{U}}_{D^{2}F}(\epsilon_{n_{k_{j}}})  
%%\times I_{4}(\epsilon_{n_{k_{j}}},t)$ converges to zero a.s. for $\epsilon\rightarrow 0$. 
%$I_{23}(\epsilon_{n_{k_{j}}},t)$ converges to zero a.s. Finally 
$I_{23}(\epsilon,t)\rightarrow 0$ in probability when $\epsilon$ goes to zero.\\
Since $I_{0}(\epsilon,t)$, $I_{1}(\epsilon,t)$, $I_{22}(\epsilon,t)$ and $I_{23}(\epsilon,t)$ converge in probability for every fixed $t\in [0,T]$, 
it follows that $I_{21}(\epsilon,t)$ converges in probability when $\epsilon\rightarrow 0$. Therefore the forward integral
\[
\int_{0}^{t}\prescript{}{B^{\ast}}{\langle} DF(s,\X_{s}),d^{-}\X_{s}\rangle_{B}
\]
exists by definition.
This in particular implies the It\^o formula \eqref{eq ITONOM}.
\end{proof}
%
%
%As corollary of Theorem \ref{thm ITONOM} we have the so-called time-homogeneous It\^o's formula, i.e. without the dependence on the time variable $t$. 
%%
%%
%\begin{cor}		\label{cor ITOOM}
%Let $B$ be a separable Banach space, $\chi$ be a Chi-subspace of $(B\hat{\otimes}_{\pi}B)^{\ast}$ 
%and $X$ a $B$-valued continuous process
%admitting a $\chi$-quadratic variation. Let
%$G:B\longrightarrow \mathbb{R}$ a function of class $C^{2}$ Fr\'echet such that
%\begin{equation} 
%D^{2}G:B\longrightarrow \chi \subset
%(B\hat{\otimes}_{\pi}B)^{\ast} \textrm{ continuously with respect to } \chi
%\end{equation}
%Then for every $t\in [0,T]$ the forward integral 
%\[
%\int_{0}^{t}\prescript{}{B^{\ast}} {\langle} DG(X_{s}),d^{-}X_{s}\rangle_{B}
%\]
%exists and following formula a.s. holds:
%\begin{equation}				\label{eq ITOOM}
%G(X_{t})=G(X_{0})+\int_{0}^{t}\prescript{}{B^{\ast}} {\langle} DG(X_{s}),d^{-}X_{s}\rangle_{B} +
%\frac{1}{2}\int_{0}^{t}
%\prescript{}{\chi} {\langle} D^{2}G(X_{s}),d\widetilde{[X]}_{s}\rangle_{\chi^{\ast}}.
%\end{equation}
%\end{cor}
%\begin{proof}\
%The proof is just an application of Theorem \ref{thm ITONOM} without the dependence on time variable $t$
%\end{proof}

\section{Applications of It\^o formula for window processes}  \label{Sec6}

\subsection{Some conventions}
\label{S6Conventions}

The scope of this section is to illustrate some applications of our Banach space valued It\^o formula to window processes. 
In this section $D^m$ denotes the classical Malliavin gradient and
$\D^{1,2}\left(L^{2}([0,T]) \right)$ (shortly $\D^{1,2}$) denotes the classical Malliavin-Sobolev space, related
to the case when $X$ is a classical Brownian motion.
For more information the reader may consult for instance
\cite{nualart}. On the other hand $D$ will denote the Fr\'echet
differentiation operator for functionals defined on $B$.   
We go on fixing some notations. Let $0< \tau\leq T$, we set $B=C([-\tau,0])$.
%In this section we will deal with Fr\'echet derivatives of functionals on $B$. 
  \begin{nota}	\label{nota2.3bis}
Let  $B = C([-\tau,0])$ and $I$ be a real interval. Consider 
 $F: I \times B \longrightarrow 
\mathbb{R}$ of class $C^{0,1}(I \times B).$
Then,
 for each $t \in I$ and $ \eta \in B$,  $\mu = D u(t,\eta)$
 is a (signed) measure on $[-\tau,0]$. We will simply  denote
$D^\perp u(t,\eta)$ (resp. $D^{\delta_0} u(t,\eta)$)
the quantity which, according to  Notation \ref{nota2.3}, should be
  $(D u(t,\eta))^\perp$ (resp.  $(D u(t,\eta))^{\delta_0}$).
We remark that, for any $t \in I$  and $ \eta \in B$, 
$ D^{\delta_0}  \, F \,(t,\eta) = DF \,(t,\eta)(\{0\})$ and
 $D^\perp \, F(t, \eta) = 
%D F\, (t, \eta) - DF \,(t,\eta)(\{0\}) \delta_0 =
  D F\, (t, \eta) - D^{\delta_0} \,F \,(t,\eta) \delta_0$.
\end{nota}
We go on fixing further conventions.
Let $F:[0,T] \times B \longrightarrow \mathbb{R} $ Fr\'echet of
 class $C^{1,2}([0,T[\times B) \cap C^{0}([0,T]\times B) $. 
We remind that the first order Fr\'echet derivative $DF$ defined on $[0,T[ \times B$ takes values in $B^* \cong \mathcal{M}([-\tau,0])$.
For all $(t,\eta)\in [0,T[\times B$, we will denote by $D_{dx} F(t,\eta )$ the measure defined by
% following duality holds $ \forall \;
\begin{equation}   \label{eq duality deriv prima}
\prescript{}{\mathcal{M}([-\tau,0])}{\langle} DF(t,\eta), h \rangle_{C([-\tau,0])}=DF(t,\eta)(h)=\int_{[-\tau,0]} h(x)D_{dx} F(t,\eta )\quad \mbox{for every $ h\in C([-\tau,0])$.}
\end{equation}
We remark that the second order Fr\'echet derivative $D^2F$ defined on $[0,T]\times B$ takes values in $L\left(B; B^{\ast} \right) 
\cong\mathcal{B}(B,B)
\cong  \left( B\hat{\otimes}_{\pi} B \right)^{\ast}$.
Recalling \eqref{eq 1.15bis}, if $D^{2}F\,(t, \eta)\in \mathcal{M}([-\tau,0]^{2})$ 
for all $(t,\eta) \in [0,T]\times  B$ (which will happen in most of
 the treated cases), we will denote with $D^{2}_{dx\,dy} F(t, \eta)$ the measure on $[-\tau,0]^{2}$ such that following duality holds for all $g\in C([-\tau,0]^{2})$
\begin{equation}		\label{eq duality deriv seconda}
\prescript{}{\mathcal{M}([-\tau,0]^{2})}{\langle} D^{2}F(t,\eta), g \rangle_{C([-\tau,0]^{2})} =
D^{2}F(t,\eta)(g)  = \int_{[-\tau,0]^{2}} g(x,y)\,D^{2}_{dx\,dy} F(t,\eta) \ .
\end{equation}
 
%\begin{nota}	\label{nota6.1}
%For simplicity $\left( D\, F \,(t,\eta) \right)^{\delta_0}$ (resp. $\left( DF \, (t,\eta)\right)^\perp$) will be denoted shortly by $D^{\delta_{0}}F \,(t,\eta) $ (resp. $D^\perp \,F(t,\eta)$).
%Taking into account Notation \ref{nota2.3} we recall that .
%\end{nota}
%We introduce a useful notation. 
%In some sense $\perp$ refers to a sort of singularity with respect to the Dirac's zero measure. 
%The Dirac zero measure will play an important role in the representation results.
We conclude the subsection with a notation which concerns deterministic
integrals of real functions.

\begin{nota}		\label{not IBPCONTFUNC}
Let $g, \eta :[a,b]\rightarrow \mathbb{R}$ be c\`adl\`ag. %(continue \`a droite and limite \`a gauche) 
%and $g$ with bounded variation.  
We extend $g$ to the real line setting $g(x)=0$ for $x<a$ 
%and $\eta(x)=0$; for 
and $g(x)=g(b)$ for $x\geq b$.
% and $\eta(x)=\eta(b)$.
 \\
If $g$ has bounded variation, and $a \le c < d \le b$,  we set 
$\int_{]c,d]} 1 d g= g(d) - g(c)$ and $\int_{[c,d]} 1  d g = g(d) - g(c-)$.
Consequently $\int_{[a,b]} 1 d g= g(b)$ since
 $g(a{-})$
%\, \eta(a^{-})$
 vanishes. Conformally to this convention, if $g: [a,b] \rightarrow \R$ has bounded variation 
and $\eta: [a,b] \rightarrow \R$, is continuous, we denote
\[
\int_{]c,d]}g\, d \eta  =g(d)\eta (d)-g(c)\eta(c)-\int_{]c,d]}\eta\,dg  \; 
\quad \mbox{and} 
\; 
\quad
\int_{[c,d]} g\,d\eta =g(d) \eta(d) -g(c-)\, \eta(c-) 
-\int_{[c,d]}\eta\,dg. 
%g(b)\, \eta(b) -\int_{[a,b]}\eta\,dg  \; .
 %It will be the case with $f$ equal to a continuous function, generally denoted by $\eta$.
 \]
For instance 
$
\int_{[a,b]} g\,d\eta = g(b) \eta(b) -  \int_{[a,b]}\eta\,dg. 
$
\end{nota}

\subsection{About anticipative integration with respect to finite quadratic variation process}	\label{6.1}

%We recall that $0 < \tau\leq T$ and we set $B=C([-\tau,0])$. 
This section aims at giving one application of 
calculus via regularization for window processes
 to anticipative calculus in a situation in which
neither It\^o nor Malliavin-Skorohod calculus can be applied.
Our methods also produce, as secondary effect,
some identities involving  path-dependent It\^o or Skorohod integrals
 with forward integrals. 
Let $X$ be a 
real finite quadratic variation process such that $X_{0}=0$ a.s.
 and prolonged as usual by continuity to the real line. 
One motivation is to express, for $\tau \in [0,T]$,
% $$
% \int_{0}^{T-\tau}\left( \int_{-\tau}^{0} g\left( X_{y+\tau+x},X_{y}\right) dx \right)d^{-}X_{y}=\int_{0}^{T-\tau} \left( \int_{y}^{y+\tau} g(X_x,X_y )dx \right)d^{-}X_y
% $$
 \begin{equation} \label{E61}
\int_{0}^{T-\tau} \left( \int_{y}^{y+\tau} g(X_x,X_y )dx \right)d^{-}X_y,
 \end{equation}
for some smooth enough $g:\R^{2}\longrightarrow \R$. 
\begin{rem} \label{RExpl}
\begin{enumerate}
\item We observe that, even when $X$ is a  semimartingale, previous forward 
integral is not an It\^o integral 
since the integrand is anticipating (non adapted). If $X$ is
 a Brownian motion, it can be 
expressed with the help of Skorohod integral.
% If $X$ is a general continuous
%semimartingale, \eqref{E61} can be discussed via calculus via regularization
%for window processes. 
\item We observe that \eqref{E61} equals
\begin{equation} \label{E61bis}
 \int_{0}^{T-\tau}\left( \int_{-\tau}^{0} g\left( X_{y+\tau+x},X_{y}\right) dx 
\right)d^{-}X_{y}.
%=\int_{0}^{T-\tau} \left( \int_{y}^{y+\tau} g(X_x,X_y )dx \right)d^{-}X_y
\end{equation}
\end{enumerate}
\end{rem}
In the perspective of evaluating \eqref{E61bis},
%given by previous remark 
we consider $f:\R^{2}\longrightarrow \R$ of class $C^{2}(\R^{2})$ such that $f(x,y)=\int_{0}^{y}g(x,z)dz$. 
In particular $g=\partial_{2}f$.
For this purpose, we start expanding 
\[
\int_{-\tau}^{0}f\left( X_{x+t}, X_{t-\tau}\right)dx 
\]
through our Banach space $B$-valued It\^o formula.
We obtain the following. 
%showing an application to a genuine anticipative calculus outside Malliavin calculus.
\begin{prop}				\label{prop B}
%Let $X$ be a finite quadratic variation process such that $X_{0}=0$.
 Let $f:\R^{2}\longrightarrow \R$ be a function of class $C^{2}$.
We have 
\begin{equation}	\label{R4}
\begin{split}
\int_{-\tau}^{0}f\left( X_{x+t}, X_{t-\tau}\right)dx
&
=\tau \, f(0,0) %\\
%&
+\int_{0}^{T} \left( \int_{y}^{(y+\tau)\wedge T} \partial_{1} f\left( X_{y},X_{t-\tau} \right)dt \right)d^{-}X_{y}\\
&\hspace{-3cm}
+\int_{0}^{T-\tau} \left( \int_{-\tau}^{0}\partial_{2}f\left( X_{y+x+\tau},X_{y}\right) dx\right)d^{-}X_{y}%\\
%&
+\frac{1}{2}\int_{0}^{T-\tau} \left( \int_{-\tau}^{0}\partial_{2\,2}^{2}f \left(X_{y+z+\tau},X_{y} \right)dz \right) d[X]_{y} \\
&\hspace{-3cm}
+\frac{1}{2}\int_{-\tau}^{0}\left( \int_{-x}^{T} \partial_{1\,1}^{2}f \left(X_{t+x},X_{t-\tau}\right)\, [X]_{dt+x} \right) dx, 
\end{split}
\end{equation}
provided that at least one of the two forward integrals above exists.
\end{prop}
\begin{rem}	\label{remA}
If $X$ is an $(\mathcal{F}_{t})$-semimartingale the forward integral 
\be		\label{eq J1}
\int_{0}^{T} \left( \int_{y}^{(y+\tau)\wedge T} \partial_{1} f\left( X_{y},X_{t-\tau} \right)dt \right)d^{-}X_{y}
\ee
coincides with the It\^o integral 
\[
\int_{0}^{T} \left( \int_{y}^{(y+\tau)\wedge T} \partial_{1} f\left( X_{y},X_{t-\tau} \right)dt \right)dX_{y}   .
\]
\end{rem}

\begin{proof}[Proof of Proposition \ref{prop B}]
We will apply Theorem \ref{thm ITONOM} to $F\left( X_{t}(\cdot) 
 \right)$ where $F:C([-\tau,0])\longrightarrow \R$ is the functional defined by 
$
F(\eta)= \int_{-\tau}^{0} f\left( \eta(x),\eta(-\tau) \right) dx 
$ which is of class $C^{2}(B)$. Below we express the first derivative as
\[
D_{dx}F\left( \eta \right)= \partial_{1}f\left(\eta(x),\eta(-\tau) \right) \1_{[-\tau,0]}(x)dx+\int_{-\tau}^{0}\partial_{2}f\left( \eta(z),\eta(-\tau) \right)dz \; \delta_{-\tau}(dx)  \; 
\]
and the second derivative as
\begin{equation}
\begin{split}
D^{2}_{dx,\, dy} F(\eta)&
= 
\partial_{1\,1}^{2}f\left( \eta(x),\eta(-\tau)\right)\1_{[-\tau,0]}(x) \delta_{y}(dx)\, dy
+
\partial_{2\,1}^{2}f\left( \eta(x), \eta(-\tau) \right)\delta_{-\tau}(dx)\,\1_{[-\tau,0]}(y)dy\\
&
+\partial_{1\,2}^{2}f\left( \eta(x), \eta(-\tau) \right)\1_{[-\tau,0]}(x)dx \, \delta_{-\tau}(dy)+
\int_{-\tau}^{0}\partial_{2\,2}^{2}f\left( \eta(z), \eta(-\tau) \right) dz\; \delta_{-\tau}(dx)\,\delta_{-\tau}(dy) \, . \label{eq 2DER}
\end{split}
\end{equation}
The second order Fr\'echet derivative $D^{2}F(\eta)$ belongs to $\chi$ with $\chi:=Diag\oplus \mathcal{D}_{-\tau}\otimes_{h}L^{2}\oplus L^{2}\otimes_{h}\mathcal{D}_{-\tau}\oplus \mathcal{D}_{-\tau,-\tau}$.
Since $X$ is a finite quadratic variation process, Propositions \ref{pr QV123}, \ref{pr DIAG tau} and \ref{prop somma diretta di qv} imply that 
$X(\cdot)$ admits a $\chi$-quadratic variation. 
We apply now Theorem \ref{thm ITONOM} to $F\left( X_{T}(\cdot)\right)$.  
The forward integral appearing in the It\^o formula $$I_{1}:=\int_{0}^{T} \langle DF(X_{t}(\cdot)) \, , \, d^{-}X_{t}(\cdot) \rangle$$ exists and 
it is given by
$
I_{11}+I_{12}
$
where 
\[
\begin{split}
I_{11}&
=\lim_{\epsilon\rightarrow 0} \int_{0}^{T}\int_{-\tau}^{0} \partial_{1}f\left( X_{t+x},X_{t-\tau} \right)\frac{X_{t+x+\epsilon}-X_{t+x}}{\epsilon}dx\, dt  \quad \mbox{ and}\\
I_{12}&
=\lim_{\epsilon\rightarrow 0}\int_{0}^{T} \left( \int_{-\tau}^{0}\partial_{2}f\left( X_{t+x},X_{t-\tau}\right) dx\right) \frac{X_{t-\tau+\epsilon}-X_{t-\tau}}{\epsilon} dt  ,
\end{split}
\]
provided that previous limits in probability exist. We have 
\[
\begin{split}
I_{11}&
=
\lim_{\epsilon\rightarrow 0} \int_{0}^{T}\int_{(-\tau)\vee (-t)}^{0} \partial_{1}f\left( X_{t+x},X_{t-\tau} \right)\frac{X_{t+x+\epsilon}-X_{t+x}}{\epsilon}dx\, dt \\
&
=
\lim_{\epsilon\rightarrow 0} \int_{0}^{T}\int_{(t-\tau)\vee (0)}^{t} \partial_{1}f\left( X_{y},X_{t-\tau} \right)\frac{X_{y+\epsilon}-X_{y}}{\epsilon}dy\, dt \; .
\end{split}
\]
By Fubini's theorem, previous limit equals \eqref{eq J1}, provided that previous forward limit exists.\\
We go on specifying $I_{12}.$
\[
\begin{split}
I_{12}
&
=\lim_{\epsilon\rightarrow 0}\int_{\tau}^{T} \left( \int_{-\tau}^{0}\partial_{2}f\left( X_{t+x},X_{t-\tau}\right) dx \right) \frac{X_{t-\tau+\epsilon}-X_{t-\tau}}{\epsilon} dt \\
&
=
\lim_{\epsilon\rightarrow 0}\int_{0}^{T-\tau} \left( \int_{-\tau}^{0}\partial_{2}f\left( X_{y+x+\tau},X_{y}\right) dx\right)  \frac{X_{y+\epsilon}-X_{y}}{\epsilon} dy \\
&
=
\int_{0}^{T-\tau} \left( \int_{-\tau}^{0}\partial_{2}f\left( X_{y+x+\tau},X_{y}\right) dx\right)d^{-}X_{y},
\end{split}
\]
provided that previous forward integral exists.\\
We evaluate now the integrals involving the second order derivative of $F$, i.e. 
\begin{equation}		\label{eq EF11}
\frac{1}{2}\int_{0}^{T}\prescript{}{\chi}{\langle} D^{2}F(X_{t}(\cdot))\, ,\, d\widetilde{[X(\cdot)]}_{t}\rangle_{\chi^{\ast}} \; .
\end{equation}
We remind that $D^{2}F(\eta)$ takes values in $\chi:=Diag\oplus \mathcal{D}_{-\tau}\otimes_{h}L^{2}\oplus L^{2}\otimes_{h}\mathcal{D}_{-\tau}\oplus \mathcal{D}_{-\tau,-\tau}$.
%By Propositions \ref{pr QV123}, \ref{pr DIAG tau} and \ref{prop somma diretta di qv} $X(\cdot)$ admits a $\chi$-quadratic variation.
The term \eqref{eq EF11} splits into a sum of four terms. 
Since by Proposition \ref{pr QV123} item 2, $X(\cdot)$ has zero 
$\mathcal{D}_{-\tau}\otimes_{h}L^{2}$ and $L^{2}\otimes_{h}\mathcal{D}_{-\tau}$-quadratic variation, the
% for more details about this technicality see Proposition 7.31 in \cite{DGR}. 
only non vanishing integrals are the two terms $I_{21}$ and 
$I_{22}$ given 
respectively by the $\mathcal{D}_{-\tau,-\tau}$ and the $Diag$-quadratic variation. 
Again by Proposition \ref{pr QV123} item 3, expression \eqref{eq EF11} becomes $I_{21}+I_{22}$ where 
%\[  
%\begin{split}
%I_{21}
%%& =\lim_{\epsilon\rightarrow 0}\frac{1}{2}\int_{0}^{T} \int_{-t}^{0}\partial_{2\,2}^{2}f \left(S_{t+x},S_{t-\tau} \right)dx \frac{(S_{t-\tau+\epsilon}-S_{t-\tau})^{2}}{\epsilon}\, dt\\
%%& =\lim_{\epsilon\rightarrow 0}\frac{1}{2}\int_{0}^{T-\tau} \int_{-\tau}^{0}\partial_{2\,2}^{2}f \left(S_{y+x+\tau},S_{y} \right)dx \frac{(S_{y+\epsilon}-S_{y})^{2}}{\epsilon}\, dy\\
%&
%=\frac{1}{2}\int_{0}^{T-\tau} \int_{-\tau}^{0}\partial_{2\,2}^{2}f \left(X_{y+z+\tau},X_{y} \right)dz \, d[X]_{y}\\
%I_{22}
%%& =\lim_{\epsilon\rightarrow 0}\frac{1}{2}\int_{0}^{T} \int_{(-t)\vee(-\tau)}^{0}\partial_{1\,1}^{2}f \left(S_{t+x},S_{t-\tau} \right)\frac{(S_{t+x+\epsilon}-S_{t+x})^{2}}{\epsilon}dx\, dt\\
%&
%=\frac{1}{2}\int_{0}^{T} \int_{(t-\tau)_{+}}^{t}\partial_{1\,1}^{2}f \left(X_{y},X_{t-\tau} \right)d[X]_{y}\, dt  \quad \textrm{and} \\
%\end{split}
%\]
\[
I_{21}
=\frac{1}{2}\int_{0}^{T-\tau} \int_{-\tau}^{0}\partial_{2\,2}^{2}f \left(X_{y+z+\tau},X_{y} \right)dz \, d[X]_{y}  \; , \quad  \quad
I_{22}
=\frac{1}{2}\int_{0}^{T}\prescript{}{Diag}{\langle} G(t)\, ,\, d\widetilde{[X(\cdot)]}_{t}\rangle_{Diag^{\ast}}   
\]
and $G(t)=g(t,x)\delta_{y}(dx)dy$, with $g(t,x)=\partial_{1\,1}^{2}f \left(X_{t+x},X_{t-\tau}\right) $. Since $\partial_{1\, 1}^{2}f$ is a continuous function, Proposition \ref{prop 5GT6} 
can be applied and we get 
\[
I_{22}=\frac{1}{2}\int_{-\tau}^{0}\left( \int_{-x}^{T} \partial_{1\,1}^{2}f \left(X_{t+x},X_{t-\tau}\right)\, [X]_{dt+x} \right) dx  \; .
\]
In conclusion we obtain \eqref{R4}.
%\[
%\begin{split}
%\int_{-\tau}^{0}f\left( X_{x+t}, X_{t-\tau}\right)dx
%&
%=\tau \, f(0,0) %\\
%%&
%+\int_{0}^{T} \left( \int_{y}^{(y+\tau)\wedge T} \partial_{1} f\left( X_{y},X_{t-\tau} \right)dt \right)d^{-}X_{y}\\
%&
%+\int_{0}^{T-\tau} \left( \int_{-\tau}^{0}\partial_{2}f\left( X_{y+x+\tau},X_{y}\right) dx\right)d^{-}X_{y}%\\
%%&
%+\frac{1}{2}\int_{0}^{T-\tau} \left( \int_{-\tau}^{0}\partial_{2\,2}^{2}f \left(X_{y+z+\tau},X_{y} \right)dz \right) d[X]_{y} \\
%&
%+\frac{1}{2}\int_{-\tau}^{0}\left( \int_{-x}^{T} \partial_{1\,1}^{2}f \left(X_{t+x},X_{t-\tau}\right)\, [X]_{dt+x} \right) dx  \; .
%\end{split}
%\]
\end{proof}

\begin{cor}	
Let $X$ be an $(\mathcal{F}_{t})$-semimartingale and
 $g:\R^{2}\longrightarrow \R$ of class $C^{2,1}(\R\times \R)$.
 Then, setting 
  $f(x,y)=\int_{0}^{y}g(x,z)dz $,
 the forward integral 
$\int_{0}^{T-\tau}\left( \int_{-\tau}^{0} g\left( X_{y+\tau+x},X_{y}\right) dx
 \right)d^{-}X_{y} $ exists and it can be explicitly given using \eqref{R4}
and the relation $\partial_2 f = g$.
\end{cor}
\begin{proof}
%We set $f(x,y)=\int_{0}^{y}g(x,z)dz $. 
The first forward integral in the right-hand side of \eqref{R4}
% Proposition \ref{prop B}
 exists and it is an It\^o integral. We apply successively
 Proposition \ref{prop B}. 
\end{proof}
\begin{cor}	\label{cor D}
Let $X=W$ be a classical Wiener process, $f\in  C^{2}(\R^{2})$. We have the following identity. 
\[
\begin{split}
\int_{-\tau}^{0}f\left( W_{x+t}, W_{t-\tau}\right)dx
&
=\tau \, f(0,0)%\\
%&
+\int_{0}^{T} \left( \int_{y}^{(y+\tau)\wedge T} \partial_{1} f\left( W_{y},W_{t-\tau} \right)dt \right)dW_{y}\\
&\hspace{-2cm}+\int_{0}^{T-\tau} \left( \int_{-\tau}^{0}\partial_{2}f\left( W_{y+x+\tau},W_{y}\right) dx\right)\delta W_{y} %\\
%&
+\int_{0}^{T-\tau} \left( \int_{-\tau}^{0} \partial_{2\, 1}^{2}  f \left( W_{t+\tau+z},W_{t}\right) dz \right) dt\\
&\hspace{-2cm}+\frac{1}{2}\int_{0}^{T-\tau} \left( \int_{-\tau}^{0}\partial_{2\,2}^{2}f \left(W_{y+z+\tau},W_{y} \right)dz \right) dy %\\
%&
+\frac{1}{2}\int_{-\tau}^{0}\left( \int_{-x}^{T} \partial_{1\,1}^{2}f \left(W_{t+x},W_{t-\tau}\right)\, dt  \right) dx.
\end{split}
\]
%where $\delta W(y)$ indicates the Skorohod integral operation.
\end{cor}

\begin{rem}		\label{rem D}
If $Y\in \mathbb{D}^{1,2}\left(L^{2}([0,T]) \right)$, $D^mY$ represents the Malliavin derivative and $\int_{0}^{t}Y_{s}\delta W_{s}$, $t\in [0,T]$, is the Skorohod integral. 
%The reader may consult for instance \cite{nualart} for more details about Malliavin calculus.
We recall that, by \cite{rv1} and \cite{Rus05}
\begin{equation}	\label{eq 34T}
\int_{0}^{t} Y_{s} d^{-}W_{s}=\int_{0}^{t}Y_{s}\delta W_{s} + \left( Tr^{-}D^m Y\right) (t)  \quad \mbox{ where }
\end{equation}
\[
\left( Tr^{-}D^m Y \right)(t)=\lim_{\epsilon\rightarrow 0}\int_{0}^{t}\left( \int_{s}^{s+\epsilon}\frac{D_{r}^mY_{s}}{\epsilon} dr \right) ds  \quad \mbox{in $L^{2}(\Omega)$.}
\]
\end{rem}
\begin{proof}  [Proof of Corollary \ref{cor D}]
%We need to prove that
It follows from Proposition \ref{prop B} provided  we prove that
\[
\int_{0}^{T-\tau} \left( \int_{-\tau}^{0}\partial_{2}f\left( W_{y+x+\tau},W_{y}\right) dx\right)d^{-}W_{y}  \quad \mbox{equals}
\]
\[
\int_{0}^{T-\tau} \left( \int_{-\tau}^{0}\partial_{2}f\left( W_{y+x+\tau},W_{y}\right) dx\right)\delta W_{y}+\int_{0}^{T-\tau} \left( \int_{-\tau}^{0} \partial_{2\, 1}^{2}  f \left( W_{t+\tau+z},W_{t}\right) dz \right) dt
  \; .
\]
%It follows from Proposition \ref{prop B} and previous
This follows by Remark \ref{rem D} with 
\[
Y_{s}=\int_{-\tau}^{0}\partial_{2}f \left( W_{s+\tau+z}, W_{s}\right) dz  \; .
\]
In fact, for $r>s$, $D^m_{r}Y_{s}=\int_{r-s-\tau}^{0} \partial_{2\, 1}^{2} f \left( W_{s+\tau+z},W_{s}\right) dz$ and so 
\begin{equation}		\label{eq 34R}
\left(Tr^{-}D^m Y \right)(t)= \lim_{r\downarrow s} \int_{0}^{t} D^m_{r}Y_{s} \, ds=
\int_{0}^{t}\left(  \int_{-\tau}^{0} \partial_{2\, 1}^{2} f 
\left(W_{s+\tau+z},W_{s}\right) dz\right) ds.
\end{equation}
Combining \eqref{eq 34R} with \eqref{eq 34T} for $t=T-\tau$ the result is now established.
\end{proof}

\begin{rem}		\label{rem E}
Another example of exploitation of Proposition \ref{prop B}
 arises when $X$ is a Gaussian centered process with covariance 
$
R(s,t)=\mathbb{E}\left[ X_{s}X_{t}\right]
$
such that $\frac{\partial^{2} R}{\partial{s} \partial{t}} $ is a 
signed finite measure $\mu$. 
We say in this case that the covariance of
$X$ has a measure structure, see \cite{rkt}. We remind that in this case $X$ is a finite quadratic variation process and 
$
[X]_{t}=\mu(\{(s,s) | s\in [0,t] \})$. With some slight technical assumptions, the following relation holds:  
\begin{equation}		\label{eq ERE}
\int_{0}^{t} Y_{s} d^{-}X_{s} =\int_{0}^{t} Y_{s}\delta X_{s}+ \int_{[0,t]^{2}}
D^m_{r{+}}Y_{s} d\mu(r,s)
%\lim_{\epsilon\rightarrow 0}\frac{1}{\epsilon} \int_{0}^{t}\left(\int_{[0,t]^{2}}D_{t_{1}}Y_{s}\1_{]s,s+\epsilon]}(t_{2})  d\mu(t_{1},t_{2})  \right) ds
 \; .
\end{equation}
This allows to show the existence of both the forward integrals in the statement of Proposition \ref{prop B} using \eqref{eq ERE}.
\end{rem}

\subsection{Infinite dimensional partial differential equation 
and Clark-Ocone type results}
\label{SIDPDE}

As motivated in the introduction, just after the definition
of window processes, one natural application consists in obtaining a
 \emph{Clark-Ocone type formula} for real finite quadratic variation  
processes. Let   $X$ be  a continuous process 
such that $[X,X]_t \equiv \sigma^2 t$ for some $\sigma \ge 0$.
and we assume again $X_0=0$ 
for simplicity. 
Consider $h= \phi(X_T)$ and let $\shu:[0,T] \times \R \rightarrow \R$
be a solution of $\partial_t \shu_t + \frac{\sigma^2}{2} \partial_{xx} \shu  = 0$ 
with final condition $\shu(T,x) = \phi(x)$ for some real Borel 
non-negative function $\phi$.
  By  It\^o formula \eqref{FITO}, we get
that 
\begin{equation} \label{EEE10bis}
 h =  h_0 + \int_0^t \xi_s d^-X_s,
\end{equation} 
where $\xi_s \equiv \partial_x \shu(s,X_s)$ 
and $h_0= \shu(0,X_0), $ see also \cite{crnsm2}
and references therein.
The integral in \eqref{EEE10bis} is indeed an improper forward integral.
If  $h$ is a path dependent random variable, 
we can express it 
 as a functional
of the corresponding window process, i.e. $h = f(\X)$
where $\X = X(\cdot)$, for $f: B \rightarrow \R$
and $B = C([-T,0]$ throughout this section.
 The idea consists in looking  for solutions $u$ of a suitable 
$B$-valued partial differential equation
%, through a $B$-valued It\^o formula, in order to
which allows to formulate    $h$ as \eqref{EEE10bis} where $h_0$ and $\xi$ depend on $u$. The proof should be again an It\^o type formula, this 
time for processes taking values 
If $h$ belongs to
% the classical Malliavin-Sobolev space 
$\mathbb{D}^{1,2}$,
% by uniqueness of the martingale representation theorem, 
then  $H_{0}=\mathbb{E}[h]$ and $\xi_{t}=\mathbb{E}\left[D^{m}_{t} h\vert 
\mathcal{F}_{t} \right]$.
This statement is the classical \textit{Clark-Ocone formula}. \\

In this subsection we set $\tau=T$ and therefore $B = C([-T,0])$.
%As mentioned in the Introduction, one object of interest is the infinite dimensional PDE \eqref{eq SYST} whose sense is now properly expressed.
%
\begin{dfn}		\label{DefSol}
Let $H:C([-T,0])\longrightarrow \R$ be a Borel functional and 
$u:[0,T]\times B \longrightarrow \R$ of class
%\item $(t,\eta)\mapsto \| D^\perp u \; (t,\eta)\|_{BV}$ is bounded on each compact $[0,T]\times K$. 
$C^{1,2}\left([0,T[\times B \right)\cap C^{0}\left([0,T]\times B \right)$.
$u$  is said to 
be a solution of (the infinite dimensional PDE) 
\begin{equation}		\label{eq SYST OK} 
\begin{dcases}
\partial_{t}u(t,\eta)+\int_{[-t,0]} D_x^{\perp}u\,(t,\eta)\; d\eta(x) +
\frac{\sigma^2 }{2} \langle D^{2}u\,(t,\eta)\; ,\; \1_{D_t}\rangle=0  & \mbox{for $t\in [0,T[$} \\
u(T,\eta)  =H(\eta)   &
\end{dcases}
\end{equation}
if the following conditions hold.\\ 
\textbf{i)} $D^\perp u(t,\eta)$ is absolutely continuous with respect to Lebesgue measure and its Radon-Nikodym derivative, still denoted by 
$x\mapsto D_{x}^{\perp}u\,(t,\eta)$, has bounded variation for any $t\in [0,T[$, $\eta\in B$; 
\textbf{ii)} $D^{2}u(t,\eta)$ is a Borel signed
 measure on $[-T,0]^2$ for all $t\in [0,T]$ and $\eta\in B$; 
\textbf{iii)} $u$ solves \eqref{eq SYST OK} where 
$\int_{[-t,0]}D_x^{\perp}u\,(t,\eta)\; d\eta(x)$ in the sense of Notation \ref{not IBPCONTFUNC},
setting $a = -T, c = -t, d =  b = 0$ and $ g: [-T,0] \rightarrow \R$ being the c\`adl\`ag version of $x \mapsto D_x^{\perp}u$.
 $\langle D^{2}u\,(t,\eta)\; ,\; \1_{D_t}\rangle$ indicates the evaluation of the second order derivative on the 
diagonal $D_t = \{(s,s) \vert s \in  [-t,0] \} $.
\end{dfn}
\begin{thm}		\label{prTGH}  %\footnote{}
Let $H:B\longrightarrow \R$ be a Borel functional and 
$u:[0,T]\times B \longrightarrow \R$ be a solution
to \eqref{eq SYST OK}. We set
 $\chi:= \chi^{0}([-T,0]^{2})\oplus Diag([-T,0]^{2}) $, (shortly $\chi^0 \oplus Diag$). 
We suppose the following.\\
\textbf{i)} $(t,\eta)\mapsto \| D^\perp u \; (t,\eta)\|_{BV}:=
|D_0^\perp u \; (t,\eta)|+\int_{[-T,0]} | D_{x}^\perp u \; (t,\eta) | dx =
|D_0^\perp u \; (t,\eta)|+\|D^\perp u \; (t,\eta)   \|_{Var}$ is bounded on
 $[0,T]\times K$ for each compact $K$ of $B$.\\
\textbf{ii)} $D^{2}u\, (t,\eta)\in \chi$ for every $t\in [0,T]$, $\eta \in B$ and that 
map $(t, \eta)\mapsto D^{2}u\, (t,\eta) $ is continuous from $[0,T]\times
 B$ to $\chi$.\\
Let $X$ be  a continuous process with
 $[X]_{t}=\sigma^2 t$, $\sigma\geq 0$, and $X_{0}=0$.\\
Then the random variable $h:=H(X_{T}(\cdot))$ %=u(T,X_{T}(\cdot))
admits the following representation
\begin{equation}	\label{eqH}
h=u(T,X_{T}(\cdot))=
H_{0}+\int_{0}^{T}\xi_{t}d^{-}X_{t}
\end{equation} 
with $H_{0}=u(0,X_{0}(\cdot))$, $\xi_{t}=D^{\delta_{0}}u\, (s,X_{s}(\cdot))$ 
and $\int_{0}^{T}\xi_{t}d^{-}X_{t}$ is an improper forward integral.
\end{thm}

\begin{proof}  % (VERSIONE ARTICOLO OSAKA)
Since $u\in C^{0}\left([0,T]\times B \right)$, $H=u(T,\cdot)$ is automatically continuous.
By Propositions \ref{cor DIAG}, \ref{pr DIAG tau} and \ref{prop somma diretta di qv} $X(\cdot)$ admits a $\chi$-quadratic variation which is the sum of the 
$\chi^0$-quadratic variation and the $Diag$-quadratic variation. 
Applying Theorem \ref{thm ITONOM} to $u\left(t, X_{t}(\cdot)\right)$ for $t < T$ we obtain 
\be \label{eqBN0}
\begin{split}
u(t,X_{t}(\cdot))&=u(0,X_{0}(\cdot))+
\int_{0}^{t}\partial_{t}u(s,X_{s}(\cdot))ds+\int_{0}^{t}\prescript{}{\mathcal{M}([-T,0])}{\langle} Du(s,X_{s}(\cdot)),d^{-}X_{s}(\cdot)\rangle_{C([-T,0])}\\
&
+\frac{1}{2}\int_{0}^{t} \prescript{}{\chi}{\langle} D^{2}u(s,X_{s}(\cdot)),
d\widetilde{[X(\cdot)]}_{s}\rangle_{\chi^{\ast}} .
\end{split}
\ee
By Assumption \textbf{i)} it is possible to show that $\int_{0}^{t}\prescript{}{\mathcal{M}([-T,0])}{\langle} D^\perp u(s,X_{s}(\cdot)),d^{-}X_{s}(\cdot)\rangle_{C([-T,0])}$ 
exists and equals 
$\int_{0}^t \left( \int_{]-s,0]}D^\perp u(s,\eta)d\eta\right)|_{\eta=X_{s}(\cdot)} ds$. We omit the technicalities. 
Consequently, by subtraction,
$\int_{0}^t D^{\delta_0} u(s,X_s(\cdot)) d^-X_s$ exists for $t\in [0,T[$.
The It\^o expansion \eqref{eqBN0} gives
\begin{equation}		\label{eq BN}
u(t,X_{t}(\cdot))=u(0,X_{0}(\cdot))+
\int_{0}^{t}D^{\delta_{0}}u\, (s,X_{s}(\cdot))d^{-}X_{s}+
\int_{0}^{t}\mathcal{L}u\,(s,X_{s}(\cdot))ds
\end{equation} 
where 
 \begin{equation}		\label{def opLst}
\mathcal{L} u \,(t,\eta)=
\partial_{t}u(t,\eta)+
\int_{]-t,0]}D^{\perp}u(t,\eta)\,d\eta+\frac{\sigma^2}{2}\langle D^{2}u\,(t,\eta)\; ,\; \1_{D_t}\rangle, 
\end{equation}
for $t\in [0,T[$, $\eta\in B$. By hypothesis $\mathcal{L} u\,(t,\eta)=0$, so \eqref{eq BN} gives  
\begin{equation}		\label{eq 34}
u(t,X_{t}(\cdot))=u(0,X_{0}(\cdot))+
\int_{0}^{t}D^{\delta_{0}}u\, (s,X_{s}(\cdot))d^{-}X_{s}  .
\end{equation}
Now for every fixed $\omega$, since $u\in C^{0}\left([0,T]\times B \right)$ and $X$ is continuous, we have 
%, the left-hand side converges, i.e. 
$ \lim_{t\rightarrow T} u(t, X_{t}(\cdot))=u(T,X_{T}(\cdot))$, which equals $H(X_{T}(\cdot))$ by \eqref{eq SYST OK}. This forces the right-hand side of \eqref{eq 34} to converge, 
so that the result follows.
\end{proof}
\begin{rem} \label{RSigmaZero}
 Previous theorem also applies in the case $\sigma=0$, i.e. $[X]=0$.  
To this purpose we observe the following.
\begin{enumerate}
\item Let 
\be   \label{eq7.4ter}
h=f\left( \int_{0}^{T}\varphi_{1}(s)d^{-}X_{s}, \ldots, \int_{0}^{T} \varphi_{n}(s)d^{-}X_{s}\right),
\ee 
with $\varphi_{i}\in C^{2}([0,T])$ and $f\in C^{2}(\mathbb{R}^{n})$. 
We observe that the integrals $ \int_{0}^{T}\varphi_{i}(s)d^{-}X_{s}, 
1 \le i \le n$ are defined because each $\varphi_i$ has bounded
variation, see item 3. of Remark \ref{R120}.
In that case the PDE in \eqref{eq SYST OK} simplifies into $\partial_{t}u + \int_{[-t,0]}D^{\perp}u\,(t,\eta)\,d\eta  =0$ and
 it is easy to provide a solution $u$ in the sense of Definition \ref{DefSol}. 
That $u:[0,T]\times C([-T,0])\longrightarrow \R$ is given by 
\be	\label{eq 6.4bis}
u(t, \eta)=f\left(\int_{[-t,0]}
 \varphi_{1}(s+t)d\eta(s), \ldots, \int_{[-t,0]} \varphi_{n}(s+t)d\eta(s) \right),
\ee
%TENERE PER CLARK-OCONE
%For this purpose we observe that, if $\varphi\in C^2([-T,0])$, the derivatives of $(t,\eta)\mapsto\int_{[-t,0]}\varphi (s+t)d\eta(s)$ are\\
%$
%\partial_{t}  \left(  \int_{[-t,0]}\varphi (s+t)d\eta(s)\right)=\int_{]-t,0]}\dot{\varphi}(s+t)d\eta(s)
%$ and 
%$
%D^\perp_s \left(   \int_{[-t,0]}\varphi (s+t)d\eta(s)  \right)= -\mathds{1}_{]-t,0]}(s) \dot{\varphi}(s+t)
%$.
adopting the same conventions as in Notation \ref{not IBPCONTFUNC}.
\item Since $D^{\delta_0}u(t,\eta)=
 \sum_{i=1}^{n}\partial_{i}f \left(\int_{[-t,0]}\varphi_{1}(s+t)d\eta(s),\ldots, \int_{[-t,0]}\varphi_{n}(s+t)d\eta(s)\right) \varphi_{i}(t)$, by Theorem \ref{prTGH}, we obtain representation \eqref{eqH}
with $H_0=f(0, \ldots, 0)$ and $\xi_{t} =D^{\delta_0}u(t,X_t(\cdot))$
%=\sum_{i=1}^{n}\partial_{i}f \left(\int_{0}^{t}\varphi_{1}(s)d^{-}X_{s},\ldots, \int_{0}^{t}\varphi_{n}(s)d^{-}X_{s})\right) \varphi^{i}(t)$.
The assumptions of Theorem \ref{prTGH} can be easily checked, but we omit the details. We remind only that $X(\cdot)$ admits $\chi^{0}$-quadratic variation.
\item In the case $\sigma=0$, representation \eqref{eqH} can be also established via an application of the finite dimensional It\^o formula for finite quadratic variation processes, 
see Proposition 2.4 in \cite{rg2}.
\item The case $\sigma\neq 0$ with the same r.v. $h$ given by \eqref{eq7.4ter} 
but with $f$ only continuous with linear growth (if $X=W$ and $\sigma=1$ even in the weaker condition $f$ 
with polynomial growth) was treated in Section 9.9 of \cite{DGR}.
\end{enumerate}
\end{rem}
\begin{rem}\label{RClark}
\begin{enumerate} 
\item Theorem \ref{prTGH}
%\ref{thm ITONOM} 
is only one significant result 
related to a generalized Clark-Ocone type formula. 
In order to obtain more precise results, one needs
to provide solutions to 
%The procedure to obtain various  Clark-Ocone type formulae consists 
%essentially in  two steps. 
 infinite dimensional PDEs of the type
 \eqref{eq SYST}.
%an important ingredient  is our Banach space valued It\^o formula,
% more precisely Theorem \ref{thm ITONOM}.
% This allows, 
%almost immediately, to obtain 
%an explicit representation of $H_0$ and $\xi$. 
%This is the object of Theorem \ref{prTGH}.
  The natural problem consists in constructing indeed  solutions of \eqref{eq SYST}. 
For a large class of random variables $h$, Chapter 9 of \cite{DGR}
 provides solutions  of \ref{eq SYST OK}  at least when $[X]_t=t$, i.e. $\sigma = 1$.
\item
 Theorem \ref{prTGH}, among others, generalizes
 Theorem 7.1 of \cite{DGRnote} and it expands its proof to the case when
 $[X]_t=\sigma^2 t$, $\sigma\geq 0$.
\end{enumerate}
\end{rem}

\begin{rem} \label{RSigma}
\begin{enumerate}
\item The assumption $[X]_{t}=\sigma^2 t$ is not crucial. 
With some more work it is possible to obtain similar representations even if $[X]_{t}=\int_{0}^{t}a^{2}(s,X_{s})ds$ for a large class of continuous 
$a:[0,T]\times \mathbb{R}\longrightarrow\mathbb{R}$.
\item A simple example of non-semimartingale $X$ verifying the property $[X]_t=\int_0^t a^2(s,X_s)ds$ is the following. 
Let $a:[0,T]\times \R\longrightarrow \R$ be a function of class $C^{1,0}([0,T]\times \R)$ which is Lipschitz in the second variable.
 Let $\beta$ be a non-semimartingale verifying $[\beta]_t=t$. 
A simple example is given by the sum of a classical Wiener process and an independent fractional Brownian motion $B^H$ with $1/2 <H \le3/4$.
 Obviously $[\beta]_t=t$ and $\beta$ is not a semimartingale according to 
\cite{cheridito}. 
Let $\psi:[0,T]\times \R\longrightarrow \R$ such that $\psi(t,x)=\int_{0}^x a(t,\psi(t,y))dy$. Such $\psi$ exists and it is unique since $a$ is 
Lipschitz. We set $X_t=\psi(t,\beta_t).$ \\
By the stability theorem for finite quadratic variation processes, see e.g. \cite{flru1} Remark 3, since $\psi$ is of class $C^1([0,T]\times \R)$ we get 
\[
[X]_t=\int_{0}^t \left( \frac{\partial \psi}{\partial x} (s,\beta_s)
 \right)^2d[\beta]_s
=\int_0^t a^2(s,\psi(s,\beta_s)) ds
=\int_0^t a^2(s,X_s) ds, t \in [0,T].
\]
This shows the desired property.
\item
Under some light technical assumptions on function $a$, using
It\^o forum la \ref{FITO}, it is possible to show the existence of $\gamma:[0,T]\times \R\longrightarrow\R$ continuous such that 
$
d^{-}X_t= a(t,X_t)d^{-}\beta_t+\gamma(t,X_t)dt.
$
For this type of calculations, the reader can consult \cite{rv4}.
\end{enumerate}
\end{rem}

\appendix

\section{Appendix: Proofs of some technical results}			\label{app proof}

\begin{proof}[Sketch of the proof of the Proposition \ref{prop P1}] 
Let $\mathbb{V}$ (resp. $\mathbb{Y}$) be an $H$-valued bounded variation (resp. continuous) process.
Proceeding as for real valued processes, see for instance \cite{Rus05}, Proposition 1.7)b), 
one can show that $(\mathbb{V},\mathbb{Y})$ has
a zero scalar covariation. 
A semilocally summable process is the sum of a locally summable process and a bounded variation process. 
Therefore, without restriction of generality, we can suppose that $\X$ is locally summable with respect to the tensor products. 
By localization we can suppose that $\X$ is summable with respect to the tensor products and bounded. Let $s \in [0,T]$ and consider 
the following identity 
\begin{equation}	\label{ed45}
\X^{\otimes^{2}}_{s+\epsilon}-\X_{s}^{\otimes^{2}}=\X_{s}\otimes (\X_{s+\epsilon}-\X_{s})+(\X_{s+\epsilon}-\X_{s})\otimes \X_{s}+(\X_{s+\epsilon}-\X_{s})\otimes^{2} \, .
\end{equation}
Dividing \eqref{ed45} by $\epsilon$ and integrating from $0$ to $t$ in the Bochner sense we obtain
\begin{equation}
I_{0}(t,\epsilon)=I_{1}(t,\epsilon)+I_{2}(t,\epsilon)+\int_{0}^{t}\frac{(\X_{s+\epsilon}-\X_{s})\otimes^{2}  }{\epsilon}ds
\end{equation}
where 
\[
I_{0}(t,\epsilon) =\int_{0}^{t} \frac{\X^{\otimes^{2}}_{s+\epsilon}-\X_{s}^{\otimes^{2}} }{\epsilon} ds \, , \quad
I_{1}(t,\epsilon) =\int_{0}^{t} \frac{\X_{s}\otimes (\X_{s+\epsilon}-\X_{s})}{\epsilon} ds  \, , \quad
I_{2}(t,\epsilon) =\int_{0}^{t} \frac{(\X_{s+\epsilon}-\X_{s})\otimes \X_{s}}{\epsilon} ds   \, .
\]
Let $t\in [0,T]$. Obviously we get $
\lim_{\epsilon \rightarrow 0}I_{0}(t,\epsilon) =\X^{\otimes^{2}}_{t}-\X_{0}^{\otimes^{2}} $.
%We recall that there exists two positive measures $\nu_{l}$ and $\nu_{r}$ on the $\sigma$-field of predictable sets of 
%$\Omega\times [0,T]$ such that the It\^o type stochastic integral $\int_{0}^{t}\Y\otimes d\X$ (resp. $\int_{0}^{t} d\X\otimes \Y$) is defined for a $(\mathcal{F}_{t})$-predictable process 
%$\Y$ such that 
%\[
%\int_{\Omega \times [0,T]} \| \Y_{s}(\omega)\| d\nu_{l}(\omega,s)  < \infty 
%\quad 
%\left( \textrm{ resp. } \int_{\Omega \times [0,T]} \| \Y_{s}(\omega)\| d\nu_{r}(\omega,s)  < \infty  \right)\; .
%\]
%Moreover $\int_{0}^{t}\Y_{s}\otimes d\X_{s}$ (resp. $\int_{0}^{t}d\X_{s}\otimes \Y_{s}$ ) exists and 
%\begin{equation}\label{eq 345}
%\mathbb{E}\left[ \left\|  \int_{0}^{t}\Y_{s}\otimes d\X_{s} \right\|  \right]  \leq 
%\int_{\Omega\times [0,T]} \| \Y_{s}(\omega) \| d\nu_{r}(\omega,s) < \infty
%\end{equation}

%\begin{equation}
%\left( \textrm{ resp. } \mathbb{E}\left[ \left\| \int_{0}^{t} d\X_{s}\otimes \Y_{s}  \right\|  \right] \leq 
%\int_{\Omega\times [0,T]} \| \Y_{s}(\omega) \| d\nu_{l}(\omega,s) < \infty 
%\right) \; .
%\end{equation}
\noindent
By an elementary Fubini argument we can show that 
\[
I_{1}(t,\epsilon)=\int_{0}^{t} \left(\frac{1}{\epsilon} \int_{u-\epsilon}^{u}\X_{s} ds  \right)\otimes d\X_{u}  \; .
\]
Since $\frac{1}{\epsilon} \int_{u-\epsilon}^{u}\X_{s}ds\longrightarrow \X_{u}$ for every $u\in [0,T]$ and $\omega \in \Omega$ and 
$\X$ being bounded,
% \eqref{eq 345} 
Theorem 1 in Section 12. A of \cite{dincuvisi}  allows to show that 
$I_{1}(t,\epsilon)\longrightarrow \int_{0}^{t} \X_{s}\otimes d\X_{s} $ in probability. Similarly one shows that $I_{2}(t,\epsilon) \longrightarrow \int_{0}^{t} d\X_{s}\otimes \X_{s} $. 
In conclusion $\X$ admits a tensor quadratic variation which equals 
\[
\X_{t}^{\otimes^{2}}-\int_{0}^{t}\X_{s}\otimes d\X_{s}-\int_{0}^{t} d\X_{s}\otimes \X_{s} \, .
\]
\end{proof}
\begin{proof} [Sketch of the proof of Proposition \ref{prop P2}]
Let $H$ be the  Hilbert values space of $\X$. Let $\mathbb{V}$ (resp. $\mathbb{Y}$) be an $H$-valued bounded variation (resp. continuous) process.
%Proceeding as for real valued processes, see for instance \cite{Rus05}, Proposition 1.7)b), 
%one can show that $\mathbb{V}$ and $\mathbb{Y}$ has
%a zero scalar covariation. 
Without restriction of generality we can suppose that $\X$ is an $(\mathcal{F}_{t})$-local martingale. After localization one can suppose 
that $\X$ is an $(\mathcal{F}_{t})$-square integrable martingale. Proceeding similarly as for the proof of Proposition \ref{prop P1}, using Remark 14.b) of 
Chapter 6.23 of \cite{dincuvisi}, it is possible to show that 
\[
\frac{1}{\epsilon}\int_{0}^{t} \| \X_{s+\epsilon}- \X_{s}\|^{2}_{H}ds \xrightarrow[\epsilon\longrightarrow 0]{} \| \X_{t}\|_{H}^{2}-2 \int_{0}^{t}\langle \X_{s}, d\X_{s}\rangle_{H} \; .
\]
The analogous of the bilinear forms considered in Proposition \ref{prop P1} proof will be the $H$ inner product.
\end{proof}

\noindent
Before writing the proof of Proposition \ref{pr IMPR} we need a technical lemma. In the sequel the indices
 $\chi$ and $\chi^{\ast}$ in the duality, will  often be omitted.
\begin{lem}		\label{lem IMPR}  
Let $t\in [0,T]$. There is a subsequence of $(n_{k})$ still denoted by the same symbol and a null subset $N$ of $\Omega$ such that
\begin{equation}
\widetilde{F}^{n_{k}}(\omega,t)(\phi)
\longrightarrow_{k \rightarrow \infty}
\widetilde{F}(\omega,t)(\phi)  \quad \mbox{for every $\phi\in \chi$ and $\omega\notin N$.}
\end{equation}
\end{lem}
\begin{proof}[Proof of Lemma \ref{lem IMPR} ]
Let $\mathcal{S}$ be a dense countable subset of $\chi$. By a diagonalization principle for extracting subsequences, there is a subsequence $(n_k)$,
a null subset $N$ of $\Omega$ 
such that for all $\omega \notin \Omega$, 
\begin{equation}	\label{eq PP1}
\widetilde{F}_{\infty}(\omega,t)(\phi):=\lim_{k\rightarrow +\infty} 
\widetilde{F}^{n_{k}}(\omega,t)(\phi) \quad {\mbox{exists for any $\phi\in \mathcal{S}$, $\omega\notin N$ and $\forall 
\, t\in [0,T]$.}}
\end{equation}
By construction, for every $t\in [0,T]$, $\phi\in \mathcal{S}$
\[
\widetilde{F}(\cdot,t)(\phi)=F(\phi)(\cdot,t)=\widetilde{F}_{\infty}(\cdot,t)(\phi)
\hspace{0.5cm} a.s.
\]
 Let $t \in [0,T]$ be fixed.
Since $\phi\in \mathcal{S}$ countable, a slight modification of the null set $N$, yields that for every $\omega \notin N$,
\[
\widetilde{F}(\omega,t)(\phi)=\widetilde{F}_{\infty}(\omega,t)(\phi)
\hspace{0.5cm} \forall\, \phi\in \mathcal{S}  \; .
\]
At this point \eqref{eq PP1} becomes
\begin{equation}	\label{eq PP2}
\widetilde{F}(\omega,t)(\phi)=\lim_{k\rightarrow +\infty}  \widetilde{F}^{n_{k}}(\omega,t)(\phi),  \quad {\mbox{for every $\omega\notin N$, $\phi\in \mathcal{S}$.}}
\end{equation} 
It remains to show that \eqref{eq PP2} still holds for $\phi\in \chi$. Therefore we fix $\phi\in \chi$, $\omega\notin N$. 
Let $\epsilon>0$ and $\phi_{\epsilon}\in \mathcal{S}$ such that $\|\phi-\phi_{\epsilon}\|_{\chi} \leq \epsilon$. We can write
\[
\begin{split}
\left| 
\widetilde{F}(\omega,t)(\phi)- \widetilde{F}^{n_{k}}(\omega,t)(\phi)
\right|
&
\leq 
\left| 
\widetilde{F}(\omega,t)(\phi-\phi_{\epsilon})\right| + 
\left|
\widetilde{F}(\omega,t)(\phi_{\epsilon})-  \widetilde{F}^{n_{k}}(\omega,t)(\phi_{\epsilon})
\right|
+
\left| 
\widetilde{F}^{n_{k}}(\omega,t)(\phi_{\epsilon}-\phi)
\right|
\leq
\\
&
\leq \left\| \widetilde{F}(\omega,t) \right\|_{\chi^{\ast}} \|\phi-\phi_{\epsilon}\|_{\chi} +
\sup_{k} \left\| \widetilde{F}^{n_{k}}(\omega,t) \right\|_{\chi^{\ast}}   \|\phi-\phi_{\epsilon}\|_{\chi}+ \\
&\hspace{2cm}
+
\left|
\widetilde{F}(\omega,t)(\phi_{\epsilon})-  \widetilde{F}^{n_{k}}(\omega,t)(\phi_{\epsilon})
\right|		\; .
\end{split}
\]
Taking the $\limsup_{k\rightarrow +\infty}$ in previous expression and using 
\eqref{eq PP2} yields 
\[
\limsup_{k\rightarrow +\infty} \left| \widetilde{F}(\omega,t)(\phi)- \widetilde{F}^{n_{k}}(\omega,t)(\phi) \right|
\leq
 \left\| \widetilde{F}(\omega,t) \right\|_{\chi^{\ast}} \epsilon +
\sup_{k} \left\| \widetilde{F}^{n_{k}}(\omega,\cdot) \right\|_{Var[0,T]}  \epsilon \; .
\]
Since $\epsilon>0$ is arbitrary, the result follows.
\end{proof}
\begin{proof} [Proof of Proposition \ref{pr IMPR} ]
Let $t\in [0,T]$ be fixed. We denote 
\begin{displaymath}
I(n)(\omega):=\int_{0}^{t} \langle H(\omega,s ),d\widetilde{F}^{n}(\omega,s) \rangle -\int_{0}^{t} \langle H(\omega, s),d\widetilde{F}(\omega,s)\rangle  \; .
\end{displaymath}
Let $\delta>0$ and a subdivision of $[0,t]$
given by $0=t_{0}<t_{1}<\cdots<t_{m}=t$ whose mesh is smaller than $\delta$. 
Let $(n_{k})$ be a sequence diverging to infinity. %and $(n_{k_{j}})$ a subsequence according to {\bf ii}. 
We need to exhibit a subsequence $(n_{k_{j}})$ such that 
\begin{equation}  \label{L1}
I(n_{k_{j}})(\omega)\longrightarrow 0 \hspace{1cm} \textrm{ a.s.}
\end{equation}
Lemma \ref{lem IMPR} implies the existence of a null set $N$, a subsequence $(n_{k_{j}})$ such that 
\begin{equation}		\label{eq PP3}
\left|   \widetilde{F}^{n_{k_{j}}}(\omega,t_{l})(\phi) - \widetilde{F}(\omega,t_{l})(\phi)\right|\xrightarrow[j \longrightarrow +\infty]{}0
\hspace{1cm} \forall \, \phi\in \chi \hspace{0.5cm} \textrm{and for every } \hspace{0.5cm} l\in \{0,\ldots, m\} \; .
\end{equation}
Let $\omega\notin N$. We have
\begin{displaymath}
\begin{split}
\left|
I(n_{k_{j}})(\omega) \right|
&=
\left|\sum_{i=1}^{m}\left(
\int_{t_{i-1}}^{t_{i}} \langle H(\omega,s),d\widetilde{F}^{n_{k_{j}}}(\omega,s) \rangle -
\langle H(\omega, s),d\widetilde{F}(\omega,s) \rangle \right)\right|
\leq\\
&
\leq
\sum_{i=1}^{m}
\left|\int_{t_{i-1}}^{t_{i}} \langle H(\omega,s)-H(\omega,t_{i-1})+H(\omega,t_{i-1}),d\widetilde{F}^{n_{k_{j}}}(\omega,s)\rangle + \right. \\
&
\hspace{3cm}
\left.
-\int_{t_{i-1}}^{t_{i}} \langle H(\omega,s)-H(\omega,t_{i-1})+H(\omega,t_{i-1}),d\widetilde{F}(\omega,s) \rangle \right|\leq\\
&
\leq  I_{1}(n_{k_{j}})(\omega)+I_{2}(n_{k_{j}})(\omega)+I_{3}(n_{k_{j}})(\omega)  \; , \\
\end{split}
\end{displaymath}
where
\begin{displaymath}
\begin{split}
I_{1}(n_{k_{j}})(\omega)
&=\sum_{i=1}^{m}\left|
\int_{t_{i-1}}^{t_{i}} \langle H(\omega,s)-H(\omega,t_{i-1}),d\widetilde{F}^{n_{k_{j}}}(\omega,s) \rangle \right|\leq
\varpi_{H(\omega,\cdot)}(\delta)\,\sup_{j}\|\widetilde{F}^{n_{k_{j}}} (\omega)\|_{Var[0,T]}
\\
I_{2}(n_{k_{j}})(\omega)
&=\sum_{i=1}^{m}
\left|\int_{t_{i-1}}^{t_{i}} \langle H(\omega,s)-H(\omega,t_{i-1}),d\widetilde{F}(\omega,s) \rangle
\right|\leq
\varpi_{H(\omega,\cdot)}(\delta)\, \|\widetilde{F}(\omega)\|_{Var[0,T]} \\
I_{3}(n_{k_{j}})(\omega)
&
=\sum_{i=1}^{m}\left|
\int_{t_{i-1}}^{t_{i}} \langle H(\omega,t_{i-1}),d(\widetilde{F}^{n_{k_{j}}}(\omega,s)-\widetilde{F}(\omega,s)) \rangle \right|=\\
&
=\sum_{i=1}^{m}
\left|   \langle H(\omega,t_{i-1}),\widetilde{F}^{n_{k_{j}}}(\omega,t_{i})-\widetilde{F}(\omega,t_{i})-\widetilde{F}^{n_{k_{j}}}(\omega,t_{i-1})+
\widetilde{F}(\omega,t_{i-1})\rangle \right|   \leq
\\
&\leq
\sum_{i=1}^{m}|F^{n_{k_{j}}}(H(\omega,t_{i-1}))(\omega,t_{i})-F(H(\omega,t_{i-1}))(\omega,t_{i})|+
\\
&\hspace{3cm}
\sum_{i=1}^{m}|F^{n_{k_{j}}}(H(\omega,t_{i-1}))(\omega,t_{i-1})-F(H(\omega,t_{i-1}))(\omega,t_{i-1})|  \; .
\end{split}
\end{displaymath}
The notation $\varpi_{H(\omega,\cdot)}$ indicates the modulus of continuity for $H$ and it is a random variable; in fact  it
depends on $\omega$ in the sense that
\[
\varpi_{H(\omega,\cdot)}(\delta)=
\sup_{|s-t|\leq\delta  }  
\left\| H(\omega,s)-H(\omega,t) \right\|_{\chi}		\; .
\]
By \eqref{eq PP3} applied to $\phi=H(\omega,t_{i-1})$ we obtain
%Assumption {\bf i} implies that $I_{3}(n_{k_{j}})(\cdot)\rightarrow 0$ in probability because 
%$F^{n} (\phi)(\cdot,t)\rightarrow F(\phi)(\cdot,t)$ for all $\phi=H(\omega,t_{i})$, $i=0,\ldots,m$. After extracting a further 
%subsequence we can suppose that $I_{3}(n_{k_{j}})(\cdot)\rightarrow 0$ a.s. for $j\rightarrow +\infty$. 
%Therefore a.s.
\begin{equation}
\lim\sup_{j\rightarrow\infty} |   I(n_{k_{j}})(\omega)  |  \leq
\left( 
\sup_{j}\|\widetilde{F}^{n_{k_{j}}} (\omega)\|_{Var[0,T]}+ \|\widetilde{F}(\omega)\|_{Var[0,T]} 
\right) \,\varpi_{H(\omega,\cdot)}(\delta)  \, .
\end{equation}
Since $\delta>0$ is arbitrary and $H$ is uniformly continuous on $[0,t]$ so that $\varpi_{H(\omega,\cdot)}(\delta)\rightarrow 0$ a.s. 
for $\delta\rightarrow 0$, then $\lim\sup_{j\rightarrow\infty}|I(n_{k_{j}})(\cdot)|=0$ a.s..\\
%\[
%\lim\sup_{j\rightarrow \infty}|I(n_{k_{j}})|=0\Rightarrow
%\lim_{j\rightarrow \infty}|I(n_{k_{j}})|=0
%\]
This concludes \eqref{L1} and the proof of Proposition \ref{pr IMPR}. 
\end{proof}

\begin{proof}[Proof of Theorem \ref{thm QV}] \
% Supposing \textbf{iv')}, Lemma 3.1 from \cite{rv4} implies that
% $ F^{n}(\phi)\longrightarrow F(\phi)$ ucp
%   for every $ \phi \in \shs$, 
% since 
%  for every $\phi\in \shs$,
%  $F(\phi)$ is an increasing process, so {\bf iv)} is established. 
% We only show the result considering \textbf{iv)}.
\begin{description}
\item [a)] 
We recall that
$\mathscr{C}([0,T])$ is an $F$-space. 
%Let $\epsilon >0$ it exists $\delta_{\epsilon}$ such that have $\alpha d(F^{n}(\phi),0)=\alpha
%\mathbb{E}\left(\sup_{t\in[0,T]}|F^{n}(\phi)(t)|\wedge 1\right)\leq
%\epsilon$ for all $\alpha \leq \delta_{\epsilon}$ because
Let $\phi\in \chi$. Clearly $\left( F^{n}(\phi)(\cdot, t)\right)_{t}$ 
and $\left( \tilde{F}^{n}(\cdot, t)(\phi)\right)_{t}$ 
are indistinguishable processes and so $\left( \tilde{F}^{n}(\phi)(\cdot, t)\right)_{t}$ is a continuous process.
So it follows 
\[
\begin{split}
\left\|F^{n}(\phi) \right\|_{\infty}& =\sup_{t\in[0,T]}|F^{n}(\phi)(t)|=\sup_{t\in [0,T]}|\tilde{F}^{n}(\cdot, t)(\phi)|
\leq
\\
&
\leq
\sup_{t\in [0,T]}\left\| \tilde{F}^{n}(\cdot, t)\right\|_{\chi^{\ast}} \left\|\phi\right\|_{\chi}
\leq
%|F^{n}(\phi)(0)|+
\sup_{n}\|\tilde{F}^{n}\|_{Var([0,T])}\|\phi\|_{\chi}< +\infty
\end{split}
\]
a.s. by the hypothesis. By Remark \ref{rem TRE}.2. and 3. it follows that the set $\left\{ F^{n}(\phi)\right\}$ is a
bounded subset of the $F$-space $\mathscr{C}([0, T])$ for every fixed
 $\phi\in \chi$.\\
%\begin{rem}
%It is easy to show that a sequence of random variables $(Y^{n})$
%such that $\sup_{n}|Y^{n}|\leq Z$ a.s. then the set $(Y^{n})$ is
%bounded in the $F$-space of random variable with the convergence in
%probability equipped with the metric $d(X,Y)=\mathbb{E}(|X-Y|\wedge
%1)$. By Lebesgue dominated convergence theorem, we have
%$\lim_{\gamma\rightarrow 0}\mathbb{E}(\gamma Z \wedge 1)=0$ so the
%result follows.
%\end{rem}
We can apply the Banach-Steinhaus Theorem II.1.18, Page. 55 in \cite{ds} and point \textbf{iv)}, 
which %II.1.18 in \cite{ds} pag. 55. 
imply the existence of
$F:\chi\longrightarrow \mathscr{C}([0,T])$ linear and continuous
such that $F^{n}(\phi)\longrightarrow F(\phi)$ ucp for every
$\phi\in \chi$. % (w.r.t. the topology in $\mathscr{C}([0,T])$) 
So {\bf a)} is established in both situations \textbf{1)} and \textbf{2)}.
\item [b)] It remains to show the rest in situation \textbf{1)}, i.e. when $\chi$ is separable.
\item [b.1)]  
%Let $(n_{k})$ be a sequence, for $\omega\in \Omega$, $t\in [0,T]$ and $\phi\in \chi$ we set 
%$$\tilde{F}(\omega,t)(\phi)=F(\phi)(\omega,t) \hspace{2cm} \textrm{and} \hspace{2cm}
%\tilde{F}^{n_{k}}(\omega,t)(\phi)=F^{n_{k}}(\phi)(\omega,t).$$ 
We first prove the existence of a 
suitable version $\tilde{F}$ of $F$ such that 
$\tilde{F}(\omega,\cdot):[0,T]\longrightarrow \chi^{\ast} $ is weakly star
continuous $\omega$ a.s.\\
Since $\chi$ is separable, we consider a dense countable subset
$\mathcal{D}\subset \chi$. Point {\bf a)} implies that for a fixed
$\phi\in\mathcal{D}$ there is a subsequence $(n_{k})$ such that 
% $\tilde{F}^{n_{k}}(\omega,\cdot)(\phi)=$
$F^{n_{k}}(\phi)(\omega,\cdot)\xrightarrow[]{C([0,T])} F(\phi)(\omega, \cdot) $ 
a.s. Since $\mathcal{D}$ is countable there is a null set $N$ and a further 
subsequence still denoted by $(n_{k})$ such that
\begin{equation}	\label{eq THMC}
\tilde{F}^{n_{k}}(\omega,\cdot)(\phi)\xrightarrow[]{C([0,T])} F(\phi)(\omega, \cdot) \quad\quad
\,\forall\,\phi\in\mathcal{D},\,\forall \omega \notin
N  \; .
\end{equation}
For $\omega\notin N$, we set $\tilde{F}(\omega,t)(\phi)=F(\phi)(\omega,t)$ $\forall\; \phi\in \mathcal{S}$, $t\in [0,T]$. 
By a slight abuse of notation
%We define
%$\tilde{F}_{(\omega,t)}(\phi),\tilde{F}^{n_{k}}_{(\omega,t)}(\phi)$
%for all $t\in[0,T]$, $\phi\in \mathcal{D}$, $\omega\notin N_{0}$ by
%$F(\phi)(t,\omega)$ and $F^{n_{k}}(t,\omega)$ rispectively. 
the
sequence $\tilde{F}^{n_{k}}$ can be seen as applications 
$$\tilde{F}^{n_{k}}(\omega,\cdot):\chi\longrightarrow C([0,T]) $$ 
which are linear continuous maps verifying the following.
\begin{itemize}
\item
$\tilde{F}^{n_{k}}(\omega,\cdot)(\phi)\longrightarrow
\tilde{F}(\omega,\cdot)(\phi)$ in $C([0,T])$ for all $\phi\in\mathcal{D}$, because of \eqref{eq THMC}.
\item For every $\phi \in \chi$, we have
\begin{eqnarray*}
\sup_{k}\sup_{t\leq T}|\tilde{F}^{n_{k}}(\omega,t)(\phi)|
&\leq&
\sup_{k}\sup_{t\leq T} \sup_{\|\phi\|_{\chi} \leq 1} |\tilde{F}^{n_{k}}(\omega,t)(\phi)|\; \|\phi\|_{\chi}\leq
\sup_{k}\sup_{t\leq T}\|\tilde{F}^{n_{k}}(\omega,t)\|\; \|\phi\|_{\chi}  \\
&\leq& 
\sup_{k}\|\tilde{F}^{n_{k}}(\omega,\cdot)\|_{Var([0,T])}\|\phi\|_{\chi}<+\infty.
\end{eqnarray*}
\end{itemize}
Banach-Steinhaus theorem 
%for Banach spaces 
%applied for each $\omega\notin N_{0}$, $t\in [0,T]$ 
implies the existence of a linear random continuous map
$$
\tilde{F}(\omega,\cdot):\chi\longrightarrow C([0,T])
$$ 
extending previous map
$\tilde{F}(\omega,\cdot)$ from $\mathcal{D}$ to $\chi$ with values on $C([0,T])$. 
Moreover 
$$
\tilde{F}^{n_{k}}(\omega,\cdot)(\phi)\xrightarrow[]{C([0,T])}
\tilde{F}(\omega,\cdot)(\phi)\quad\quad 
\,\forall\phi\in\chi,\,\forall \omega \notin
N
$$
and for every $\omega\notin N$ the application
$$
\tilde{F}(\omega,\cdot):[0,T]\longrightarrow \chi^{\ast} \hspace{1cm} t\mapsto \tilde{F}(\omega,t)
$$ 
is weakly star continuous.
%, i.e. for
%$t_{n}\rightarrow t$,
%%$\tilde{F}_{(\omega,t_{n})}\rightarrow^{w\ast}\tilde{F}_{(\omega,t)}\Leftrightarrow
%$\tilde{F}(\omega,t_{n})(\phi)\rightarrow\tilde{F}(\omega,t)(\phi)
%$ for all $\phi\in\chi$. 
$\tilde{F}$ is measurable from
$\Omega \times [0,T]$ to $\chi^{\ast}$ being limit of measurable processes. 

\item  [b.2)] We prove now that the $\chi^{\ast}$-valued process $\tilde{F}$ 
has bounded variation.\\ 
Let $\omega \notin N$ fixed again. Let $(t_{i})^{M}_{i=0}$ be a subdivision of $[0,T]$ and let $\phi\in \chi$. 
Since the functions 
%In the
%last step we show that its total variation is finite and bounded by
%$\|\tilde{F}\|_{Var}$a.s.\\
\[
F^{t_{i},t_{i+1}}:\phi\longrightarrow \left(
\tilde{F}(t_{i+1})-\tilde{F}(t_{i}) \right)(\phi) \quad\quad\quad
F^{n_{k},t_{i},t_{i+1}}:\phi\longrightarrow \left(
\tilde{F}^{n_{k}}(t_{i+1})-\tilde{F}^{n_{k}}(t_{i}) \right)(\phi)
\]
belong to $\chi^{\ast}$, Banach-Steinhaus theorem says
\[
\begin{split}
\sup_{\|\phi\|\leq
1}\left|\left( \tilde{F}(t_{i+1})-\tilde{F}(t_{i}) \right)(\phi)
\right|&=
\|F^{t_{i},t_{i+1}}\|_{\chi^{\ast}}\leq \lim\inf_{k\rightarrow \infty}
\|F^{n_{k},t_{i},t_{i+1}}\|_{\chi^{\ast}}=\\
&=
\lim\inf_{k\rightarrow \infty} \sup_{\|\phi\|\leq 1}\left|\left(
\tilde{F}^{n_{k}}(t_{i+1})-\tilde{F}^{n_{k}}(t_{i}) \right)(\phi)
\right|  \; .
\end{split}
\]
Taking the sum over $i=0,\ldots,(M-1)$ we get
\[
\begin{split}
\sum_{i=0}^{M-1}\sup_{\|\phi\|\leq 1}\left|\left(
\tilde{F}(t_{i+1})-\tilde{F}(t_{i}) \right)(\phi) \right| 
&\leq
\sum_{i=0}^{M-1} \lim\inf_{k\rightarrow \infty}\sup_{\|\phi\|\leq
1}\left|\left( \tilde{F}^{n_{k}}(t_{i+1})-\tilde{F}^{n_{k}}(t_{i})
\right)(\phi) \right|\leq
\\
&
\leq \sup_{k}\sum_{i=0}^{M-1}\sup_{\|\phi\|\leq 1}\left|\left(
\tilde{F}^{n_{k}}(t_{i+1})-\tilde{F}^{n_{k}}(t_{i}) \right)(\phi)
\right|\leq \sup_{k}\|\tilde{F}^{n_{k}}\|_{Var([0,T])}		\; ,
\end{split}
\]
where the second inequality is justified by the relation $\lim\inf
a_{i}^{n}+\lim\inf b_{i}^{n}\leq \sup(a_{i}^{n}+b_{i}^{n})$.\\
Taking the sup over all subdivision $(t_{i})^{M}_{i=0}$ we obtain
\[
\|\tilde{F}\|_{Var([0,T])}\leq\sup_{k}\|\tilde{F}^{n_{k}}\|_{Var([0,T])}<+\infty \; .
\]
This shows finally the fact that $\tilde{F}(\omega,\cdot):[0,T]\longrightarrow
\chi^{\ast}$ has bounded variation.
\end{description}
\end{proof}

\section*{Acknowledgments}

This research was partially supported by the ANR Project MASTERIE 2010
 BLAN 0121 01 and by the project Stochastic Analysis and Applications 
 at Centre Interfacultaire Bernoulli of the EPFL (Lausanne).\\
We are grateful to an anonymous Referee for her/his comments which helped 
considerably the authors to improve the paper.

\bibliographystyle{plain}
\bibliography{biblio}

\end{document}